\documentclass[12pt]{amsart}
\pdfoutput=1
\usepackage{tikz-cd}
\usepackage{amsmath,amscd}
\usepackage{amssymb}
\usepackage{hyperref}
\usepackage{xcolor}
\usepackage{booktabs}
\usepackage{tikz}
\usepackage{graphicx}
\usepackage[outdir=./]{epstopdf}
\graphicspath{ {pictures/} }
\usetikzlibrary{cd}
\usetikzlibrary{knots}
\usetikzlibrary{calc}
\usetikzlibrary{decorations,decorations.pathreplacing,decorations.markings}
\tikzset{string/.style={ultra thick}}
\tikzset{smallstring/.style={thick,scale=0.75,every node/.style={transform shape}}}

\usepackage{graphicx}
\usepackage{fullpage}
\usepackage[T1]{fontenc}
\usepackage{ifthen}
\usepackage[all]{xy}
\usepackage{pdflscape}
\usepackage{array}
\usepackage{multirow}
\usepackage{rotating}

\usepackage{xcolor}
\definecolor{dark-red}{rgb}{0.7,0.25,0.25}
\definecolor{dark-blue}{rgb}{0.15,0.15,0.55}
\definecolor{medium-blue}{rgb}{0,0,0.65}
\definecolor{DarkGreen}{RGB}{0,150,0}

\newcommand{\doi}[1]{\href{https://dx.doi.org/#1}{{\small DOI:#1}}}

\newcommand{\googlebooks}[1]{(preview at \href{https://books.google.com/books?id=#1}{google books})}

\newcommand{\numdam}[1]{}

\theoremstyle{plain}
\newtheorem{prop}{Proposition}[section]

\newtheorem{thm}[prop]{Theorem}
\newtheorem{lem}[prop]{Lemma}
\newtheorem{cor}[prop]{Corollary}
\numberwithin{equation}{section}

\theoremstyle{remark}
\newtheorem{example}[prop]{Example}

\newtheorem{remark}[prop]{Remark}

\theoremstyle{definition}
\newtheorem{defn}[prop]{Definition}         

\newcommand{\sslash}{\mathbin{/\mkern-6mu/}}

\tikzstyle{mid>}=[decoration={markings, mark=at position 0.5 with {\arrow{>}}}, postaction={decorate}]
\tikzstyle{mid<}=[decoration={markings, mark=at position 0.5 with {\arrow{<}}}, postaction={decorate}]
\tikzstyle{upper>}=[decoration={markings, mark=at position 0.8 with {\arrow{>}}}, postaction={decorate}]
\tikzstyle{upper<}=[decoration={markings, mark=at position 0.8 with {\arrow{<}}}, postaction={decorate}]
\tikzstyle{lower>}=[decoration={markings, mark=at position 0.25 with {\arrow{>}}}, postaction={decorate}]
\tikzstyle{lower<}=[decoration={markings, mark=at position 0.25 with {\arrow{<}}}, postaction={decorate}]
\tikzstyle{tag}=[decoration={markings, mark=at position 0.5 with {\arrow{Rays[n=1]}}}, postaction={decorate}]

\DeclareMathOperator{\ev}{ev}

\DeclareMathOperator{\Inv}{Inv}

\newcommand{\id}{\operatorname{id}}

\newcommand{\cC}{\mathcal{C}}

\newcommand{\cF}{\mathcal{F}}
\newcommand{\cZ}{\mathcal{Z}}
\newcommand{\cH}{\mathcal{H}}
\newcommand{\cc}[1]{\mathcal{#1}}

\newcommand{\eq}{\quad = \quad}

\newcommand{\Id}{\operatorname{Id}}

\renewcommand{\Vec}{{\mathsf {Vec}}}

\newcommand{\Aut}{{\operatorname{Aut}}}
\newcommand{\ad}{{\operatorname{Ad}}}
\newcommand{\Z}[1]{\mathbb{Z}_{#1}}

\newcommand{\Ext}{{ \operatorname{Ext}_G(\mathcal{C}) }}
\newcommand{\BrPic}{{\operatorname{BrPic}}}
\newcommand{\BBrPic}{{\underline{\operatorname{BrPic}}}}
\newcommand{\BBBrPic}{{\underline{\underline{\operatorname{BrPic}}}}}
\def\semicolon{;}
\def\applytolist#1{
    \expandafter\def\csname multi#1\endcsname##1{
        \def\multiack{##1}\ifx\multiack\semicolon
            \def\next{\relax}
        \else
            \csname #1\endcsname{##1}
            \def\next{\csname multi#1\endcsname}
        \fi
        \next}
    \csname multi#1\endcsname}

\def\calc#1{\expandafter\def\csname c#1\endcsname{{\mathcal #1}}}
\applytolist{calc}QWERTYUIOPLKJHGFDSAZXCVBNM;
\def\bbc#1{\expandafter\def\csname bb#1\endcsname{{\mathbb #1}}}
\applytolist{bbc}QWERTYUIOPLKJHGFDSAZXCVBNM;
\def\bfc#1{\expandafter\def\csname bf#1\endcsname{{\mathbf #1}}}
\applytolist{bfc}QWERTYUIOPLKJHGFDSAZXCVBNM;

\makeatletter
\newlength{\L@UnitsRaiseDisplaystyle}
\newlength{\L@UnitsRaiseTextstyle}
\newlength{\L@UnitsRaiseScriptstyle}
\DeclareRobustCommand*{\@UnitsNiceFrac}[3][]{%
  \ifthenelse{\boolean{mmode}}{%
    \settoheight{\L@UnitsRaiseDisplaystyle}{%
      \ensuremath{\displaystyle#1{M}}%
    }%
    \settoheight{\L@UnitsRaiseTextstyle}{%
      \ensuremath{\textstyle#1{M}}%
    }%
    \settoheight{\L@UnitsRaiseScriptstyle}{%
      \ensuremath{\scriptstyle#1{M}}%
    }%
    \settoheight{\@tempdima}{%
      \ensuremath{\scriptscriptstyle#1{M}}%
    }%
    \addtolength{\L@UnitsRaiseDisplaystyle}{%
      -\L@UnitsRaiseScriptstyle%
    }%
    \addtolength{\L@UnitsRaiseTextstyle}{%
      -\L@UnitsRaiseScriptstyle%
    }%
    \addtolength{\L@UnitsRaiseScriptstyle}{-\@tempdima}%
    \mathchoice
      {%
        \raisebox{\L@UnitsRaiseDisplaystyle}{%
          \ensuremath{\scriptstyle#1{#2}}%
        }%
      }%
      {%
        \raisebox{\L@UnitsRaiseTextstyle}{%
          \ensuremath{\scriptstyle#1{#2}}%
        }%
      }%
      {%
        \raisebox{\L@UnitsRaiseScriptstyle}{%
          \ensuremath{\scriptscriptstyle#1{#2}}%
        }%
      }%
      {%
        \raisebox{\L@UnitsRaiseScriptstyle}{%
          \ensuremath{\scriptscriptstyle#1{#2}}%
        }%
      }%
    \mkern-2mu{\sslash}\mkern-1mu%
    \bgroup
      \mathchoice
        {\scriptstyle}%
        {\scriptstyle}%
        {\scriptscriptstyle}%
        {\scriptscriptstyle}%
      #1{#3}%
    \egroup
  }%
  {%
    \settoheight{\L@UnitsRaiseTextstyle}{#1{M}}%
    \settoheight{\@tempdima}{%
      \ensuremath{%
        \mbox{\fontsize\sf@size\z@\selectfont#1{M}}%
      }%
    }%
    \addtolength{\L@UnitsRaiseTextstyle}{-\@tempdima}%
    \raisebox{\L@UnitsRaiseTextstyle}{%
      \ensuremath{%
        \mbox{\fontsize\sf@size\z@\selectfont#1{#2}}%
      }%
    }%
    \ensuremath{\mkern-2mu}{\sslash}\ensuremath{\mkern-1mu}%
    \ensuremath{%
      \mbox{\fontsize\sf@size\z@\selectfont#1{#3}}%
    }%
  }%
}

\DeclareRobustCommand*{\nicefrac}{\@UnitsNiceFrac}%
\makeatother

\allowdisplaybreaks

\newcommand{\highlight}[1]{\textcolor{blue}{#1}}
\author{Cain Edie-Michell}
\address{Cain Edie-Michell\\
Department of Mathematics, Vanderbilt University\\
Nashville\
USA}
\email{cain.edie-michell@vanderbilt.edu}

\title{Equivalences of Graded Fusion Categories}
\begin{document}
\begin{abstract}
We further the techniques developed by Etingof, Nikshych, and Ostrik in \cite{MR2677836} to classify the $\cC$-based equivalences between two $G$-graded extensions of $\cC$. The main result of this paper (which follows from this classification) shows that there is an action of the group $\Aut(G)\times \Aut_\otimes(\cC)$ on the set of all $G$-graded extensions of $\cC$, and further, any two extensions in the same orbit of this action are monoidally equivalent. As a warm up for the proof of our classification result we reprove the classification of graded extensions of a fusion category, making extensive use of graphical calculus. Aside from our main result, we provide several other practical applications of our classification of $\cC$-based equivalences.
\end{abstract}
\maketitle

\section{Introduction}

For any fusion category $\cC$, there is an important invariant $\BBBrPic(\cC)$, the Brauer-Picard 3-category of $\cC$. This invariant is defined as the 3-group of invertible bimodules, bimodule equivalences, and bimodule functor natural isomorphisms. The 3-category $\BBBrPic(\cC)$ has many important applications to various areas of mathematics, and thus there exists significant literature dedicated to computing $\BBBrPic(\cC)$ for various examples \cite{MR3808052, 1810.09469,MR3373393, MR3778972,MR3449240}.

The application of the Brauer-Picard 3-category we focus on for this paper is the classification of $G$-graded extensions of a fusion category $\cC$. Introducing new notation, we write $\Ext$ for the set of all $G$-graded extensions of $\cC$, up to \textit{equivalence of extensions}. In \cite{MR2677836} it is shown that $G$-graded extensions of $\cC$, up to equivalence of extensions, are classified by triples $(c,M,A)$, where 
\begin{itemize}
\item $c$ is a group homomorphism $G\to \BrPic(\cC)$, such that a certain obstruction \[o_3(c) \in H^3(G, \Inv(\cZ(\cC)))\] is trivial,
\item $M$ is an element of a certain $H^2(G, \Inv(\cZ(\cC)))$-torsor, such that a certain obstruction \[o_4(c,M) \in H^4(G, \mathbb{C}^\times)\] is trivial, and
\item $A$ is an element of a certain $H^3(G, \mathbb{C}^\times)$-torsor.
\end{itemize}
Hence there is a bijection between $\Ext$ and the set of such triples $(c,M,A)$. More detail can be found in the mentioned paper, or in Section~\ref{sec:gradedex} of this paper where we reprove this classification using graphical calculus. A detail that is often overlooked with this classification result is that categories are only classified up to equivalence of extensions. An equivalence of extensions between two $G$-graded extensions of $\cC$ is a monoidal equivalence between the two categories that restricts to the identity on $\cC$, and preserves the grading group $G$. Equivalence of extensions is a strictly weaker condition than monoidal equivalence. As the above classification is only up to equivalence of extensions, it can be less than ideal using this classification of graded extensions to prove classification results for fusion categories, where typically one wishes to classify categories up to monoidal equivalence. The Authors motivation behind this paper was exactly this problem of refining the ENO classification of $G$-graded extensions of $\cC$, to make the classification suitable for classification results for fusion categories.

This paper is an attempt to fix the issue of the classification of graded extensions over-counting the number of monoidally inequivalent categories that are $G$-graded extensions of $\cC$. Our tool to fix this problem will be to study \textit{$\cC$-based equivalences} between $G$-graded extensions of $\cC$. These are monoidal equivalences that simply map the trivial component to the trivial component. This is a significantly weaker condition than equivalence of extensions. In fact if $\cC$ is its own adjoint subcategory, then every monoidal equivalence between $G$-graded extensions of $\cC$ is $\cC$-based. A large portion of this paper will be devoted to proving Theorem~\ref{thm:mainclass}, which gives a classification of $\cC$-based equivalences between two $G$-graded extensions of $\cC$.

The main motivation behind this paper is to develop practical techniques to determine when two $G$-graded extensions of $\cC$ are monoidally equivalent. While in theory Theorem~\ref{thm:mainclass} gives a way to quotient out $\Ext$ to get the set of $G$-graded extensions of $\cC$, up to $\cC$-based equivalence, this Theorem is near impossible to work with in practice, as it requires intimate knowledge of the higher layers of the 3-group $\BBBrPic(\cC)$. The main issue being that in order to apply Theorem~\ref{thm:mainclass}, we have to show a certain obstruction living in $H^3(G, \mathbb{C}^\times)$ is trivial. As the group $H^3(G, \mathbb{C}^\times)$ is always non-trivial for $G$ non-trivial, we thus have to compute this obstruction which involves 3-morphisms in $\BBBrPic(\cC)$, and thus is near impossible to compute in practice.

In order to get a more practical way to determine when two $G$-graded extensions of $\cC$ are monoidally equivalent, we apply Theorem~\ref{thm:mainclass} to prove the following Theorem, which we regard as the main result of this paper.
\begin{thm}\label{thm:mainprac}
The group $\Aut_{\otimes}(\cC) \times \Aut(G)$ acts on the set $\text{Ext}_{G}(\cC)$ in a specified way, and any two extensions in the same orbit are monoidally equivalent.
\end{thm}
For details on the action of $\Aut_{\otimes}(\cC) \times \Aut(G)$ on $\text{Ext}_{G}(\cC)$ see Lemma~\ref{lem:Gact}. 

The power of this Theorem comes from the fact that the action of $\Aut_{\otimes}(\cC) \times \Aut(G)$ on the $c$ component of a triple $(c,M,A)$ is straightforward, and can be computed purely from knowing the structure of the group $\BrPic(\cC)$. This gives us the following Corollary.
\begin{cor}\label{cor:cor1}
In order to get a representative from each monoidal equivalence class in $\text{Ext}_{G}(\cC)$, one only has to take triples $(c,M,A)$, where $c:G \to \BrPic(\cC)$ is considered up to pre-composition by automorphisms of $G$, and post-composition by inner automorphisms of $\BrPic(\cC)$ induced by the elements in the image of the homomorphism 
\begin{align*} 
\Aut_\otimes(\cC) &\to \BrPic(\cC)   \\
\mathcal{F} &\mapsto _\cF\cC.
\end{align*}
\end{cor}

Determining how the group $\Aut_{\otimes}(\cC) \times \Aut(G)$ acts on the higher layers $M$ and $A$ is more involved, as it depends intricately on how the group $\Aut_{\otimes}(\cC) \times \Aut(G)$ acts on $c$. However, assuming some fairly restrictive conditions on $\cC$ and $G$, we can say something very computable about the action of $\Aut_{\otimes}(\cC) \times \Aut(G)$ on the $M$ component of a triple $(c,M,A)$.
\begin{cor}\label{cor:cor2}
Let $\cH \in \Aut_\otimes(\cC)$, and $\psi \in \Aut(G)$. Fix $c: G \to \BrPic(\cC)$ such that $c$ is a fixed point under the action of $\cH\times \psi$. Assume that the order of the group $\Aut_\otimes(\cC)$ is equal to some prime p, and that $p$ does not divide the order of the group $H^2(G, \Inv(\cZ(\cC)))$. Let $T \in  H^2(G,\Inv(\cZ(\cC)))$, then 
\[    (T \triangleright M)^{\cH\times \psi} = T^{\cH\times \psi} \triangleright M,\]
where $\Aut(G)$ acts on $T$ in the obvious way, and $\Aut_{\otimes}(\cC)$ acts on $T$ via the map $\Aut_\otimes(\cC) \to \BrPic(\cC)$ and then by the standard action of $\BrPic(\cC)$ on $\Inv(\cZ(\cC))$.
\end{cor}
In practice this Corollary allows to give a far better upper bound on the possible number of $G$-graded extensions, up to monoidal equivalence, realising a fixed homomorphism $c$, assuming all the above conditions are satisfied. Naively we have a bound given by 
\[ |  H^2(G,\Inv(\cZ(\cC))) | \cdot | H^3(G,\mathbb{C}^\times)|.    \]
But with an application of Corollary~\ref{cor:cor2}, this bound improves to 
\[ |  H^2(G,\Inv(\cZ(\cC))) /  (\Aut_{\otimes}(\cC) \times \Aut(G))  | \cdot | H^3(G,\mathbb{C}^\times)|.    \]
Computing the size of these sets is a straightforward exercise in group cohomology, provided we know how the group $\BrPic(\cC)$ acts on $\Inv(\cZ(\cC))$. While there are certainly a bunch more results along the lines of Corollary~\ref{cor:cor2} that could be proven, the results we have so far are sufficient for all applications the Author currently has in mind.

Our paper is structured as follows:

In Section~\ref{sec:pre} we recall the basics of $G$-graded categories, and define $\BBBrPic(\cC)$, the Brauer-Picard 3-category of a fusion category. In particular we describe a 3-dimensional graphical calculus for $\BBBrPic(\cC)$, and explain what moves we are allowed to perform in this graphical calculus. This graphical calculus will be one of the key tools for this paper.

In Section~\ref{sec:gradedex} we re-derive the ENO classification of $G$-graded extensions of a fusion category $\cC$. We include this re-derivation for two reasons. The first is that our the proof of the classification of $\cC$-based equivalences borrows many ideas from the ENO proof. Thus the re-derivation of the ENO proof gives a great warm-up to our proof of $\cC$-based equivalences. The second is that our re-derivation of the ENO classification is done via the graphical calculus for $\BBBrPic(\cC)$. While the high level arguments of the proofs are the same, we believe that interpreting these arguments in a graphical manner clarifies many of the subtleties of the ENO classification result, and hence adds to the literature.

In Section~\ref{sec:twist} we define the abstract nonsense of twisted bimodule functors. This is a straightforward generalisation of bimodule functors, where now we require the functor to preserve the bimodule action up to some action of a fixed monoidal auto-equivalence. The motivation behind such a definition is that the restriction of a $\cC$-based equivalence to a graded piece gives exactly a twisted bimodule equivalence. Unfortunately twisted bimodule functors are in practice difficult to work with, as they have no nice graphical calculus. To fix this issue we prove Theorem~\ref{lem:untwisting}, which shows that we can find all the information we need for twisted bimodule equivalences (and their natural isomorphisms), within the Brauer-Picard 3-category. The proof of this theorem is somewhat long and technical. We hide the proof in Appendix~\ref{app:proof}, which the reader can view at their own risk.

In Section~\ref{sec:data} we de-construct a $\cC$-based equivalence between two $G$-graded extensions of a fixed fusion category $\cC$. We show that from a $\cC$-based equivalence one can extract a quadruple of data $(\cF_e, \phi, \widehat{F}, \widehat{T})$. Here $\cF_e$ is a monoidal auto-equivalence of $\cC$, and $\phi$ is an automorphism of the grading group $G$. The third piece of data $\widehat{F}$ is a certain collection of bimodule equivalences, which we call a \textit{system of equivalences for $(\cF_e, \phi)$} (defined in Definition~\ref{def:syseq}). The final piece of data $\widehat{T}$ is a certain collection of bimodule natural isomorphisms, which we call a \textit{system of tensorators for $(\cF_e, \phi,\widehat{F})$} (defined in Definition~\ref{def:systen}). Conversely, given a quadruple of the above data, we show how to reconstruct the initial $\cC$-based equivalence. Hence $\cC$-based equivalences are classified by such quadruples. We remark that our classification of $\cC$-based equivalences is up to \textit{natural isomorphism of $\cC$-based equivalences} (Definition~\ref{def:basediso}).

In Section~\ref{sec:quasi} we fix $\cF_e$ and $\phi$, and classify systems of equivalences for $(\cF_e, \phi)$. We define two obstructions. The first is an equation of two group homomorphisms $G\to \BrPic(\cC)$. The second is 
\[o_2(\cF, \phi) \in H^2(G, \Inv(\cZ(\cC))). \]
We show that there exists a system of equivalences for $(\cF_e, \phi)$ if and only if the equation is satisfied and $o_2(\cF, \phi)$ is trivial. Further, if the equation is satisfied and $o_2(\cF, \phi)$ is trivial, then we show that systems of equivalences for $(\cF_e, \phi)$ form a torsor over $Z^1(G, \Inv(\cZ(\cC)))$.

In Section~\ref{sec:ten} we fix $\widehat{F}$ a system of equivalences for $(\cF_e, \phi)$, and classify systems of tensorators for $(\cF_e, \phi,\widehat{F})$. We define an obstruction 
\[o_3(\cF, \phi, \widehat{F}) \in H^3(G, \mathbb{C}^\times ) \]
and show that there exists a system of tensorators for $(\cF_e, \phi, \widehat{F})$ if and only if $o_3(\cF, \phi, \widehat{F})$ is trivial. Further, if $o_3(\cF, \phi, \widehat{F})$ is trivial, we prove that systems of equivalences for $(\cF_e, \phi,\widehat{F})$ form a torsor over $H^2(G, \mathbb{C}^\times)$.

In Section~\ref{sec:together} we tie together the results of Sections~\ref{sec:data} through \ref{sec:ten} to state and prove Theorem~\ref{thm:mainclass}, which classifies $\cC$-based equivalences between two $G$-graded extensions of a fusion category $\cC$.

We finish our paper with Section~\ref{sec:examples} by providing several practical applications of Theorem~\ref{thm:mainclass}. These are
\begin{itemize}
\item A classification of cyclic pointed fusion categories, up to monoidal equivalence,
\item A theorem to help us classify gauge auto-equivalences of a $G$-graded fusion category,
\item The proofs of Theorem~\ref{thm:mainprac}, along with Corollaries~\ref{cor:cor1} and \ref{cor:cor2},
\item A result constructing a monoidal auto-equivalence of any modular category with distinguished boson or fermion.
\end{itemize}

This paper is fairly lengthy, and several parts of this paper are independent from each other. Thus this paper can be read in several different ways. If one just wishes to see the ENO classification of $G$-graded extensions from a graphical perspective then one can read Sections~\ref{sec:pre} and \ref{sec:gradedex}. If one just wishes to see the practical applications of Theorem~\ref{thm:mainclass}, then one can read Section~\ref{sec:examples} without reading the proofs. If the reader wishes to read the entire paper, then we recommend that they read Section~\ref{sec:examples} without the proofs, and then read this paper from Section~\ref{sec:pre} till the end.

This paper overlaps with the results of Ian Marshall's thesis \cite{ianPhD}. In this thesis a classification of the kernel and image of the restriction map
\[ \Aut_\otimes(\bigoplus \cC_g) \to \Aut_\otimes(\cC_e) \]
is given. In the language of Theorem~\ref{thm:mainclass}, computing the kernel corresponds to classifying all quadruples such that $\cF = \Id_{\cC_e}$, and computing the image corresponds to characterizing when, for a fixed $\cF \in \Aut_\otimes(\cC)$, there exists a quadruple $(\cF,\phi,\widehat{F},\widehat{T})$. Our result is more general as it classifies all quadruples, and it applies to $\cC$-based equivalences between different $G$-graded extension of $\cC$, not just auto-equivalences.

\subsection*{Acknowledgments}
We would like to thank Corey Jones and Scott Morrison for many helpful conversations. This research is supported by an Australian Government Research Training Program (RTP) Scholarship. The author was partially supported by the Discovery Project 'Subfactors and symmetries' DP140100732 and 'Low dimensional categories' DP160103479. Parts of this paper were written during a visit to the Mathematical Sciences Research Institute, who we thank for their hospitality.

\section{Preliminaries}\label{sec:pre}
We refer the reader to \cite{MR3242743} for the basics of fusion categories, and bimodule categories.
\subsection*{$G$-graded fusion categories and functors}
Let $\cD$ be a fusion category, and $G$ a finite group. We say $\cD$ is $G$-graded if we can write
\[  \cD=\bigoplus \cD_g  \]
with $\cD_g$ abelian subcategories of $\cD$, such that the tensor product restricted to $\cD_g \times \cD_h$ has image in $\cD_{gh}$.

We say $\cD$ is a $G$-graded extension of $\cC$ if $\cD$ is $G$-graded, and $\cD_e = \cC$. Every fusion category $\cD$ is a $G$-graded extension of its \textit{adjoint subcategory}.
\begin{defn}
Let $\cD$ be a fusion category. The adjoint subcategory of $\cD$ is the fusion subcategory generated by the objects
\[   \{   X\otimes X^* : X \in \cD \}. \]
\end{defn}

The standard notion of equivalence of extensions between two $G$-graded extensions is given as follows.

\begin{defn}
Let $\cD_1$ and $\cD_2$ be two $G$-graded extensions of $\cC$. An equivalence of $G$-graded extensions $\cD_1 \to \cD_2$ is a monoidal functor $\mathcal{F}: \cD_1 \to \cD_2$ such that $\mathcal{F}|_{\cC} =  \Id_{\cC}$, and $\mathcal{F}$ induces the identity automorphism on the grading group $G$.
\end{defn}

An equivalence of extensions is a much weaker condition than plain monoidal equivalence. We illustrate this fact with the following example.

\begin{example}\label{ex:Z9}
Consider the fusion category $\Vec(\Z{9})$. There are two ways we can regard $\Vec(\Z{9})$ as a $\Z{3}$-graded extension of $\Vec(\Z{3})$. In both cases we have that the 0-graded piece contains the objects $\{0,3,6\}$, the 1-graded piece contains the objects $\{1,4,7\}$, and the 2-graded piece contains the objects $\{2,5,8\}$. However on one hand we can identify the $\Vec(\Z{3})$ sub-category of $\Vec(\Z{9})$ via
\[
0 \to 0, \quad 1\to 3, \quad \text {and } \quad 2 \to 6
\]
while on the other hand we can identity the $\Vec(\Z{3})$ subcategory of $\Vec(\Z{9})$ via
\[
0 \to 0, \quad 1\to 6, \quad \text {and } \quad 2 \to 3.
\]
These two identifications give $\Vec(\Z{9})$ the structure of a $\Z{3}$-graded extension of $\Vec(\Z{3})$ in two ways.

An equivalence of $\Z{3}$-graded extensions between these two extensions would have to map $3 \mapsto 6$, as this equivalence must restrict to the identity on the identified $\Vec(\Z{3})$ subcategories. But then the object $1$ would have to map to one of the objects $2$, $5$, or $8$, which does not preserve the grading group. Hence these two extensions are non-equivalent as extensions of $\Vec(\Z{3})$. However clearly $\Vec(\Z{9})$ is monoidally equivalent to itself.
\end{example}

\begin{defn}
We write $\Ext$ for the set of all $G$-graded extensions of $\cC$, up to equivalences of extensions.
\end{defn}

For this paper we introduce a stronger notion of equivalence between $G$-graded extensions.
\begin{defn}\label{def:basedeq}
Let $\cD_1$ and $\cD_2$ be two $G$-graded extensions of $\cC$. A $\cC$-based equivalence $\cD_1 \to \cD_2$ is a monoidal equivalence $\mathcal{F}: \cD_1 \to \cD_2$ such that $\cF(\cC) \simeq \cC$.
\end{defn}
It is straightforward to see that an equivalence of $G$-graded extensions of $\cC$ is also a $\cC$-based equivalence. However the converse is not true. For example, there exists a $\Vec(\Z{3})$-based equivalence between the two different $\Z{3}$-graded extensions constructed in Example~\ref{ex:Z9}. Thus $\cC$-based equivalence is a strictly weaker condition than equivalence of extensions. In fact if a fusion category $\cC$ is equivalent to its own adjoint subcategory, then every monoidal equivalence between $G$-graded extensions of $\cC$ is a $\cC$-based equivalence.

We also have the several different notions of natural isomorphisms between $\cC$-based equivalences. The following notion we will use for this paper.

 \begin{defn}\label{def:basediso}
Let $\cD_1$ and $\cD_2$ be two $G$-graded extensions of $\cC$, and $\cF_1, \cF_2 : \cD_1 \to \cD_2$ be $\cC$-based equivalences. We say $\cF_1, \cF_2$ are naturally isomorphic as $\cC$-based equivalences if 
 \[ \cF_1|_\cC = \cF_2|_\cC,   \]
 and there exists a monoidal natural isomorphism $\mu : \cF_1 \to  \cF_2$ such that 
 \[  \mu|_{\cC} = \id_{\cF_1|_\cC }   \]
\end{defn}

\subsection*{Quasi-monoidal categories, functors, and natural transformations}
Crucial to our main classification result will be the intermediate classification of graded quasi-monoidal equivalences. We briefly define the basic details of quasi-monoidal fusion categories, functors, and natural transformations. 

A \textit{quasi-monoidal fusion category} is defined exactly the same as a fusion category, except there is no data of an associator, we only require the existence (but not choice) of a natural isomorphism $(X\otimes Y) \otimes Z \to X\otimes (Y\otimes Z)$.

A \textit{quasi-monoidal functor} $\cF$ between two quasi-monoidal categories is a functor that admits a natural isomorphism $\cF(X) \otimes \cF(Y)\to \cF(X\otimes Y)$. Note that the information of quasi-monoidal functor is just the functor itself, and not the choice of this natural isomorphism.

A \textit{quasi-monoidal natural transformation} between quasi-monoidal functors is simply a natural transformation with no extra conditions.

Clearly every monoidal category, functor and natural transformation is also quasi, via forgetting. 

\subsection*{The 3-category of invertible bimodules over a fusion category}

This subsection recalls the information of $\BBBrPic(\cC)$, the 3-category of invertible bimodules over a fusion category $\cC$. More details can be found in the papers \cite{MR2678824,MR2677836}. 

We have in $\BBBrPic(\cC)$ that morphisms are given by
\begin{itemize}
\item 0-morphism : The category $\cC$
\item 1-morphisms $\cC\to \cC$ : Invertible $\cC$-bimodules $\mathcal{M}$
\item 2-morphisms $\mathcal{M} \to \mathcal{N}$ : Bimodule equivalences $\cF$
\item 3-morphisms $\cF \to \mathcal{G}$ : Natural isomorphisms of bimodule functors $\mu$
\end{itemize}

Composition at the 2-level is the composition $\circ$ of bimodule functors, and composition at the 3-level is the vertical composition $\cdot$ of natural transformations. Composition at the 1-level is more complicated, and is given by $\boxtimes$, the so called relative tensor product over $\cC$. 

\begin{remark}
It is very important to note that we are using very non-standard notation for the Deligne product of bimodules and for the relative tensor product over $\cC$. Typical notation is to write $\boxtimes$ for the Deligne product, and $\boxtimes_\cC$ for the relative tensor product over $\cC$. Due to the sheer number of relative tensor products in this paper, we were forced to drop the standard notation, as it made many equations unreadable. Instead we write $\times$ for the Deligne product, and $\boxtimes$ for the relavtive tensor product over $\cC$.
\end{remark}

To define this relative tensor product we first define the \textit{balancing category}, as was done in \cite{MR2678824}. Let $\cc{M}_1,\cc{M}_2$, and $\cc{N}$ be $\cC$-bimodules, then we define $\operatorname{Fun}^\text{bal}(\cc{M}_1\times \cc{M}_2 \to \cc{N})$.

The objects of $\operatorname{Fun}^\text{bal}(\cc{M}_1\times \cc{M}_2 \to \cc{N})$ are bimodule functors

\[ \cc{F}: \cc{M}_1\times \cc{M}_2 \to \cc{N},\]

 along with a collection of balancing isomorphisms 

\[b^\cc{F}_{m_1,c,m_2} : \cc{F}(m_1 \triangleleft c \times m_2) \to \cc{F}(m_1, c \triangleright m_2),\] 
satisfying the consistency condition:

 \begin{tikzpicture}[baseline= (a).base]
\node[scale=.8] (a) at (0,0){
 \begin{tikzcd}[column sep = .1cm, row sep = 3cm]
 & \cF(m_1 \triangleleft c_1 \triangleleft c_2 \times m_2)  \arrow[dl, "\cF(r^{\cM_1}_{m_1,c_1,c_2} \times \id_{m_2})"]  \arrow[rr, "b^\cF_{m_1\triangleleft c_1,c_2,m_2}"] & &\cF(m_1 \triangleleft c_1 \times c_2 \triangleright m_2)\arrow[dr, "b^\cF_{m_1,c_1,c_2\triangleright m_2}"]  \\
\cF(m_1 \triangleleft( c_1 \otimes c_2) \times m_2) \arrow[rr, "b^\cF_{m_1,c_1\otimes c_2,m_2}"]& & \cF(m_1  \times ( c_1 \otimes c_2) \triangleright m_2)\arrow[rr, "\cF(\id_{m_1} \times {l^{M_2}_{c_1,c_2,m_2}}^{-1})"] & & \cF(m_1 \times c_1\triangleright c_2 \triangleright m_2)
\end{tikzcd}   };
\end{tikzpicture}
We call such a pair $(\cc{F},b^\cc{F})$ a $\cc{C}$-balanced bimodule functor.

The morphisms $(\cc{F},b^\cc{F}) \to (\cc{G},b^\cc{G})$ of $\operatorname{Fun}^\text{bal}(\cc{M}_1\times \cc{M}_2 \to \cc{N})$ are natural transformations of bimodule functors $\mu: \cc{F} \to \cc{G}$ that commute with the balancing morphisms in the following sense:

\begin{center}\begin{tikzcd}[row sep = 1cm, column sep = 1.5cm]
\cF(m_1 \triangleleft c \times m_2) \arrow[ d, "b^\cF_{m_1,c,m_2}"]  \arrow[r , "\mu_{ m \triangleleft c \times m_2}" ]& \cc{G}(m_1 \triangleleft c \times m_2)\arrow[ d, "b^\cc{G}_{m_1,c,m_2}"] \\
\cF(m_1  \times c \triangleright m_2) \arrow[r , "\mu_{ m  \times c\triangleright m_2}" ]& \cc{G}(m_1 \times c \triangleright m_2)
\end{tikzcd}\end{center}
We call such a morphism $\mu : (\cc{F},b^\cc{F}) \to (\cc{G},b^\cc{G})$ a $\cC$-balanced natural transformation.

With the category $\operatorname{Fun}^\text{bal}(\cM_1\times \cM_2 \to \cc{N})$ in hand, we can now define the composition of bimodules in the 3-category $\BBBrPic(\cC)$. 

\begin{defn}
Let $\cM_1$ and $\cM_2$ be bimodule categories. The relative tensor product of $\cM_1$ and $\cM_2$, which we denote by $\cM_1 \boxtimes \cM_2$, is the unique bimodule category inducing, for every bimodule category $\cc{N}$, an equivalence of categories:
\[ \operatorname{Fun}^\text{bal}(\cc{M}_1\times \cM_2 \to \cc{N}) \simeq \operatorname{Fun}(\cM_1 \boxtimes \cM_2 \to \cc{N}). \]
\end{defn}

From the above definition it isn't clear that such a category $\cM_1 \boxtimes \cM_2$ exists. In \cite{MR2677836} it is shown that the category $\cM_1 \boxtimes \cM_2$ can be explicitly realised as $\cZ_\cC(\cM_1 \times \cM_2)$, the relative Drinfeld centre of $\cM_1 \times \cM_2$.

Towards generalising the above definition of the relative tensor product of bimodules over $\cC$, we can define $\operatorname{Fun}^\text{bal}(\cM_1 \times \hdots \times \cM_n\to \cc{N})$, the category of bimodule functors, balanced in each position, in a similar fashion. Here we now have objects a bimodule functor, along with $n-1$ balancing isomorphisms that are consistent between each other. We require natural transformations in this category to commute with each balancing isomorphism.

The same way as we defined the relative tensor product $\cM_1 \boxtimes \cM_2$, we can define the $n$-fold relative tensor product as the unique bimodule category $\cM_1 \boxtimes \hdots \boxtimes \cM_n$ inducing an equivalence:
\begin{equation}\label{eq:funball}
\operatorname{Fun}^\text{bal}(\cM_1 \times \hdots \times \cM_n\to \cc{N}) \simeq \operatorname{Fun}(\cM_1 \boxtimes \hdots \boxtimes \cM_n\to \cc{N}).
\end{equation}

An important class of 1-morphisms in the category $\BBBrPic(\cC)$ are given by the following.

\begin{defn}\label{def:righttwist}
 Let $(\cF,\tau)$ be a monoidal auto-equivalence of $\cC$. We define the bimodule $\cC_\cF$, which is just $\cC$ as an abelian category, with actions are given by
 \[    X \triangleright Y \triangleleft Z := X\otimes Y \otimes \cF(Z).\]
The left and central bimodule structure maps are identities, and the right bimodule structure maps are given by $\tau$, the tensor structure map for $\cF$.
 
 Similarly we can define the bimodule $_\cF\cC$, which is again just $\cC$ as an abelian category, but with actions now given by
 \[    X \triangleright Y \triangleleft Z := \cF(X) \otimes Y \otimes Z,\]
Now the central and right bimodule structure maps are identities, and the left bimodule structure map is given by $\tau$.
 \end{defn}
 
It is shown in \cite[Proposition 3.1.]{MR3210925} that $\cC \simeq \cC_\cF$ and $\cC \simeq _\cF\cC$ if and only if $\cF$ is an inner auto-equivalence of $\cC$. Thus the two maps 
\begin{align*} 
\operatorname{Out}_\otimes(\cC) &\to \BrPic(\cC) \\
\cF &\mapsto   \cC_\cF, \\
\cF &\mapsto   _\cF\cC
 \end{align*}
 are injections. The latter is a homomorphism of groups, while the former is an anti-homomorphism.

\subsection*{Graphical calculus for $\BBBrPic(\cC)$}

As $\BBBrPic(\cC)$ is a 3-category, we have a 3-dimensional graphical calculus. While the results of \cite{1312.7188} show that $\BBBrPic(\cC)$ is in fact a rigid category allowing us a fully fledged 3-dimensional rigid graphical calculus, this rigid graphical calculus turns out to be not so useful for this paper. 

Instead we use a much weaker plain 3-category graphical calculus. We pick a coordinate system:
\[\raisebox{-.5\height}{ \includegraphics[scale = .45]{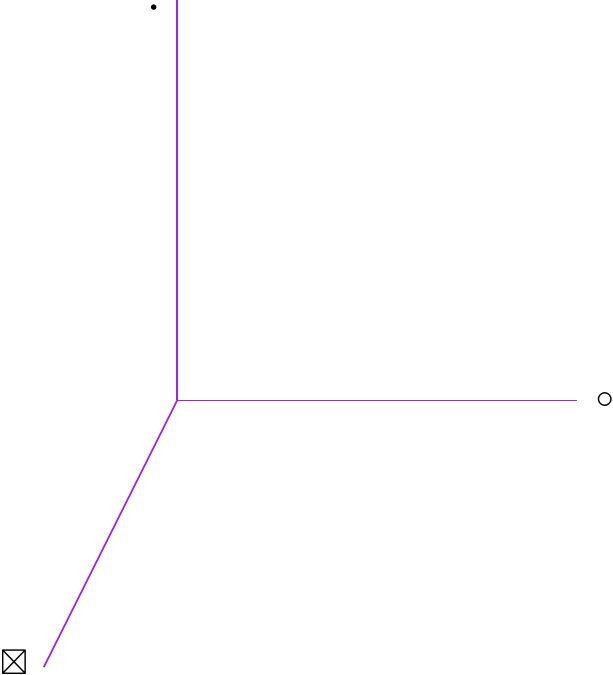}}\]
Recall that in a 3-category we have the interchange law, which allows us to commute suitably disjoint 3-morphisms. In our graphical calculus this interchange allows us isotopy that preserves orientation of 3-morphisms with respect to these axis. For example we have the relation:
\[\raisebox{-.5\height}{ \includegraphics[scale = .45]{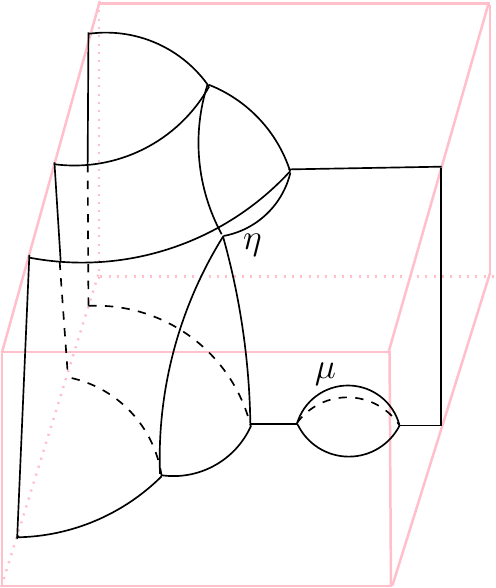}} \eq  \raisebox{-.5\height}{ \includegraphics[scale = .45]{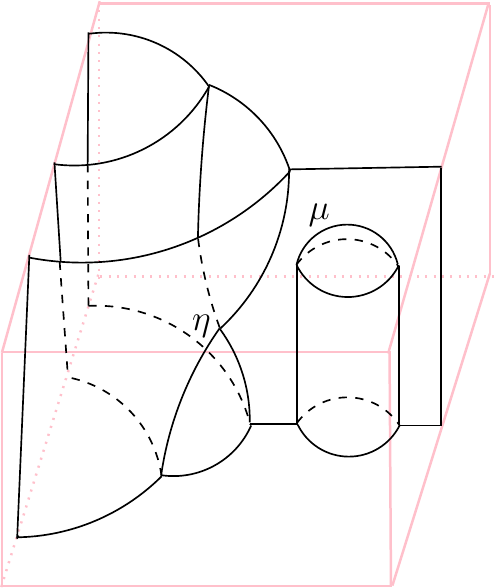}} \]
as the orientations of the 3-morphisms $\eta$ and $\mu$ are not changed. We are not allowed to perform the isotopy:
\[\raisebox{-.5\height}{ \includegraphics[scale = .45]{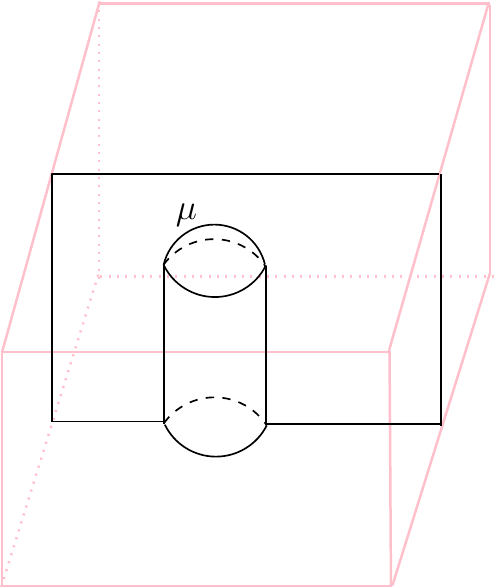}} \eq  \raisebox{-.5\height}{ \includegraphics[scale = .45]{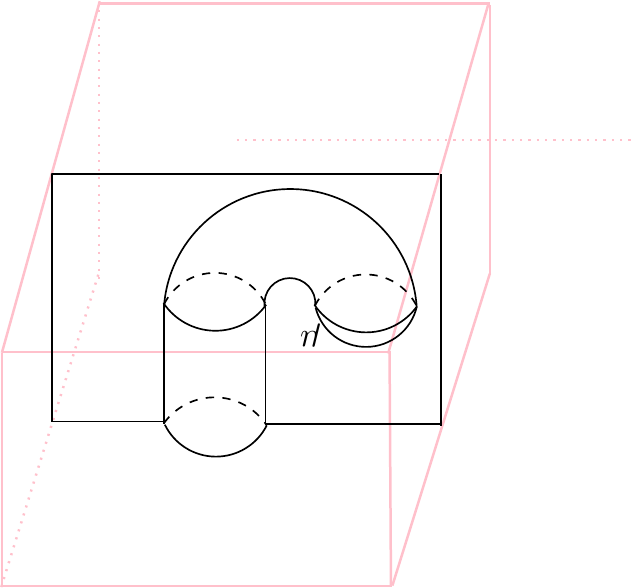}} \]
as it rotates the 3-morphism $\mu$.
The orientation preserving isotopies coming from the interchange law will be the only moves we allow ourselves in the 3-category $\BBBrPic(\cC)$.


We can truncate $\BBBrPic(\cC)$ to a 2-category $\BBrPic(\cC)$ by collapsing natural isomorphisms to identities. Thus we also have a 2-dimensional calculus for $\BBrPic(\cC)$. Again we don't use the fact that $\BBrPic(\cC)$ is rigid, and just use the plain 2-dimensional graphical calculus. We pick a coordinate system:
\[\raisebox{-.5\height}{ \includegraphics[scale = .45]{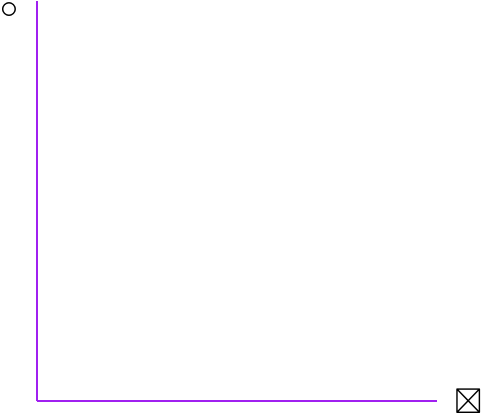}}\]
and again only allow isotopy that preserves orientation of 2-morphisms with respect to these axis. Hence we are allowed to perform moves such as
\[\raisebox{-.5\height}{ \includegraphics[scale = .65]{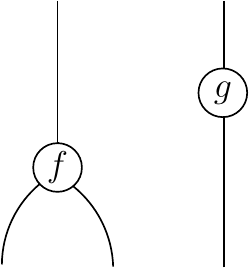}} \eq  \raisebox{-.5\height}{ \includegraphics[scale = .65]{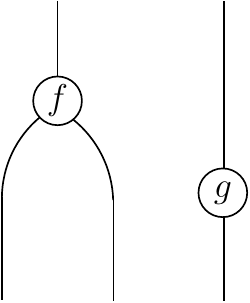}} \]
but not moves such as
\[\raisebox{-.5\height}{ \includegraphics[scale = .65]{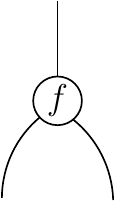}} \eq  \raisebox{-.5\height}{ \includegraphics[scale = .65]{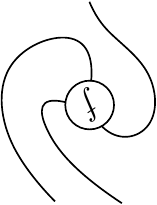}} \]

We will make heavy use of these graphical calculi in this paper. To help with computations we introduce some distinguished morphisms and relations between them.

Let's start with the 2-category $\BBrPic(\cC)$. Let $\cM$ be an invertible $\cC$ bimodule. As $\cM$ is invertible, there exists a bimodule $\cM^{op}$, and (non-unique) bimodule equivalences::
\[    \operatorname{L}_\cM : \cM^{op} \boxtimes \cM \to \cC \quad \text{ and } \quad \operatorname{R}_\cM : \cM \boxtimes \cM^{op} \to \cC\] 
These bimodule equivalences give us 2-morphisms in the category $\BBrPic(\cC)$. Graphically we draw these 2-morphisms as:
\[ \raisebox{-.5\height}{ \includegraphics[scale = .45]{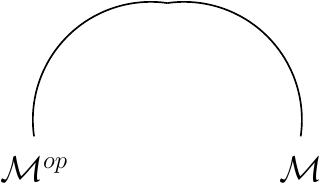}}\quad  \text{ and } \quad \raisebox{-.5\height}{ \includegraphics[scale = .45]{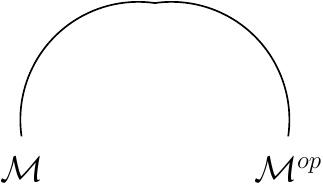}} \]
respectively.
As the morphisms $\operatorname{L}_\cM$ and $\operatorname{R}_\cM$ are equivalences, we also have their inverses
\[  \operatorname{L}_\cM^{-1} : \cC \to \cM^{op} \boxtimes \cM  \quad \text{ and } \quad 	 \operatorname{R}_\cM^{-1} : \cC \to \cM \boxtimes \cM^{op},\]
which we draw as 
\[ \raisebox{-.5\height}{ \includegraphics[scale = .45]{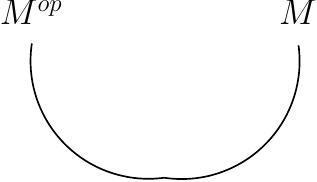}} \quad \text{ and }  \quad \raisebox{-.5\height}{ \includegraphics[scale = .45]{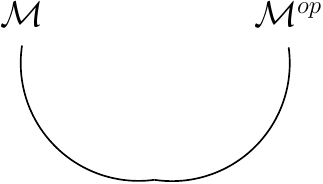}} \]
respectively.

These distinguished 2-morphisms satisfy the relations 
\[ \operatorname{L}_\cM \circ \operatorname{L}_\cM^{-1} = \Id_\cC = \operatorname{R}_\cM \circ \operatorname{R}_\cM^{-1},\] 
\[ \operatorname{L}_\cM^{-1} \circ \operatorname{L}_\cM = \Id_{\cM^\text{op}} \boxtimes \Id_{\cM},\]
and 
\[ \operatorname{R}_\cM^{-1} \circ \operatorname{R}_\cM = \Id_{\cM} \boxtimes \Id_{\cM^\text{op}}.\]
 Hence we get the graphical relations:
\[ \raisebox{-.5\height}{ \includegraphics[scale = .45]{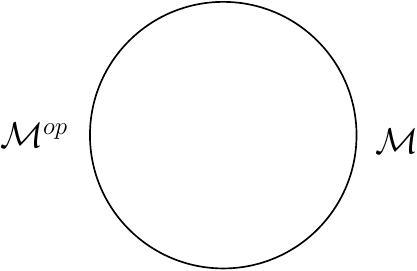}}\quad  = \quad  1 \quad  = \quad \raisebox{-.5\height}{ \includegraphics[scale = .45]{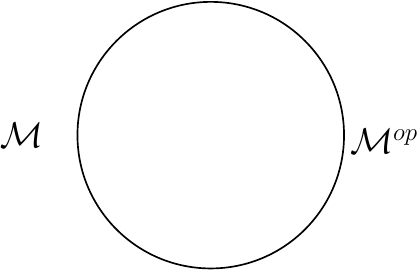}}, \]
\vspace{1em}
\[       \raisebox{-.5\height}{ \includegraphics[scale = .45]{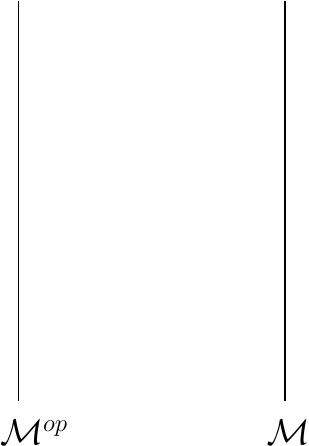}} \quad = \quad \raisebox{-.5\height}{ \includegraphics[scale = .45]{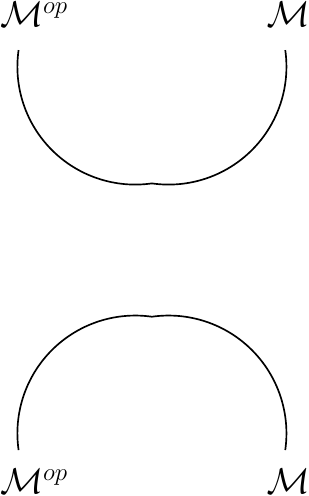}} \quad \text{ and }\quad  \raisebox{-.5\height}{ \includegraphics[scale = .45]{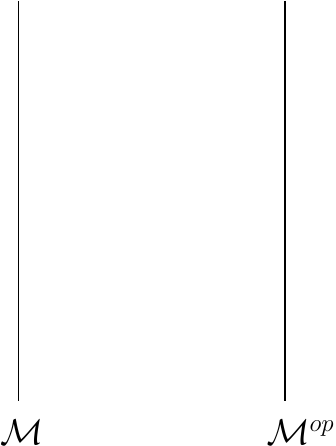}} \quad = \quad \raisebox{-.5\height}{ \includegraphics[scale = .45]{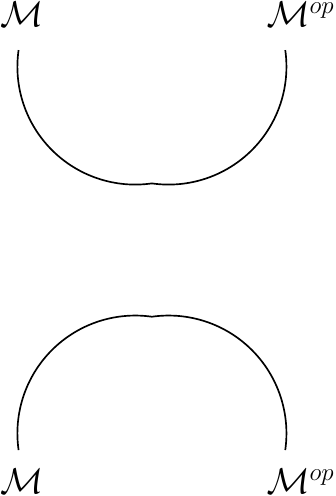}}\]
Note that we do not have graphical relations to straighten out zig-zags in this graphical calculus. That is
\[   \raisebox{-.5\height}{ \includegraphics[scale = .45]{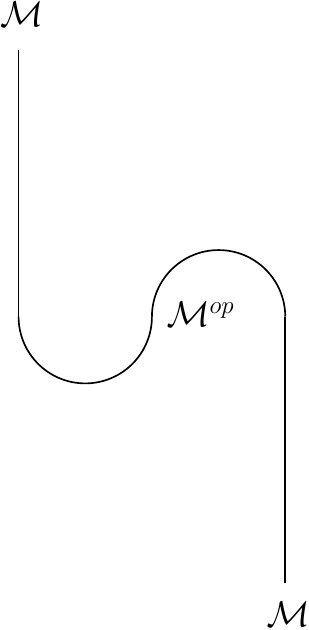}}  \quad  \neq \quad  \raisebox{-.5\height}{ \includegraphics[scale = .45]{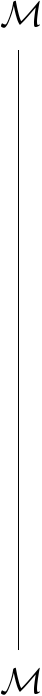}}\]

However as a substitute we do have the following relations which are of key importance for this paper. 
\begin{lem}\label{lem:anue}
Let $\cM$ an invertible $\cC$-bimodule, and $\mathcal{F}$ a bimodule auto-equivalence of $\cM$. We have the following relations of 2-morphisms in $\BBrPic(\cC)$:
\[        \raisebox{-.5\height}{ \includegraphics[scale = .45]{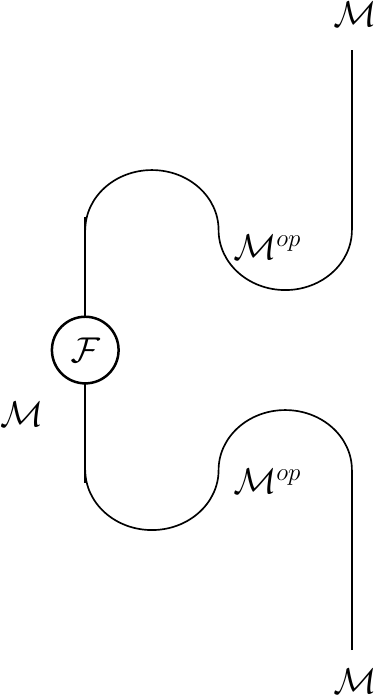}} \eq \raisebox{-.5\height}{ \includegraphics[scale = .45]{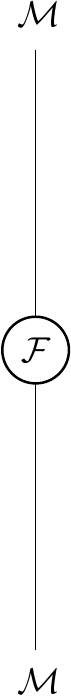}}  \eq \raisebox{-.5\height}{ \includegraphics[scale = .45]{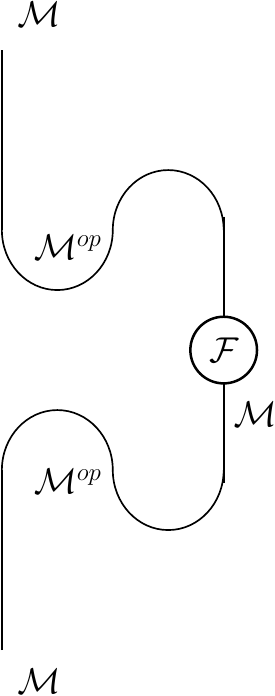}}  \]
\end{lem}
\begin{proof}
The first relations holds as:
\[        \raisebox{-.5\height}{ \includegraphics[scale = .45]{anue}} \eq \raisebox{-.5\height}{ \includegraphics[scale = .45]{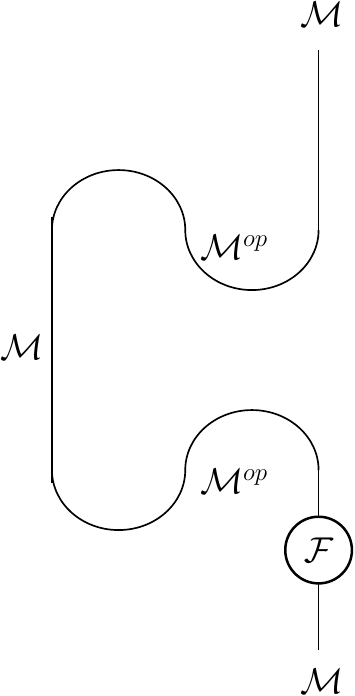}} \eq    \raisebox{-.5\height}{ \includegraphics[scale = .45]{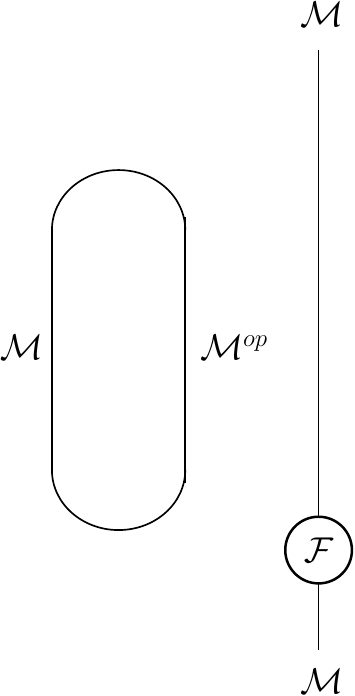}} \quad  = \quad \raisebox{-.5\height}{ \includegraphics[scale = .45]{idM2}},\]
where the first equality follows as $\Aut(\cM)$ is abelian, the second equality follows by recoupling the adjacent cups and caps, and the third equality is from popping the bubble.

The proof of the second relation is near identical.
\end{proof}
%
These distinguished morphisms in $\BBrPic(\cC)$, along with the relations we have described above are sufficient to prove everything regarding $\BBrPic(\cC)$ in this paper. 

We now focus on the 3-dimensional graphical calculus for $\BBBrPic(\cC)$. As with the 2-dimensional calculus for $\BBrPic(\cC)$ we pick out some distinguished 3-morphisms, and give graphical interpretations of the relations they satisfy.

Let $\mathcal{F}:\cM_1\boxtimes \cM_2 \to \cc{N}$ be a bimodule equivalence, then there exists a bimodule equivalence $\cF^{-1} : \cc{N} \to \cM_1\boxtimes \cM_2$ along with (non-unique) natural isomorphisms:
\[ \operatorname{l}_\cF : \cF^{-1}\circ \cF \to \Id_{\cM_1} \boxtimes \Id_{\cM_2} \quad \text{ and }\quad   \operatorname{r}_\cF : \cF\circ \cF^{-1} \to \Id_{\cc{N}}.\]

Graphically we draw these natural isomorphisms as:
\[ \raisebox{-.5\height}{ \includegraphics[scale = .45]{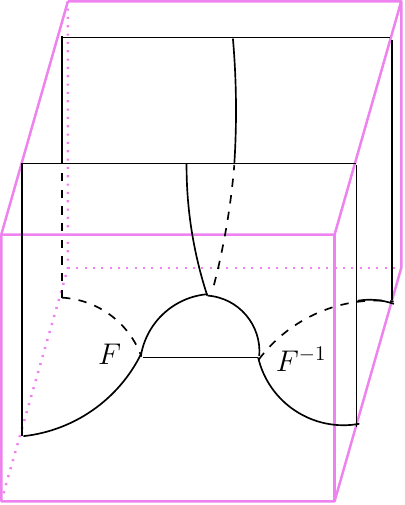}} \qquad  \text{ and } \qquad \raisebox{-.5\height}{ \includegraphics[scale = .45]{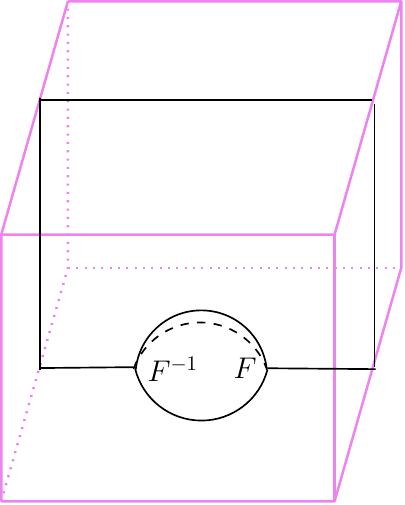}} \]
respectively.
As $\operatorname{l}_\cF$ and $\operatorname{r}_\cF$ are isomorphisms, they have inverses 
\[ \operatorname{l}^{-1}_\cF : \Id_{\cM_1}\boxtimes \Id_{\cM_2} \to  \cF^{-1}\circ \cF \quad \text{ and }\quad   \operatorname{r}^{-1}_\cF :  \Id_{\cc{N}}\to  \cF\circ \cF^{-1},\]
which we draw as:
\[ \raisebox{-.5\height}{ \includegraphics[scale = .45]{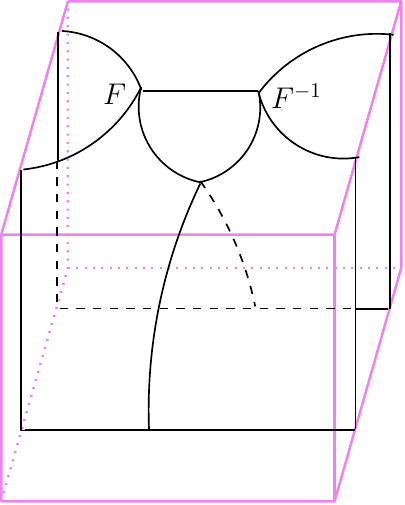}} \qquad  \text{ and } \qquad \raisebox{-.5\height}{ \includegraphics[scale = .45]{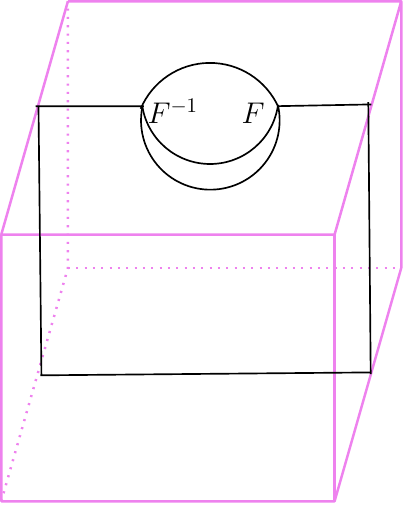}} \]

These distinguished 3-morphisms satisfy the relations 
\begin{align*}
 \operatorname{l}_\cF \circ \operatorname{l}_{\cF}^{-1} &= \Id_{\cF^{-1}} \circ \id_\cF   \qquad \qquad \operatorname{l}^{-1}_\cF \circ \operatorname{l}_{\cF} = \id_{\Id_{\cM_1}} \boxtimes \id_{\Id_{\cM_2}}\\
\operatorname{r}_\cF \circ \operatorname{r}_{\cF}^{-1} &=\Id_{\cF} \circ \id_{\cF^{-1}}  \qquad \qquad \operatorname{r}^{-1}_\cF \circ \operatorname{r}_{\cF} = \id_{\Id_{\cc{N}}}.
\end{align*}
We leave it to the reader to draw the resulting graphical relations we get from the four equations above. These general relations will not be so important for this paper. What is important are the specialisations of $\mathcal{F}$ to the bimodule equivalences $L_\cM: \cM^{op} \boxtimes \cM \to \cC$ and $R_{\cM}: \cM \boxtimes \cM^{op}\to \cC$. For these cases we get the following graphical relations (here $X$ can be either $L_\cM$ or $R_\cM$):
\[ \raisebox{-.5\height}{ \includegraphics[scale = .45]{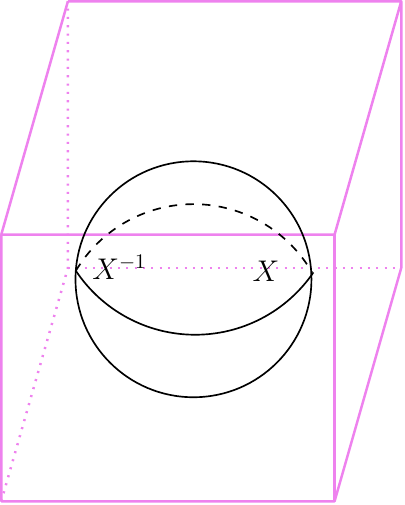}} \eq 1  \qquad , \qquad  \raisebox{-.5\height}{ \includegraphics[scale = .45]{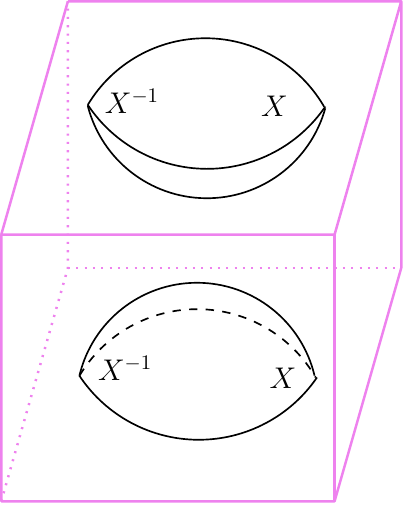}} \quad = \quad \raisebox{-.5\height}{ \includegraphics[scale = .45]{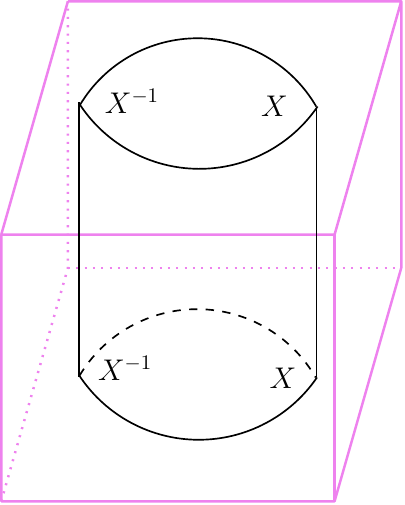}},\]
\[\raisebox{-.5\height}{ \includegraphics[scale = .45]{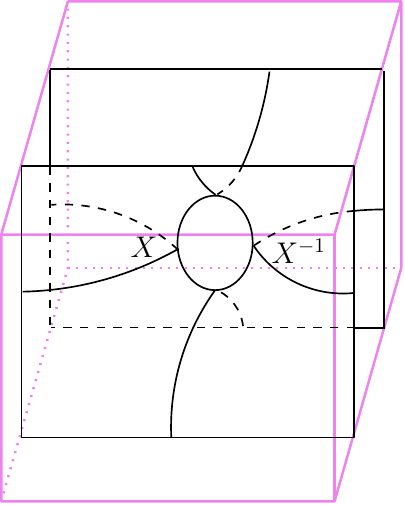}}  \quad=\quad  \raisebox{-.5\height}{ \includegraphics[scale = .45]{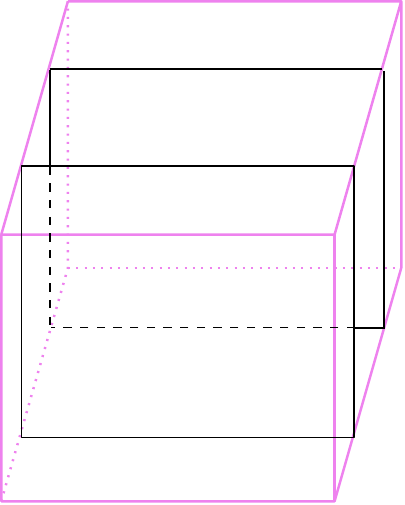}} \qquad \text{and}\qquad \raisebox{-.5\height}{ \includegraphics[scale = .45]{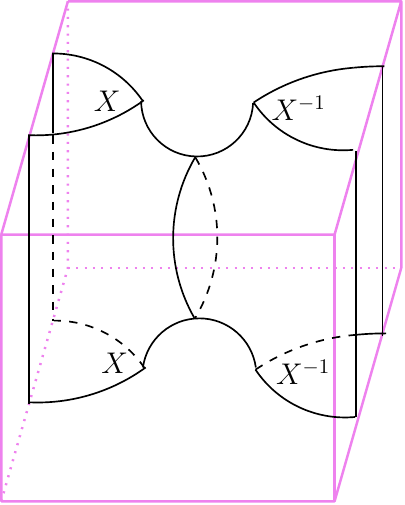}}  \quad=\quad  \raisebox{-.5\height}{ \includegraphics[scale = .45]{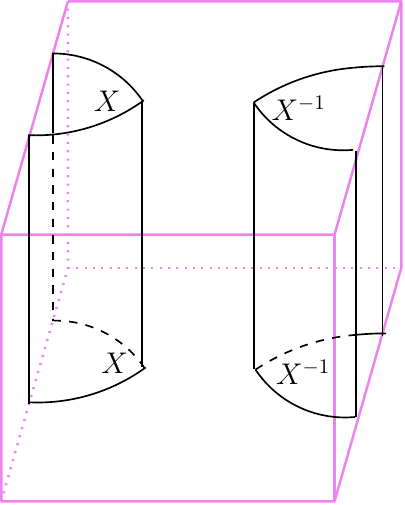}}. \]

These distinguished 3-morphisms and relations will be sufficient for this paper to prove everything regarding $\BBBrPic(\cC)$ we need for this paper.

\subsection*{An action of $\operatorname{BrPic}(\cC)$ on $\Inv(\cZ(\cC))$}
It is shown in \cite{MR2677836} that the group of invertible elements of the centre of $\cC$ and the group of bimodule auto-equivalences of $\cC$ are isomorphic. This isomorphism sends the invertible central object $(X,\gamma_{X,-})$ to the bimodule auto-equivalence $ X\triangleright ?$. The left bimodule structure map is given by $\gamma_{X,-}$, and the right bimodule structure map is trivial. We will implicitly use this isomorphism throughout this paper, regarding elements of $\Inv(\cZ(\cC))$ as auto-equivalences of the trivial bimodule $\cC$, and vice versa.

Given an invertible bimodule $\cM$, and $z\in \Inv(\cZ(\cC))$, we can define

\[   z^{\cM} := \operatorname{R}_\cM \circ ( \id_\cM \boxtimes z \boxtimes \id_{\cM^{op}}) \circ \operatorname{R}_{\cM}^{-1}: \cC \to \cC,  \]
a new element of $\Inv(\cZ(\cC))$. Using the graphical calculus for $\BBrPic(\cC)$ we draw:
\[z^{\cM} = \raisebox{-.5\height}{ \includegraphics[scale = .7]{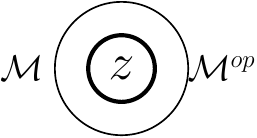}}\]

It is easy to check using the graphical relations that this determines an action of the group $\BrPic(\cC)$ on the group $\Inv(\cZ(\cC))$. In fact this is exactly the action defined in \cite{MR2677836}, which they prove is independent of the choice of equivalence $\operatorname{R}_\cM: \cM \boxtimes \cM^{op} \to \cC$. The following Lemma about this action is repeatedly used throughout this paper.

\begin{lem}\label{lem:lact}
Let $z \in \Inv(\cZ(\cC))$ and $\mathcal{M}$ an invertible $\cC$-bimodule. Then we have the following relation in $\BBrPic(\cC)$: 
\[\raisebox{-.5\height}{ \includegraphics[scale = .7]{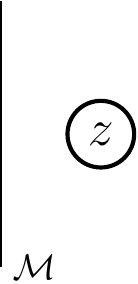}} \eq  \raisebox{-.5\height}{ \includegraphics[scale = .7]{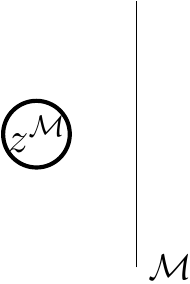}}\]
\end{lem}
\begin{proof}
We directly compute using the string recoupling and bubble popping relations:

\[\raisebox{-.5\height}{ \includegraphics[scale = .7]{actionlem0}} \eq \raisebox{-.5\height}{ \includegraphics[scale = .7]{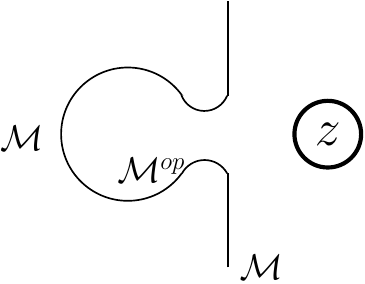}}  \eq \raisebox{-.5\height}{ \includegraphics[scale = .7]{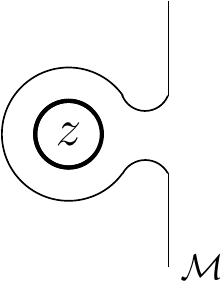}}\eq \raisebox{-.5\height}{ \includegraphics[scale = .7]{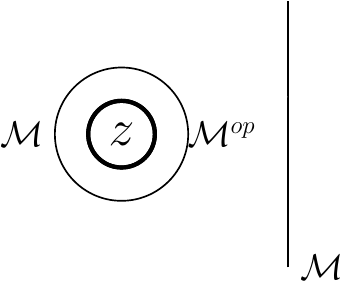}} \] 
The first equality is Lemma~\ref{lem:anue}, the second is from an isotopy, the third is from recoupling the adjacent cup and caps. 
\end{proof}
In fact the object $z^\cM$ can be defined as the unique object of $\Inv(\cZ(\cC))$ such that Lemma~\ref{lem:lact} holds.

\section{Classification of graded extensions}\label{sec:gradedex}

Let $\cC$ be a fusion category, and $G$ a finite group. Recall from \cite{MR2677836} that $\Ext$, the set of all $G$-graded extensions of $\cC$ (up to equivalence of extensions) is in bijection with the set of triples $(c,M,A)$, where: 
\begin{itemize}
\item $c: G \to \BrPic(\cC)$ is a group homomorphism such that a certain element 
\[o_3(c) \in H^3(G,\Inv(\cZ(\cC)))\]
 is trivial,
\item $M$ is an element of a $H^2(G,\Inv(\cZ(\cC)))$-torsor, such that a certain element 
\[o_4(c,M) \in H^3(G,\Inv(\cZ(\cC)))\]
 is trivial,
\item $A$ is an element of an $H^3(G,\mathbb{C}^\times)$-torsor.
\end{itemize}
The proof of our classification of $\cC$-based equivalences makes use of the techniques developed in the proof of this classification of $G$-graded extensions. To warm the reader up to these techniques, we quickly reprove this classification of Etingof, Nikshyck, and Ostrik, putting an emphasis on string diagrams.

Let $c :G \to \BrPic(\cC)$ be a group homomorphism. Then we can form the $\cC$-bimodule category
\[ \cc{D} := \bigoplus c_g. \]
We wish to make $\cD$ a graded quasi-monoidal category. Recall this is a category with a tensor product satisfying $(X\otimes Y) \otimes Z \cong X \otimes (Y\otimes Z)$ via an unspecified isomorphism. As $c$ is a homomorphism there exists a collection of bimodule equivalences:
\[ M:= \{ M_{g,h}: c_g \boxtimes c_h \to c_{gh}\} = \left\{  \raisebox{-.5\height}{ \includegraphics[scale = .6]{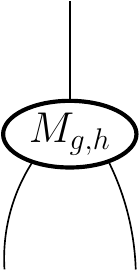}} \right \}  .\] 

Via Equation~\eqref{eq:funball} the equivalences $M_{g,h}$ induce $\cC$-balanced bimodule functors $m_{g,h}: c_g \times c_h \to c_{gh}$. We define a tensor product on $\cD$ by:
\[    X_g \otimes Y_h := m_{g,h}(X_g,Y_h). \]
This tensor product on $\cc{D}$ very much depends on the choice of $M$.

For the category $\cD$ with tensor product coming from the collection $M$ to be quasi-monoidal we exactly need $\cC$-balanced natural isomorphisms of bimodule functors: 
\[m_{fg,h}\circ(m_{f,g} \times \id_{c_h}) \cong m_{f,gh}\circ(\id_{c_f} \times m_{g,h}), \]
which via Equation~\eqref{eq:funball} is equivalent to the existence of bimodule natural isomorphisms:
\[M_{fg,h}\circ(M_{f,g} \boxtimes \id_{c_h}) \cong M_{f,gh}\circ(\id_{c_f} \boxtimes M_{g,h}). \]
To help us determine when we have such bimodule isomorphisms we define 
\[ T(c,M)_{f,g,h} :=\operatorname{R}_{c_{fgh}}\circ ( [M_{fg,h}\circ(M_{f,g} \boxtimes \id_{c_h}) \circ(\id_{c_f} \boxtimes M_{g,h}^{-1}) \circ M_{f,gh}^{-1}]\boxtimes \Id_{c_{fgh}^{op}})\circ \operatorname{R}^{-1}_{c_{fgh}}: G\times G \times G \to \Inv(\cZ(\cC)),\]
which we draw graphically as 
\begin{center} $T(c,M)_{f,g,h} =\raisebox{-.5\height}{ \includegraphics[scale = .6]{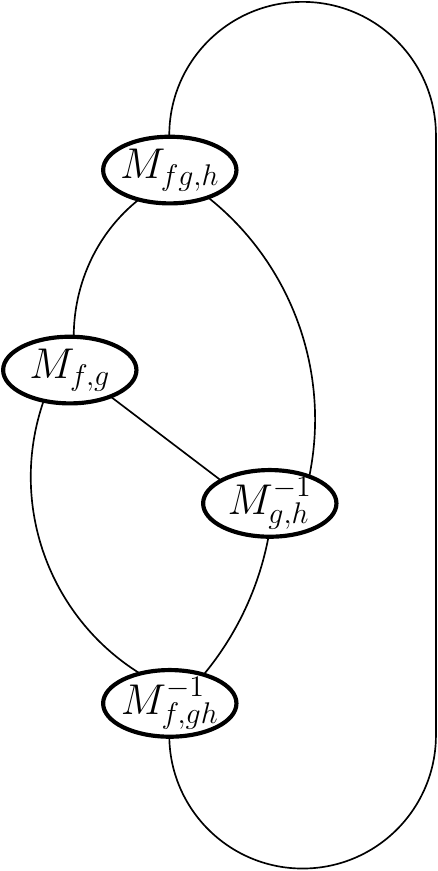}}$  \end{center} 
The category $\cD$ with tensor product coming from $M$ is quasi-monoidal if and only if $T(c,M)_{f,g,h} = \Id_\cC$ for all $f,g,h \in G$. We call a collection of bimodule equivalences $M$ such that $T(c,M)$ vanishes a \textit{system of products for $c$}. We consider systems of products for $c$ up to quasi-monoidal equivalence of extensions of the corresponding quasi-monoidal extensions of $\cC$. 

We directly verify that $T(c,M)$ is a 3-cocycle valued in $\Inv(\cZ(\cC))$.
\begin{lem}
We have 
\[T(c,M)_{f,g,h}T(c,M)^{c_f}_{g,h,k}T(c,M)_{f,gh,k}  = T(c,M)_{f,g,hk}T(c,M)_{fg,h,k}.\]
\end{lem}
\begin{proof}
We prove that $T(c,M)$ is a 3-cocycle via the graphical calculus. At each stage in this calculation, we indicate with a red box the region we will modify, and explain the modification in written words below the picture proof.

\begin{align*}
&T(c,M)_{f,g,h}T(c,M)^{c_f}_{g,h,k}T(c,M)_{f,gh,k} =  \\
&\raisebox{-.5\height}{ \includegraphics[scale = .45]{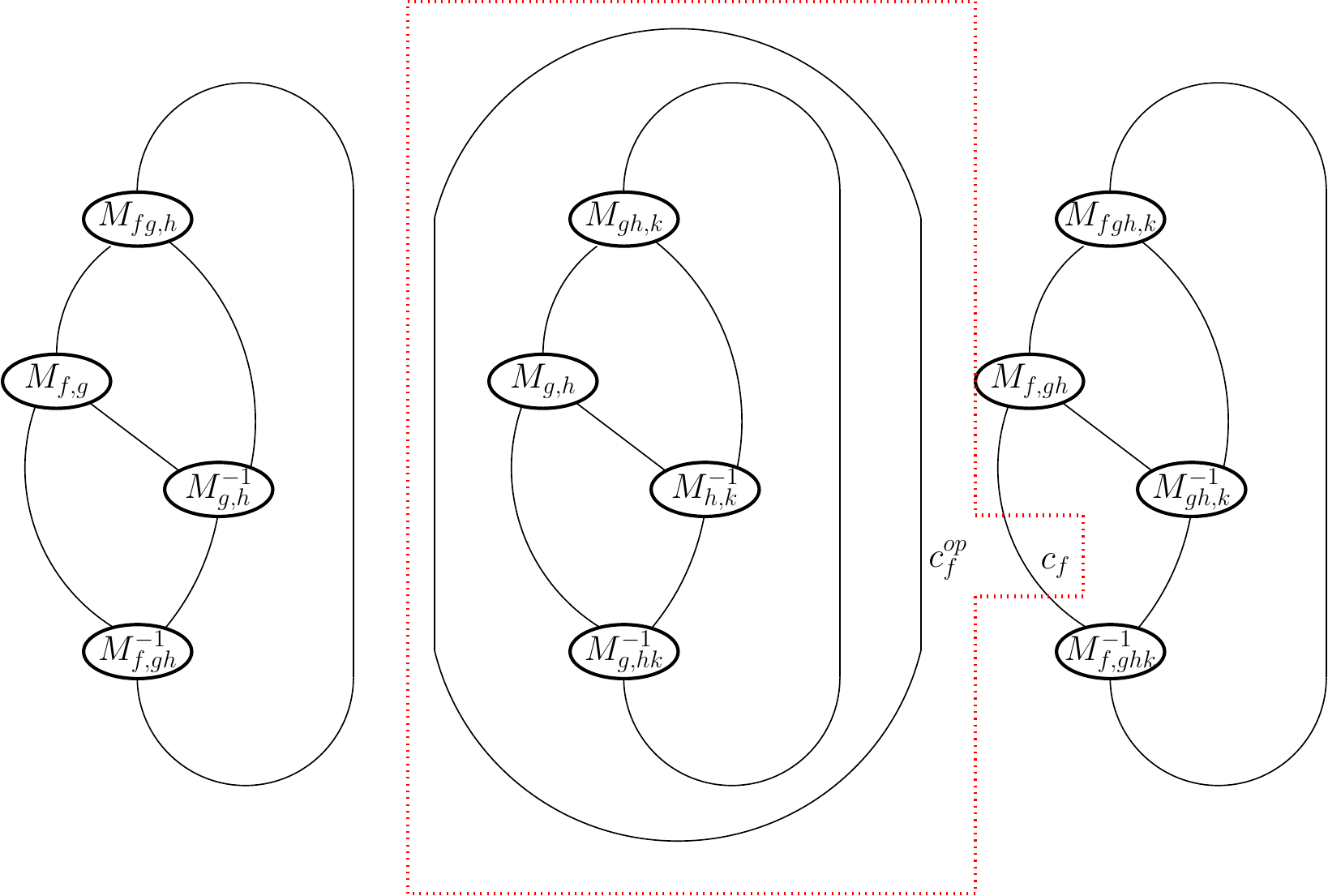}} \quad = \\
&\raisebox{-.5\height}{ \includegraphics[scale = .45]{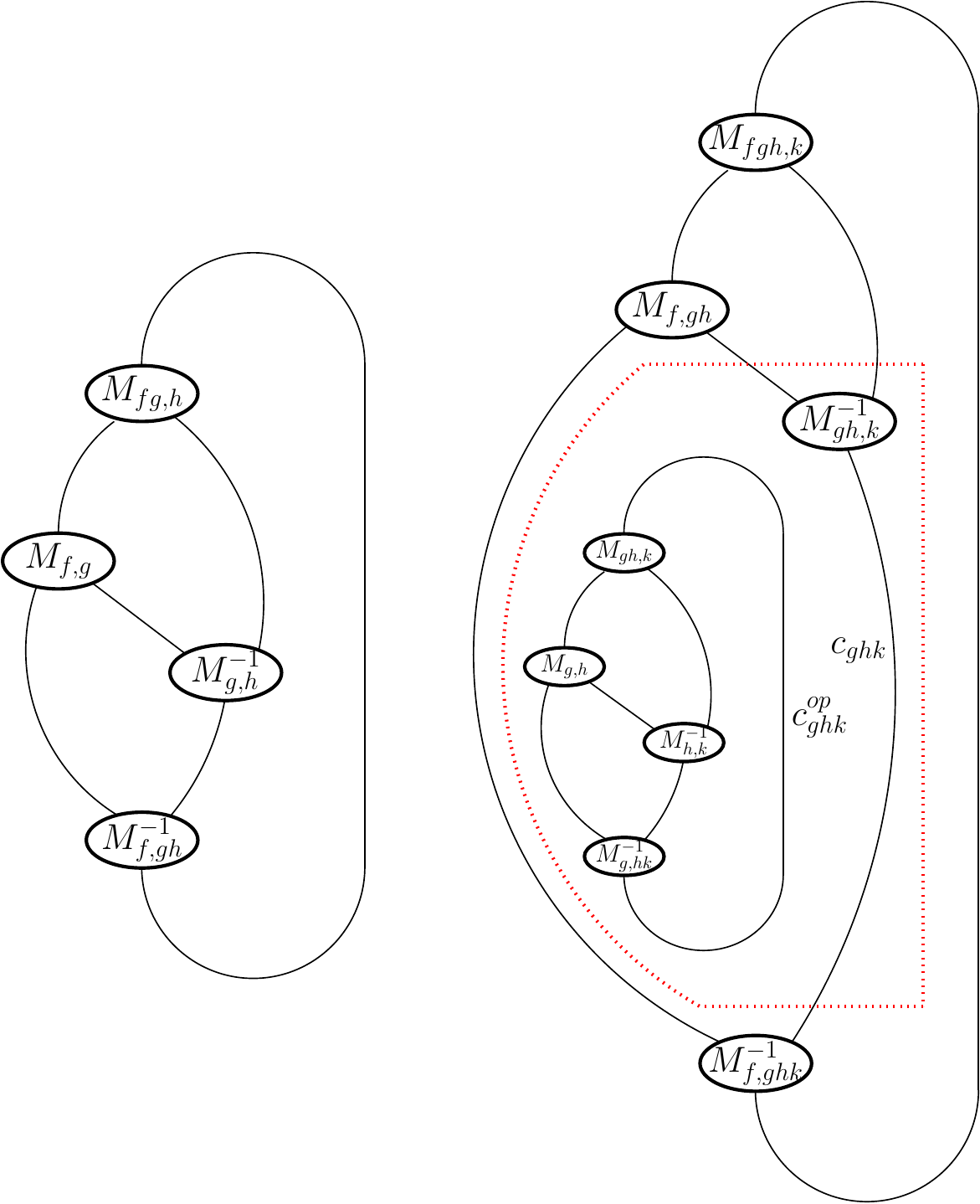}} \quad =  \\
&\raisebox{-.5\height}{ \includegraphics[scale = .45]{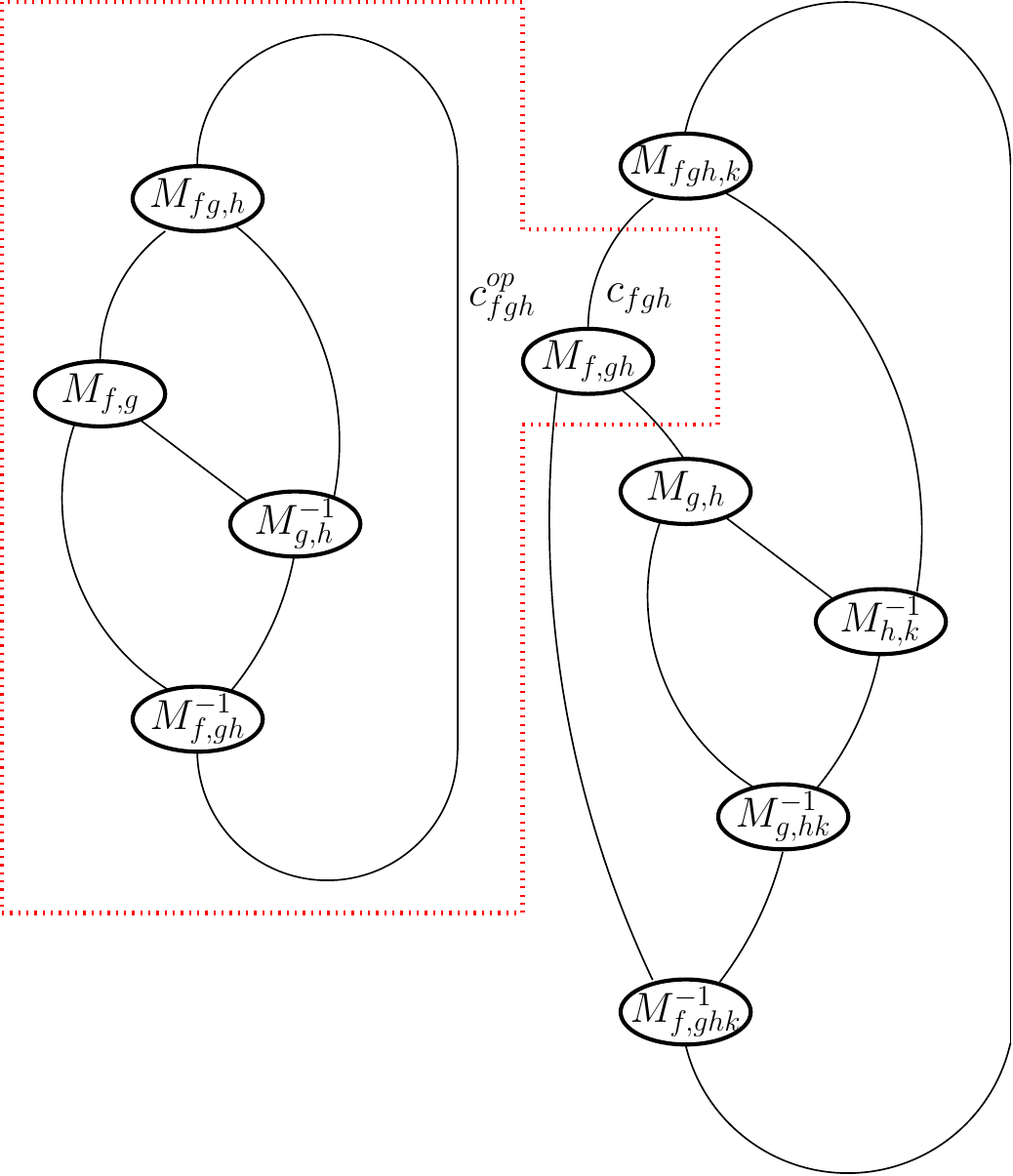}} \eq \raisebox{-.5\height}{ \includegraphics[scale = .4]{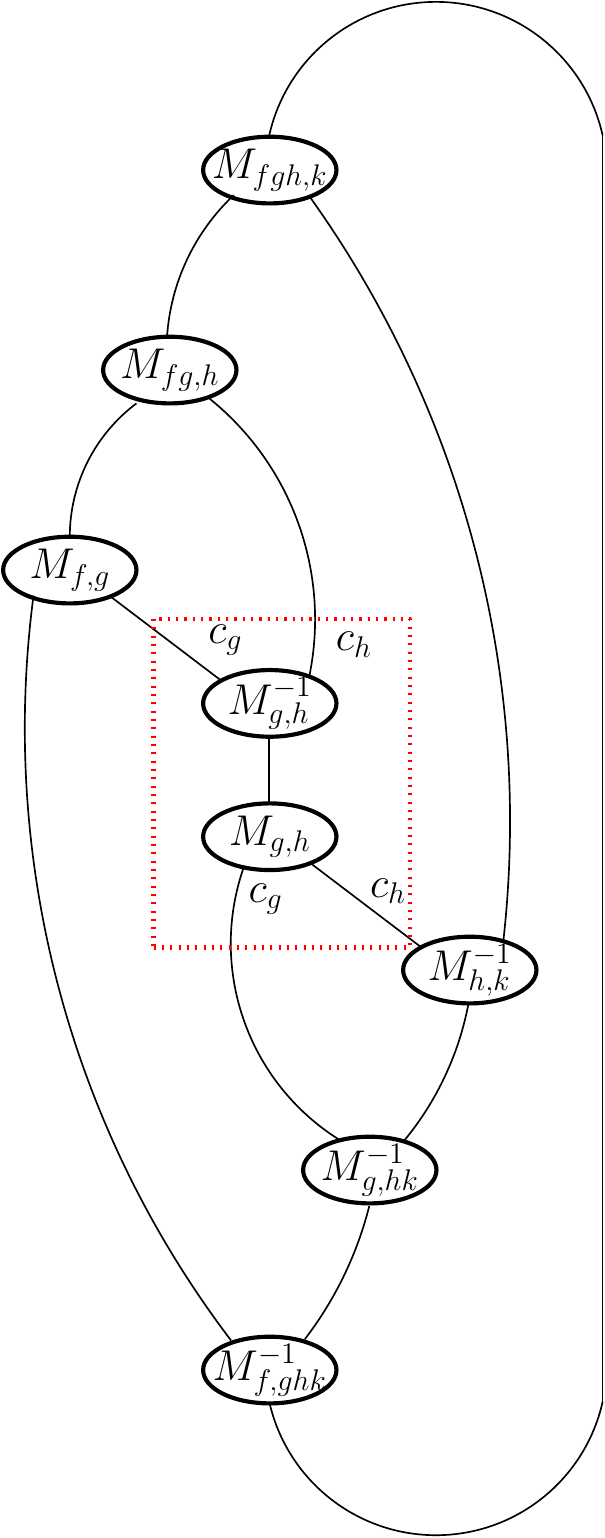}} \quad =  \\ 
& \raisebox{-.5\height}{ \includegraphics[scale = .45]{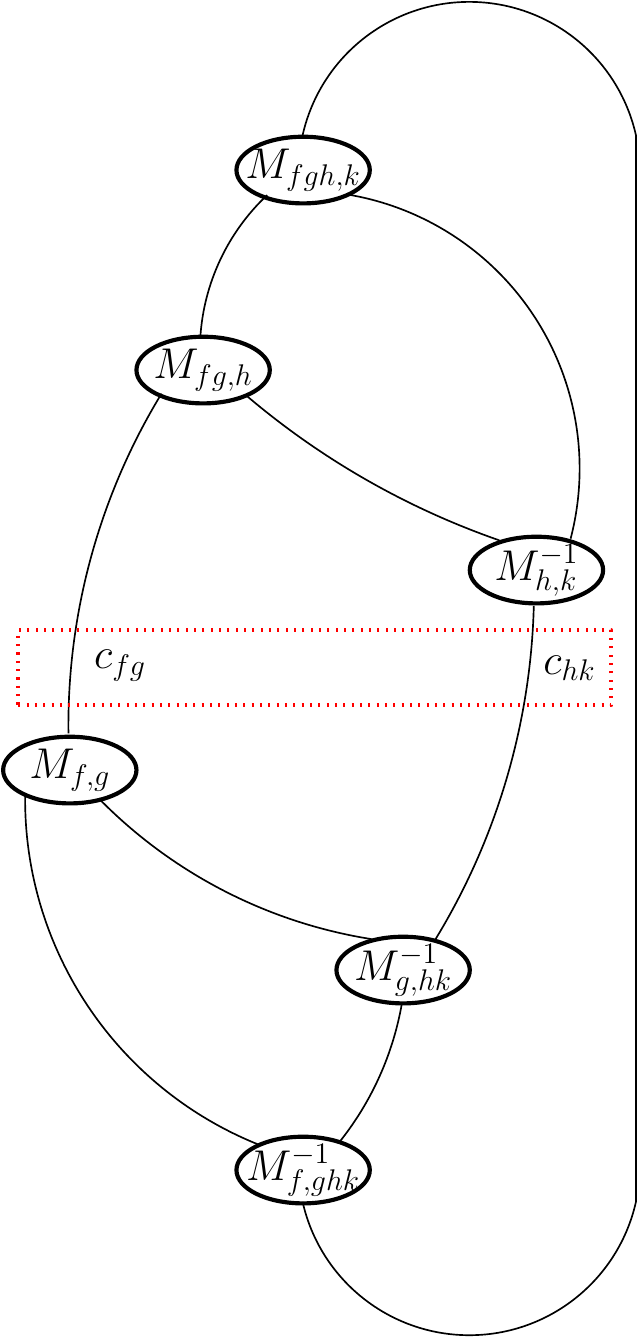}}  \eq \raisebox{-.5\height}{ \includegraphics[scale = .4]{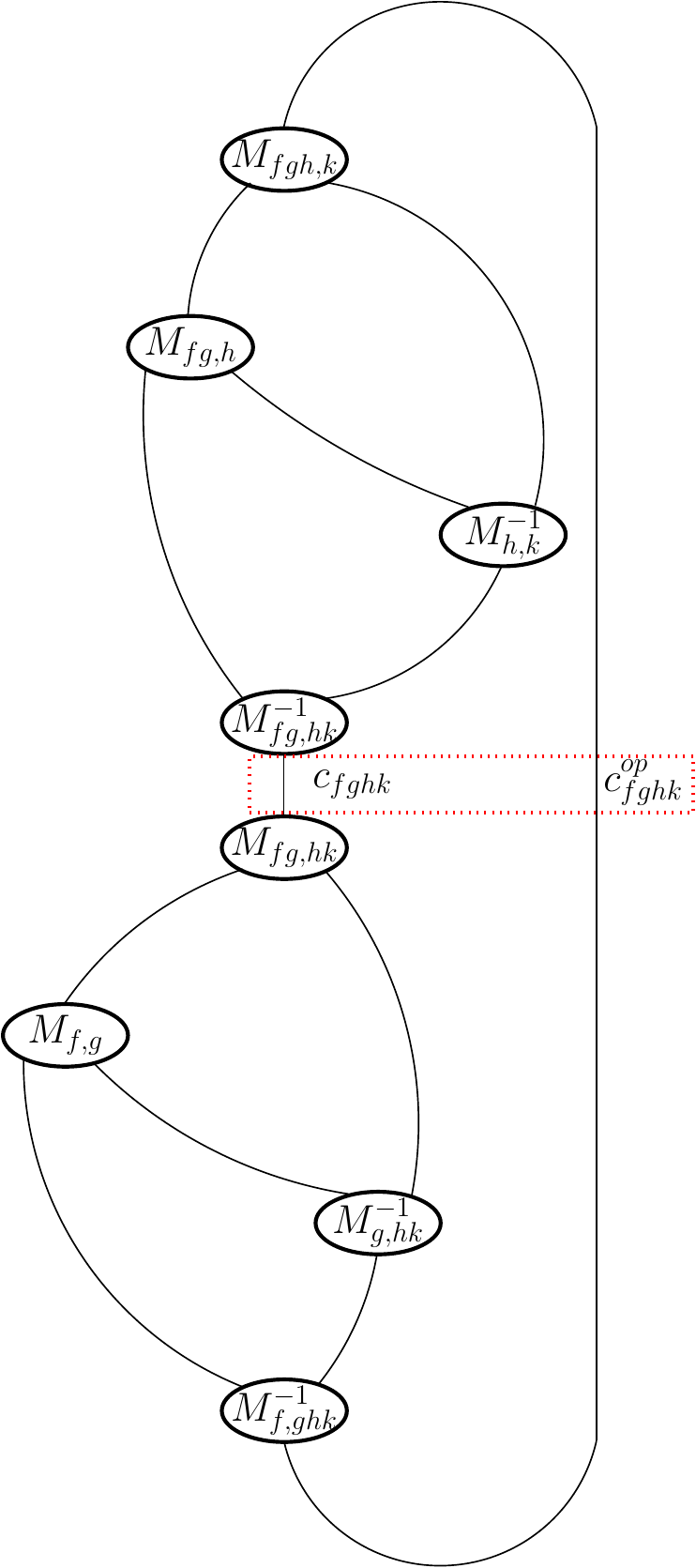}} \quad = \\ 
& \raisebox{-.5\height}{ \includegraphics[scale = .45]{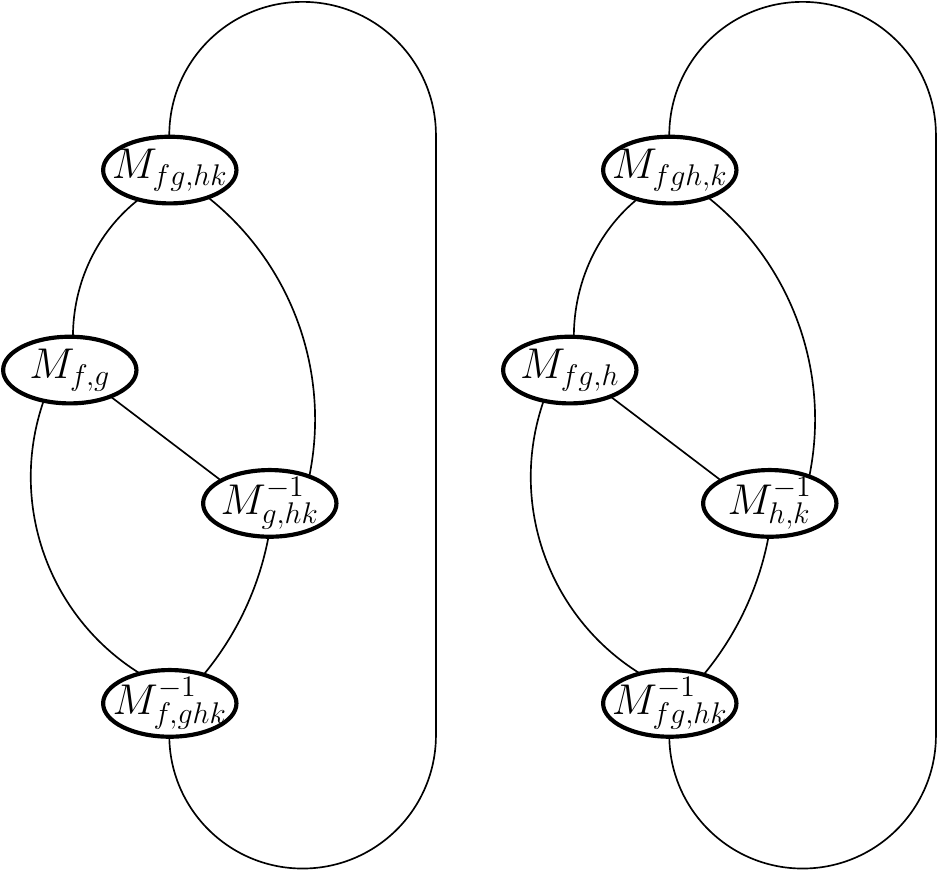}}\quad=  T(c,M)_{f,g,hk}T(c,M)_{fg,h,k}.
\end{align*}
For the first equality we just apply Lemma~\ref{lem:lact}. For the second equality we recouple the $c^{op}_{ghk}$ and $c_{ghk}$ strings, then we apply Lemma~\ref{lem:anue}, and cancel the $M_{gh,k}$ with its inverse. For the third equality we recouple the $c^{op}_{fgh}$ and $c_{fgh}$ strings, then we apply Lemma~\ref{lem:anue}, and cancel the $M_{f,gh}$ with its inverse. For the fourth equality we cancel $M_{g,h}$ with its inverse. For the fifth equality we replace $\Id_{c_{fg}} \boxtimes \Id_{c_{hk}}$ with $M_{fg,hk}^{-1} \circ M_{fg,hk}$. For the final equality we recouple the $c^{op}_{fghk}$ and $c_{fghk}$ strings.
\end{proof}

To simplify notation and proofs for the remainder of the Section we introduce the following notation.
\begin{defn}\label{def:Maction}
Let $\rho \in C^2(G,\Inv(\cZ(\cC)))$ and $M$ a collection of bimodule equivalences $\{  c_g\boxtimes c_h \to c_{gh} \}$. We define 
\[     M^\rho:= \{ \rho_{g,h}\boxtimes M_{g,h}:   c_g \boxtimes c_h \to c_{gh}\} = \left\{  \raisebox{-.5\height}{ \includegraphics[scale = .6]{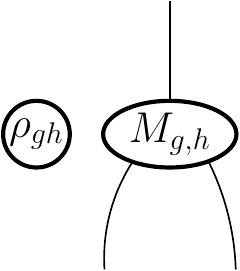}} \right \} .\]
\end{defn}
Note that if $\rho \in C^2(G,\Inv(\cZ(\cC)))$ is arbitrary, and $M$ is a system of products for $c$, i.e. $T(c,M)$ vanishes, then it is not guaranteed that $M^\rho$ is a system of products for $c$. In fact the following Lemma shows $M^\rho$ is a system of products for $c$ if and only if $\rho$ is a coboundary.

\begin{lem}\label{lem:3cob}
We have 
\[	T(c,M^\rho)_{f,g,h} = 	\partial^3(\rho)_{f,g,h}T(c,M)_{f,g,h} .	\]
\end{lem}
\begin{proof}
We compute using Lemma~\ref{lem:lact} that
\[ T(c,M^\rho)_{f,g,h} =  \raisebox{-.5\height}{ \includegraphics[scale = .5]{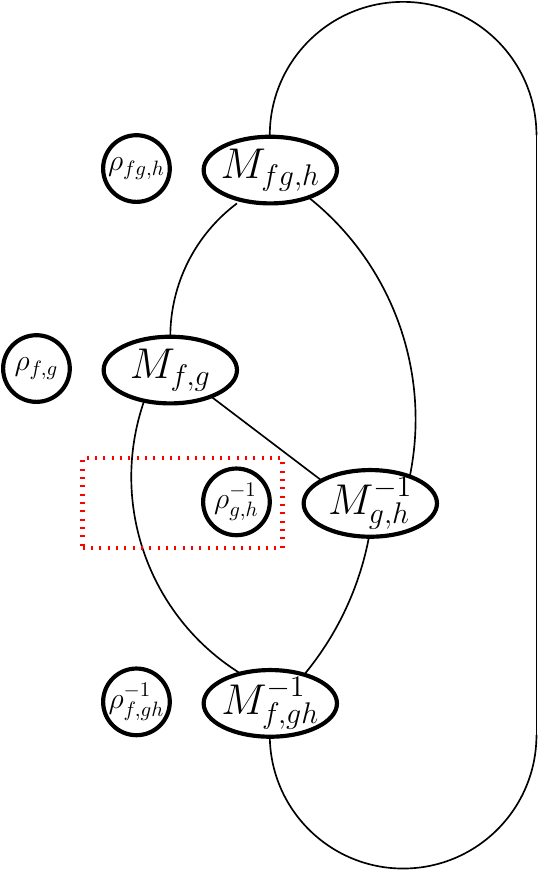}}   = \raisebox{-.5\height}{ \includegraphics[scale = .5]{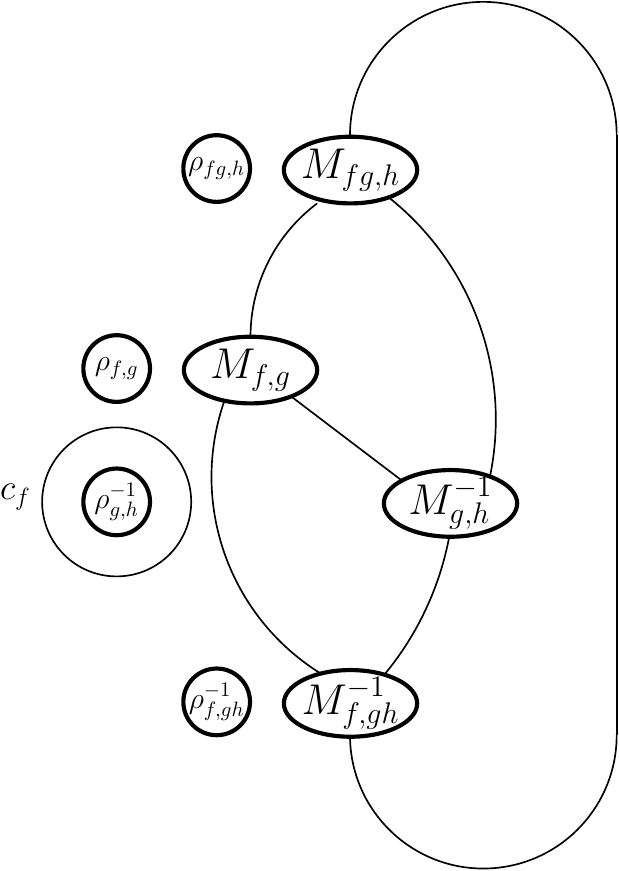}} = \rho_{fg,h}\rho_{f,g}{\rho_{g,h}^{-1}}^{c_f}\rho_{f,gh}^{-1}T(c,M).\]
\end{proof}

Suppose we have $M$ and $M'$  both collections of bimodule equivalences $\{  c_g\boxtimes c_h \to c_{gh} \}$. Then each $M'_{g,h}$ differs from $M_{g,h}$ by a bimodule auto-equivalence of $c_{gh}$, and thus by Lemma~\ref{lem:anue} we have $M'_{g,h} = L_{g,h}M_{g,h}$ for some $L_{g,h} \in \Inv(\cZ(\cC))$. In the notation of Definition~\ref{def:Maction} we thus have $M' = M^L$. Applying Lemma~\ref{lem:3cob} then gives 
\[ T(c,M') = \partial^3(L)T(c,M). \]
Hence $T(c,M')$ differs from $T(c,M)$ by a coboundary in $Z^3(G,\Inv(\cZ(\cC)))$. Thus the cohomology class of $T(c,M)$ is independent of the choice of $M$. This implies that \[o_3(c):= T(c,-)\] is an element of $H^3(G,\Inv(\cZ(\cC)))$ that depends only on $c$.

From the element $o_3(c)$ we can deduce a necessary and sufficient condition for there to exist a system of products for $c$, and hence for the category $\cD$ to be quasi-monoidal.
\begin{lem}\label{lem:nonempo3}
There exists a system of products for $c$ if and only if $o_3(c)$ is the trivial element of $H^3(G,\Inv(\cZ(\cC)))$.
\end{lem}
\begin{proof}
First we assume $o_3(c)$ is the trivial element of $H^3(G,\Inv(\cZ(\cC)))$. Choose $M$ an arbitrary collection of bimodule equivalences $\{ M_{g,h}: c_g \boxtimes c_h \to c_{gh} \}$, which we can do as $c$ is a group homomorphism. As $o_3(c)$ is trivial in $H^3(G,\Inv(\cZ(\cC)))$ we have that $T(c,M) = \partial^3(\rho)$ for some $\rho \in C^2(G, \Inv(\cZ(\cC)))$. We choose a new collection of bimodule equivalences $M^{\rho^{-1}}$. An application of Lemma~\ref{lem:3cob} gives 
\[ T(c,M^{\rho^{-1}}) = \partial^3(\rho^{-1}) T(c,M) = \partial^3(\rho^{-1})\partial^3(\rho) = \Id.\]
Thus $M^{\rho^{-1}}$ is a system of products for $c$.

Conversely, suppose that $M$ is a system of products $M$ for $c$. Then by definition $T(c,M)$ vanishes in $Z^3(G,\Inv(\cZ(\cC)))$, and so $o_3(c)$ is trivial in $H^3(G,\Inv(\cZ(\cC)))$.
\end{proof}

Given that we have found exactly when there exists a system of products for $c$, our next goal is to classify all systems of products for $c$.

\begin{lem}\label{lem:tor2}
Suppose $o_3(c)$ is trivial, then systems of products for $c$ form a torsor over $H^2(G,\Inv(\cZ(\cC)))$.
\end{lem}
\begin{proof}\hspace{1em}

\noindent We begin by showing that the group  $Z^2(G,\Inv(\cZ(\cC)))$ acts on system of products for $c$.
\begin{trivlist}\leftskip=3em
\item
Let $M$ be a system of products for $c$, and $\psi \in Z^2(G,\Inv(\cZ(\cC)))$. Then $\psi$ acts on $M$ as in Definition~\ref{def:Maction}. An application of Lemma~\ref{lem:3cob} shows that 
\[  T(c,M^\psi) =\partial^3(\psi)T(c,M). \]
As $\psi$ is a 2-cocycle, $\partial^3(\psi)$ is trivial, and as $M$ is a system of products for $c$, we have that $T(c,M)$ vanishes. Thus $T(c,M^\psi)$ also vanishes, and so $M^\psi$ is also a system of products for $c$. 
\end{trivlist}

\noindent Next we show that the action of $Z^2(G,\Inv(\cZ(\cC)))$ on system of products for $c$ is transitive.
\begin{trivlist}\leftskip=3em
\item
Let $M$ and $M'$ be two system of products for $c$. As explained in earlier discussion, each $M'_{g,h} = L_{g,h}M_{g,h}$ for some $L_{g,h} \in \Inv(\cZ(\cC))$. Lemma~\ref{lem:3cob} then shows that 
\[     T(c,M') =  \partial^3(L)T(c,M). \]
As $M$ and $M'$ are both system of products for $c$, we have that both $T(c,M)$ and $T(c,M')$ are trivial. Hence $\partial^3(L)$ must also be trivial, and thus $L \in Z^2(G, \Inv(\cZ(\cC)))$. Therefore $M'  = M^L$ with $L \in Z^2(G, \Inv(\cZ(\cC)))$.
\end{trivlist}

\noindent Now we show that the action on $Z^2(G,\Inv(\cZ(\cC)))$ on system of products for $c$ decends to a well-defined action of $H^2(G,\Inv(\cZ(\cC)))$.
\begin{trivlist}\leftskip=3em
\item
Let $M$ a system of products for $c$, and $\psi \in B^2(G,\Inv(\cZ(\cC)))$. Then $\psi_{g,h}\rho_{g,h} = \rho_g\rho_h^{c_g}$ for some $\rho \in C^1(G,\Inv(\cZ(\cC)))$. We define an equivalence of extensions 
\[  F_\rho := \bigoplus \rho_g \Id_{C_g}\]
 between the quasi-monoidal extensions of $\cC$ associated to the tuples $(c,M^\psi)$ and $(c,M)$. We compute
\[ \raisebox{-.5\height}{ \includegraphics[scale = .7]{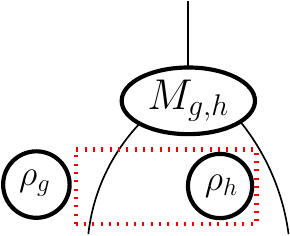}} \eq \raisebox{-.5\height}{ \includegraphics[scale = .7]{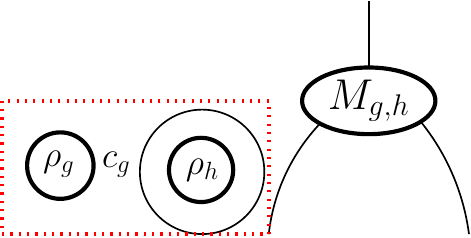}} \eq \raisebox{-.5\height}{ \includegraphics[scale = .7]{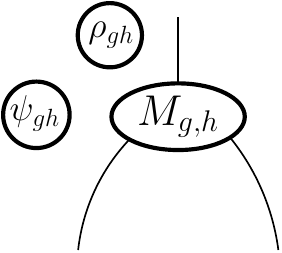}},\]
with the first equality following from Lemma~\ref{lem:lact}, and the second equality following as $\psi$ is a coboundary. This equation shows that 
\[ F_\rho(X_g) \otimes F_\rho(Y_h) \cong F_\rho(X_g\otimes Y_h),\] 
hence $F_\rho$ is a quasi-monoidal graded equivalence. Thus the quasi-monoidal extensions of $\cC$ associated to the tuples $(c,M^\psi)$ and $(c,M)$ are equivalent as quasi-monoidal extensions of $\cC$, and so $M$ and $M^\psi$ are equivalent as systems of products for $c$.
\end{trivlist}

\noindent Finally we show that the action of $H^2(G,\Inv(\cZ(\cC)))$ on systems of products for $c$ is free.

\begin{trivlist}\leftskip=3em
\item
Let $M$ be a system of products for $c$, and $\psi \in Z^2(G, \Inv(\cZ(\cC)))$. Suppose there exists $\cF$ a quasi-monoidal equivalence of extensions between the two quasi-monoidal extensions of $\cC$ coming from the tuples $(c,M^\psi)$ and $(c,M)$. Then $\cF$ restricted to $c_g$ gives a bimodule auto-equivalence of $c_g$, and thus by Lemma~\ref{lem:anue}, we have $\cF|_{c_g} = \rho_g \Id_{c_g}$ for $\rho_g \in \Inv(\cZ(\cC))$. As $\cF$ is quasi-monoidal, we have that 
\[ \raisebox{-.5\height}{ \includegraphics[scale = .7]{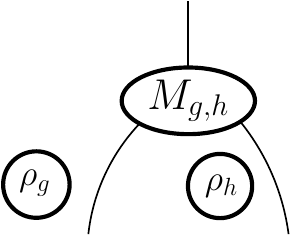}} \eq \raisebox{-.5\height}{ \includegraphics[scale = .7]{quasiequivstep3}}.\]
An application of Lemma~\ref{lem:lact} shows that $\psi_{g,h}\rho_{g,h} = \rho_g\rho_h^{c_g}$. Thus $\psi$ is a coboundary.
\end{trivlist}

Putting everything together we have that systems of products for $c$ form a torsor over $H^2(G,\Inv(\cZ(\cC)))$.
\end{proof}

Now that we have classified pairs $(c,M)$ giving rise to quasi-monoidal categories, we wish to classify all associators for these quasi-monoidal categories. As $M$ is a system of products for $c$, by definition there exists a collection of bimodule natural isomorphisms: 
\[ A :=\{ A_{f,g,h}: M_{fg,h}\circ(M_{f,g} \boxtimes \Id_{c_h}) \cong M_{f,gh}\circ(\Id_{c_f} \boxtimes M_{g,h}) \} = \left \{ \raisebox{-.5\height}{ \includegraphics[scale = .6]{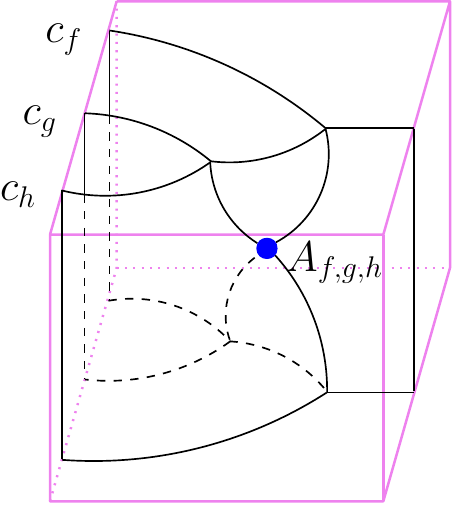}} \right \} . \]
Applying Equation~\eqref{eq:funball} to the collection $A$ we get $\cC$-balanced isomorphisms:
\[a_{f,g,h}: m_{fg,h}\circ(m_{f,g} \times \id_{c_h}) \cong m_{f,gh}\circ(\id_{c_f} \times m_{g,h}). \]
We define an associator on the quasi-monoidal category associated to the tuple $(c,M)$ by:
\[ \alpha_{X_f,Y_g,Z_h} :=( a_{f,g,h})_{X_f,Y_g,Z_h}.\]

For this associator to satisfying the pentagon axiom we exactly need the equality of $\cC$-balanced isomorphisms:
\begin{align*}   & [\id_{m_{f,ghk}}\circ(\id_{\Id_{c_f}} \times a_{g,h,k})]  [a_{f,gh,k}\circ ( \id_{\Id_{c_f}} \times \id_{m_{g,h}}\times \id_{\Id_{c_k}} )] [\id_{m_{fgh,k}}\circ( a_{f,g,h} \times \id_{\Id_{c_k}}) ] = \\ &[a_{f,g,hk}\circ ( \id_{\Id_{c_f}} \times \id_{\Id_{c_g}} \times \id_{m_{h,k}})] [a_{fg,h,k}\circ (\id_{m_{f,g}} \times \id_{\Id_{c_h}} \times \id_{\Id_{c_k}})].  \end{align*}

Via Equation~\eqref{eq:funball} and a rearrangement, the above equality is equivalent to showing that the following bimodule functor natural automorphism:

\begin{align*}   v(c,M,A)_{f,g,h,k} :=& [A_{fg,h,k}^{-1}\circ (\id_{M_{f,g}} \boxtimes \id_{\Id_{c_h}} \boxtimes \id_{\Id_{c_k}})][   A_{f,g,hk}^{-1}\circ (\id_{\Id_{c_f}} \boxtimes \id_{\Id_{c_g}} \boxtimes \id_{M_{h,k}})] \\
& [\id_{M_{f,ghk}}\circ(\id_{\Id_{c_f}} \boxtimes A_{g,h,k})]  [A_{f,gh,k}\circ ( \id_{\Id_{c_f}} \boxtimes \id_{M_{g,h}}\boxtimes \id_{\Id_{c_k}} )] [\id_{M_{fgh,k}}\circ( A_{f,g,h} \boxtimes \id_{\Id_{c_k}}) ] \end{align*}
is equal to the identity of the functor 
\[  H_{f,g,h,k} := M_{fgh,k} \circ ( M_{fg,h} \boxtimes \Id_{c_k}) \circ (M_{f,g} \boxtimes \Id_{c_h} \boxtimes \Id_{c_k}).\]
 A graphical description of $v(c,M,A)_{f,g,h,k}$ is given in Figure~\ref{fig:v4}. 
 
 We have that the graded extension $\cD$ is monoidal if and only if $v(c,M,A)_{f,g,h,k}$ is trivial for all $f,g,h,k \in G$. We call a collection of isomorphisms $A$ such that $v(c,M,A)$ vanishes, a \textit{system of associators for $(c,M)$}. We consider system of associators for $(c,M)$ up to equivalence of extensions of the corresponding graded extensions of $\cC$.

\begin{figure}
\caption{Graphical interpretation of $v(c,M,A)_{f,g,h,k}$}\label{fig:v4}
\begin{tikzcd}[row sep = 5cm,column sep = 5cm]
 \raisebox{-.5\height}{ \includegraphics[scale = .5]{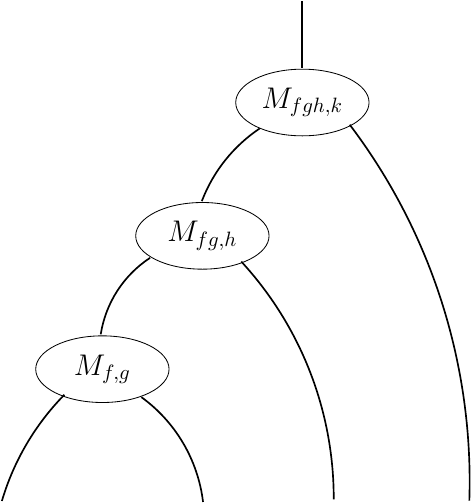}} \arrow{r}[swap]{\raisebox{-.5\height}{ \includegraphics[scale = .55]{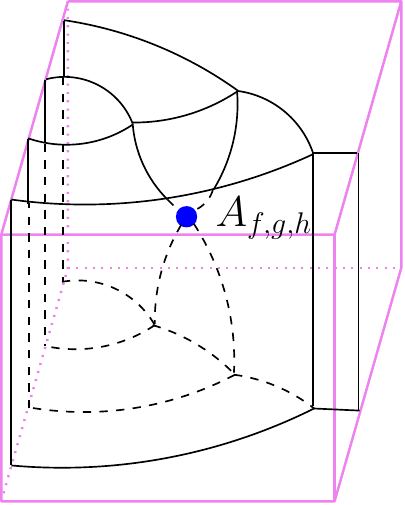}}} & \raisebox{-.5\height}{ \includegraphics[scale = .55]{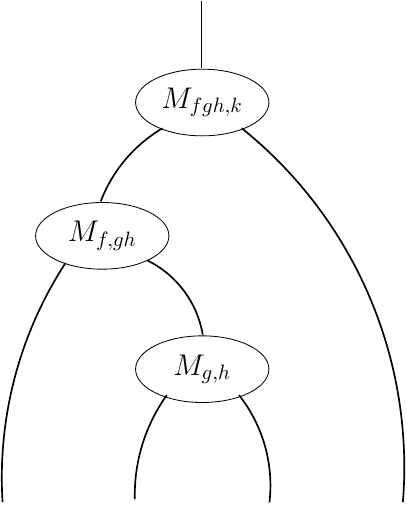}}  \arrow[d, "\raisebox{-.5\height}{ \includegraphics[scale = .5]{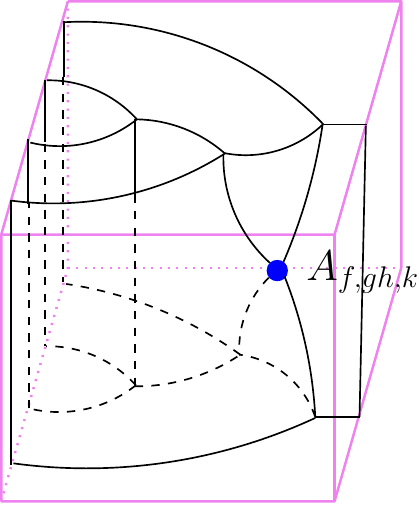}}"] \\
  											       & \raisebox{-.5\height}{ \includegraphics[scale = .5]{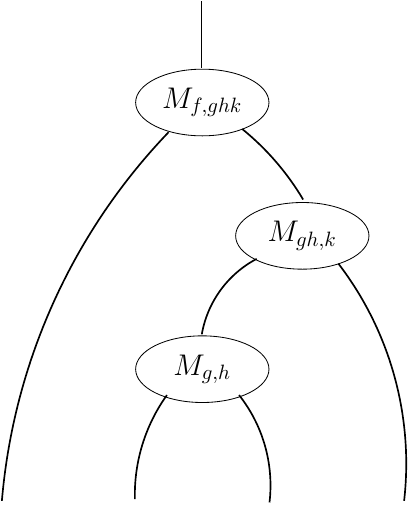}}  \arrow[d, "\raisebox{-.5\height}{ \includegraphics[scale = .55]{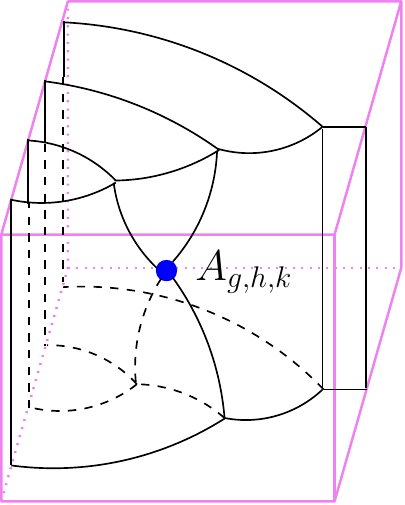}}"] \\
 \raisebox{-.5\height}{ \includegraphics[scale = .5]{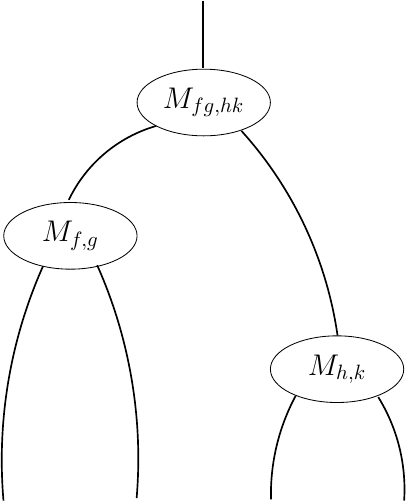}} \arrow[uu, "\raisebox{-.5\height}{ \includegraphics[scale = .55]{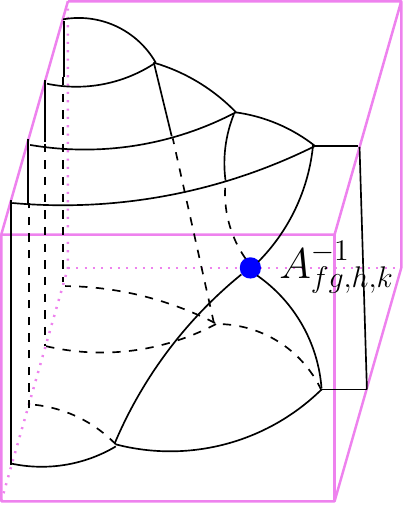}}"]   & \raisebox{-.5\height}{ \includegraphics[scale = .5]{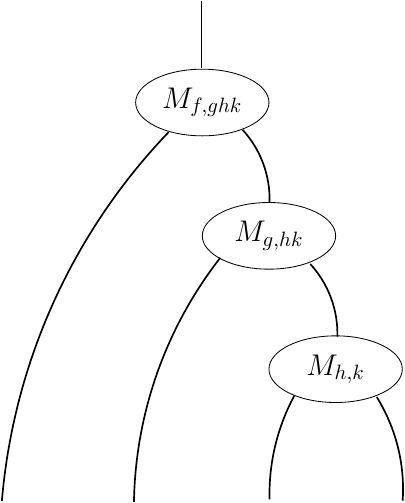}} \arrow{l}[swap]{\raisebox{-.5\height}{ \includegraphics[scale = .55]{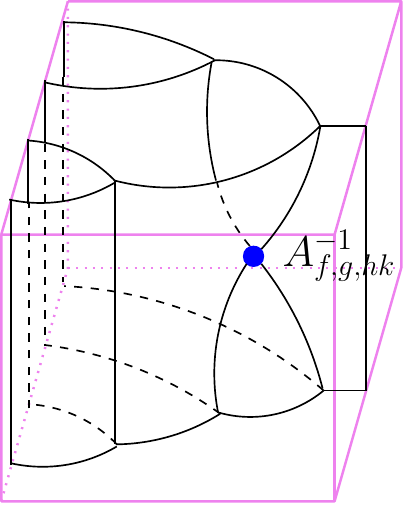}}} \\
\end{tikzcd}

\end{figure}

As $v(c,M,A)_{f,g,h,k}$ is a bimodule functor natural automorphism, we can identify it with a non-zero complex number. Thus $v(c,M,A)$ is a function $G \times G \times G\times G \to \mathbb{C}^\times$. We can also identify $v(c,M,A)$ with a natural automorphism of $\Id_\cC$ in the following way:

\begin{align*} V(c,M,A)_{f,g,h,k} :=&     \operatorname{r}_{\operatorname{R}_{c_{fghk} }}  \cdot [ \id_{\operatorname{R}_{c_{fghk}}} \circ (  \operatorname{r}_{H_{f,g,h,k}} \boxtimes \id_{\Id_{c_{fghk}^\text{op}}}) \circ \id_{\operatorname{R}^{-1}_{c_{fghk}}}    ]  \\
 &\cdot [ \id_{\operatorname{R}_{c_{fghk}}} \circ ( \id_{H^{-1}_{f,g,h,k}} \boxtimes \id_{\Id_{c_{fghk}^{op}}}) \circ (  v(c,M,A)_{f,g,h,k} \boxtimes \id_{\Id_{c_{fghk}^\text{op}}}) \circ \id_{\operatorname{R}^{-1}_{c_{fghk}}}    ] \\ 
&\cdot [ \id_{\operatorname{R}_{c_{fghk}}} \circ (  \operatorname{r}^{-1}_{H_{f,g,h,k}} \boxtimes \id_{\Id_{c_{fghk}^\text{op}}}) \circ \id_{\operatorname{R}^{-1}_{c_{fghk}}}    ]\cdot  \operatorname{r}^{-1}_{\operatorname{R}_{c_{fghk} }}.
\end{align*}
Graphically we can think of $V(c,M,A)_{f,g,h,k}$ as a compactification of the foam describing $v(c,M,A)_{f,g,h,k}$ in Figure~\ref{fig:v4}. 

As with $v(c,M,A)$, we can regard $V(c,M,A)$ as a function $G\times G \times G \times G\to \mathbb{C}^\times$. In fact when considered as functions $G\times G \times G \times G\to \mathbb{C}^\times$, we have that $v(c,M,A)$ and $V(c,M,A)$ are equal. The reason we have two representatives of the same function is because each that each representative is easier to work with in certain situations, and we need both to prove all the properties of the function $G\times G \times G \times G\to \mathbb{C}^\times$ required for this section. For example, we show in the following proof that $V(c,M,A)$ is a 4-cocycle, which then implies that the corresponding function $G\times G \times G \times G\to \mathbb{C}^\times$ is also a 4-cocycle. We remark that \cite{MR2677836} mentions an easy proof showing that $v(c,M,A)$ is a 4-cocycle. We were unable to come up with an easy proof, and instead have to use the 3-dimensional graphical calculus for $\BBBrPic(\cC)$ to prove this fact.

\begin{lem}
We have
\[ V(c,M,A)_{g,h,k,i}V(c,M,A)_{f,gh,k,i}V(c,M,A)_{f,g,h,ki} = V(c,M,A)_{fg,h,k,i}V(c,M,A)_{f,g,hk,i}V(c,M,A)_{f,g,h,k}\]
\end{lem}
\begin{proof}
Showing that \[ V(c,M,A)_{g,h,k,i}V(c,M,A)_{f,gh,k,i}V(c,M,A)_{f,g,h,ki} = V(c,M,A)_{fg,h,k,i}V(c,M,A)_{f,g,hk,i}V(c,M,A)_{f,g,h,k}\] follows somewhat similar to the calculation showing $T(c,M)$ is a 3-cocycle, except now the proof uses the 3-dimensional graphical calculus, so we need to recouple sheets instead of strings.  This computation is in fact easier in one sense, as each $V(c,M,A)_{-,-,-,-}$ can be contracted to a scaler and moved anywhere in the diagram through linearity.

The 4-cocycle calculation now follows by merging both \[V(c,M,A)_{g,h,k,i}V(c,M,A)_{f,gh,k,i}V(c,M,A)_{f,g,h,ki}\] and \[V(c,M,A)_{fg,h,k,i}V(c,M,A)_{f,g,hk,i}V(c,M,A)_{f,g,h,k}\] into single components, and rearranging to show they are equal. We omit the exact details of this calculation as it is fairly straightforward, but tedious to draw.
\end{proof}

Suppose we have two different collections of bimodule isomorphisms 
\[A_{f,g,h}, A'_{f,g,h}: M_{fg,h}\circ(M_{f,g} \boxtimes \id_{c_h}) \cong M_{f,gh}\circ(\id_{c_f} \boxtimes M_{g,h}). \]
Then each $A'_{f,g,h}$ differs from $A_{f,g,h}$ by a natural automorphism of $c_{fgh}$, and thus a non-zero complex number. Therefore $A'_{f,g,h} = \lambda_{f,g,h}A_{f,g,h}$ for some $\lambda_{f,g,h} \in \mathbb{C}^\times$. Using linearity it is easy to see that 
\[ v(c,M,A')_{f,g,h,k} = \lambda_{g,h,k}\lambda_{f,gh,k}\lambda_{f,g,h}\lambda^{-1}_{fg,h,k}\lambda^{-1}_{f,g,hk}v(c,M,A)_{f,g,h,k} = \partial^4(\lambda) v(c,M,A)_{f,g,h,k}. \]
Hence $v(c,M,A')$ differs from $v(c,M,A)$ by a coboundary in $Z^4(G,\mathbb{C}^\times)$. Thus the cohomology class of $v(c,M,A)$ is independent of the choice of collection $A$. This implies that \[o_4(c,M):= V(c,M,-)\] is an element $H^4(G,\mathbb{C}^\times)$ that depends only on $c$ and $M$.

From the element $o_4(c,M)$ we can deduce a necessary and sufficient condition for there to exist a system of associators for $(c,M)$, and hence for the category $\cD$ to be monoidal.

\begin{lem}\label{lem:nonempo4}
There exists a system of associators for $(c,M)$ if and only if $o_4(c,M)$ is the trivial element in $H^4(G,\mathbb{C}^\times)$.
\end{lem}

\begin{proof}
Suppose $o_4(c,M)$ is trivial. As $M$ is a system of products for $c$, there exists some collection of bimodule isomorphisms 
\[A = \{ A_{f,g,h}: M_{fg,h}\circ(M_{f,g} \boxtimes \id_{c_h}) \cong M_{f,gh}\circ(\id_{c_f} \boxtimes M_{g,h})\}.\] 
As $o_4(c,M)$ is trivial we know that $v(c,M,A)= \partial^4(\rho)$ for some $\rho \in C^3(G,\mathbb{C}^\times)$. We define a new collection of bimodule isomorphisms \[A^{\rho^{-1}}_{f,g,h} := \rho_{f,g,h}^{-1}A_{f,g,h}.\] Using linearity it is easy to see that 
\[ v(c,M,A^{\rho^{-1}})_{f,g,h,k} = \rho_{g,h,k}^{-1}\rho_{f,gh,k}^{-1}\rho_{f,g,h}^{-1}\rho_{fg,h,k}\rho_{f,g,hk} v(c,M,A)_{f,g,h,k} = \partial^4(\rho)^{-1}_{f,g,h,k} \partial^4(\rho)_{f,g,h,k} =1. \]
Thus $A^{\rho^{-1}}$ is a system of associators for $(c,M)$.

Conversely, suppose there exists a system of associators for $(c,M)$. Then by definition $v(c,M,A)$ is trivial in $Z^4(G,\mathbb{C}^\times)$, and so $o_4(c,M)$ is trivial in $H^4(G,\mathbb{C}^\times)$.
\end{proof}

Given that we have determined when there exists a system of associators for the tuple $(c,M)$, we now wish to classify all systems of associators for $(c,M)$.

\begin{lem}\label{lem:tor3}
Suppose $o_4(c,M)$ is trivial, then systems of associators for $(c,M)$ form a torsor over $H^3(G,\mathbb{C}^\times)$.
\end{lem}
\begin{proof} \hspace{1em}

\noindent We first show that the group $Z^3(G,\mathbb{C}^\times)$ acts on systems of associators for $(c,M)$.
\begin{trivlist}\leftskip=3em
\item Let $A$ be such a system of associators for $(c,M)$, and $\psi \in Z^3(G,\mathbb{C}^\times)$. We define a new collection of bimodule isomorphisms by:
\begin{equation}\label{eq:act} A^\psi_{f,g,h} := \psi_{f,g,h}A_{f,g,h}. \end{equation}

We compute using linearity that 
\begin{equation}\label{eq:psiA} V(c,M,A^\psi)_{f,g,h,k}  = \psi_{g,h,k}\psi_{f,gh,k}\psi_{f,g,h}\psi_{fg,h,k}^{-1}\psi_{f,g,hk}^{-1}V(c,M,A) =  \partial^4(\psi)_{f,g,h,k}V(c,M,A) = V(c,M,A).\end{equation}
Thus $ V(c,M,A^\psi)$ is trivial, and so $A^\psi$ is a system of associators for $(c,M)$.
\end{trivlist}

\noindent Next we show that the action of $Z^3(G,\mathbb{C}^\times)$ on systems of associators for $(c,M)$ is transitive.
\begin{trivlist}\leftskip=3em
\item Let $A$ and $A'$ be two systems of associators for $(c,M)$. We then have that each $A'_{f,g,h} = \psi_{f,g,h}A_{f,g,h}$ for some $\psi_{f,g,h}\in \mathbb{C}^\times$. Hence $A' = A^\psi$ in the notation of Equation~\eqref{eq:act}. The Equation~\eqref{eq:psiA}, along with the fact that $V(c,M,A)$ and $V(c,M,A')$ vanish, implies that $\psi$ is an element of $Z^3(G,\mathbb{C}^\times)$. Thus $A' = A^\psi$ with $\psi \in Z^3(G,\mathbb{C}^\times)$.
\end{trivlist}

\noindent Now we show the action $Z^3(G,\mathbb{C}^\times)$ on systems of associators for $(c,M)$ descends to a well defined action of $H^3(G,\mathbb{C}^\times)$.
\begin{trivlist}\leftskip=3em
\item Let $A$ be a system of associators for $(c,M)$, and $\psi$ a coboundary in $Z^3(G,\mathbb{C}^\times)$. As $\psi$ is a coboundary we have $\psi = \partial^3(\rho)$ for some $\rho \in C^2(G,\mathbb{C}^\times)$. We define an equivalence $F_\rho$ of $G$-graded extensions between the two graded extensions of $\cC$ associated to the triples $(c,M,A^\rho)$ and $(c,M,A)$. As an abelian functor $F_\rho$ is the identity. The tensorator of $F_\rho$ is given by $\tau_{X_g,Y_h} := \rho_{g,h}\id_{X_g \otimes Y_h}$. To verify that $F_\rho$ is monoidal we compute
\begin{align*}   \alpha_{X_f,Y_g,Z_h}\tau_{fg,h}[\tau_{X_f,Y_g} \otimes \id_{Z_h}] =& (a_{f,g,h})_{X_f,Y_g,Z_h}[\rho_{fg,h}\rho_{f,g}\id_{X_f\otimes Y_g \otimes Z_h}] \\
															   =& (a_{f,g,h})_{X_f,Y_g,Z_h}[\psi_{f,g,h}\rho_{f,gh}\rho_{g,h}\id_{X_f\otimes Y_g \otimes Z_h}]\\
 															    =&[\rho_{g,h}\rho_{f,gh}\id_{X_f\otimes Y_g \otimes Z_h}]\psi_{f,g,h} (a_{f,g,h})_{X_f,Y_g,Z_h}\\
         														   =& [\id_{X_f} \otimes \tau_{g,h}]\tau_{f,gh}\alpha^\psi_{X_f,Y_g,Z_h}. \end{align*}
As there is an equivalence of $G$-graded extensions between the categories associated to the triples $(c,M,A)$ and $(c,M,A^\psi)$, we have that the systems of associators $A$ and $A^\psi$ are equivalent.
\end{trivlist}

\noindent Finally we show that the action of $H^3(G,\mathbb{C}^\times)$  on systems of associators for $(c,M)$ is free.
\begin{trivlist}\leftskip=3em
\item Let $A$ a system of associators for $(c,M)$ and $\psi \in Z^3(G,\mathbb{C}^\times)$. Suppose that there exists $(\cF,\tau)$ an equivalence of $G$-graded extensions between the graded extensions of $\cC$ corresponding to the triples $(c,M,A)$ and $(c,M,A^\psi)$. The restriction of the tensorator $\tau$ to $c_g \times c_h$ gives a bimodule automorphism of the functor $M_{g,h}$, and thus a non-zero complex number. As $(\cF,\tau)$ is monoidal we have 
\[ \tau |_{c_f,c_{gh}}\tau |_{c_g,c_h} \psi_{f,g,h}\alpha_{f,g,h} =  \alpha_{f,g,h} \tau |_{c_{fg},c_{h}}\tau |_{c_f,c_g}.   \]
Thus $\psi = d^3(\tau |_{-,-})$, and so $\psi$ is a coboundary in $Z^3(G,\mathbb{C}^\times)$. 
\end{trivlist}

Putting everything together we have that systems of associators for $(c,M)$ form a torsor over $H^3(G,\mathbb{C}^\times)$.
\end{proof}

We can now prove the classification of $G$-graded extensions of $\cC$, up to equivalence of extensions. Given a $G$-graded extension of $\cC$, we can extract a group homomorphism $c: G \to \BrPic(\cC)$, a system of products $M$ for $c$, and a system of associators $A$ for $(c,M)$. Conversely given this triple of data, we can reconstruct the $G$-extension as described earlier in this Section. Thus to classify $G$-graded extensions of $\cC$, we have to classify triples $(c,M,A)$.

Let $c$ be any group homomorphism $G \to \BrPic(B)$, then Lemma~\ref{lem:nonempo3} says we can pick a system of products for $c$ if and only if $o_3(c)$ vanishes. As $o_3(c)$ vanishes, Lemma~\ref{lem:tor2} says that systems of products for $c$ form a torsor over $H^2(G, \Inv(\cZ(\cC)))$. Let $M$ be a system of products for $c$, then Lemma~\ref{lem:nonempo4} says we can pick a system of associators for $(c,M)$ if and only if $o_4(c,M)$ vanishes. As $o_4(c,M)$ vanishes, Lemma~\ref{lem:tor3} says that systems of associators for $(c,M)$ form a torsor over $H^3(G, \mathbb{C}^\times)$. 

This completes the classification of $G$-graded extensions of $\cC$, up to equivalence of extensions. 

\section{Twisted bimodule functors}\label{sec:twist}

In this section we introduce the abstract nonsense for twisted bimodule functors, and natural transformations between them. If the reader desires motivation for such abstract nonsense then we encourage them to skip to the next section, where the motivating example can be found.
\begin{defn}
Let $\cM$ and $\cN$ be invertible bimodule categories over a fusion category $\cC$, and $(\cF,\tau)$ a monoidal auto-equivalence of $\cC$. An $\cF$-twisted bimodule functor $\cc{H}: \cM \to \cN$ is an abelian functor, along with natural isomorphisms $L_{X,M}:  \cF(X) \triangleright \cH(M) \to \cH(X \triangleright M) $ and $R_{M,X}: \cH(M) \triangleleft \cF(X) \to \cH(M \triangleleft X)$ satisfying:

$$\begin{CD}\label{dia:com1}
\cF(X) \triangleright (\cF(Y) \triangleright \cH(M))@> l^\cN_{F(X),F(Y),H(M)}>> (\cF(X)\otimes \cF(Y) )\triangleright \cH(M)\\
@V \id_{\cF(X)}\triangleright L_{Y,M} VV @VV \tau_{X,Y}\otimes \id_{\cH(M)} V\\
\cF(X) \triangleright \cH(Y\triangleright M)@.\cF(X\otimes Y) \triangleright \cH(M) \\
@V L_{X,Y\triangleright M} VV @VV  L_{X\otimes Y,M} V\\
\cH(X\triangleright Y \triangleright m)@> \cH(l^\cM_{X,Y,M}) >> \cH((X\otimes Y)\triangleright M)
\end{CD}$$

$$\begin{CD}
(\cH(M)\triangleleft \cF(X)) \triangleleft \cF(Y)@> r^\cN_{\cH(M),\cF(X),\cF(Y)}>> \cH(M) \triangleleft (\cF(X)\otimes \cF(Y))\\
@V R_{M,X}\triangleleft \id_{\cF(Y)} VV @VV \id_{\cH(M)}\otimes \tau_{X,Y} V\\
\cH(M \triangleleft X) \triangleleft \cF(Y)@.\cH(M)\triangleleft \cF(X\otimes Y) \\
@V R_{M\otimes X, Y} VV @VV  R_{M,X\otimes Y} V\\
\cH(M \triangleleft X \triangleleft Y)@> H(r^\cM_{M,X,Y}) >> \cH(M \triangleleft (X\otimes Y))
\end{CD}$$ and

$$\begin{CD}
( \cF(X) \triangleright(\cH(M) )\triangleleft \cF(Y) @> c^\cN_{\cF(X),\cH(M),\cF(Y)}>> \cF(X) \triangleright (\cH(M) \triangleleft \cF(Y))\\
@V L_{X,M} \triangleleft \id_{\cF(Y)} VV @VV \id_{\cF(X)} \triangleright R_{M,Y} V\\
\cH(X\triangleright M) \triangleleft \cF(Y) @.  \cF(X) \triangleright \cH(M \triangleleft Y)\\
@V R_{X\triangleright M, Y} VV @VV L_{X,M\triangleleft Y} V\\
\cH( (X \triangleright M)\triangleleft Y)@>  \cH(c^\cM_{X,M,Y})>> \cH( X \triangleright (M \triangleleft Y))
\end{CD}$$
\end{defn}
While we don't expect this definition to find any use outside of this paper, there is a slight generalization which could prove useful. If we relax the condition that $\cF$ be an auto-equivalence of $\cC$, to just being an arbitary monoidal functor between two fusion categories, then we would be able to consider functors between modules over different categories. For example $E_7$ is a module over both the $A_{17}$ and $D_{10}$ fusion categories. Up to twisting by the free module functor $A_{17} \to D_{10}$ these two modules are equivalent. We neglect to define this generalization as it is far beyond the scope of this paper.

\begin{remark}
We will surpress the structure morphisms $R$ and $L$ when declaring $\cF$-twisted bimodule functors from now on. Given an $\cF$-twisted bimodule functor $\cH$ it will be implicitly assumed that $R^\cH$ and $L^\cH$ are the $\cF$-twisted bimodule functor structure maps for $\cH$.
\end{remark}

Given two $\cF$-twisted bimodule functors we can define the tensor product in an almost analogous way to the untwisted case, see \cite{MR2678824}.

\begin{remark}
Given $\cN_1$ and $\cN_2$ two $\cC$-bimodules we write $B_{\cN_1, \cN_2}$ for the canonical $\cC$-balanced bimodule functor $\cN_1\times \cN_2 \to \cN_1\boxtimes \cN_2$. We write $b^{B_{\cN_1, \cN_2}}$ for the balancing isomorphisms of this functor.
\end{remark}
\begin{defn}\label{def:twten}
Let $\cH_1 : \cM_1 \to \cN_1$ and $\cH_2 : \cM_2 \to \cN_2$ be $\cF$-twisted bimodule functors. There is a functor $\cM_1 \times \cM_2 \to \cN_1 \boxtimes \cN_2$ given by
\begin{equation*}
B_{\cN_1,\cN_2} \circ (\cH_1 \times \cH_2).
\end{equation*}
This functor is $\cC$-balanced by the natural isomorphism
\begin{align*}
& B_{\cN_1,\cN_2}(\id_{H_1(m_1)} \times {L^{\cH_2}_{X,m_2}})    \cdot  b^{B_{\cN_1,\cN_2}}_{\cH_1(m_1),\cF(X),\cH_2(m_2)} \cdot B_{\cN_1,\cN_2}({R^{\cH_1}_{m_1,X}}^{-1}\times \id_{\cH_2(m_2)}) \\
 & : B_{\cN_1,\cN_2}( \cH_1( m_1 \triangleleft X) , \cH_2(m_2 )   ) \to B_{\cN_1,\cN_2}( \cH_1( m_1 ) , \cH_2(X \triangleright m_2 )   ).
\end{align*}
Thus via Equation~\eqref{eq:funball}, the above functor induces a functor 
\[\cH_1 \boxtimes \cH_2 : \cM_1\boxtimes \cM_2 \to \cN_1 \boxtimes \cN_2.\]
We define the relative tensor product of the functors $\cH_1$ and $\cH_2$ as the above functor $\cH_1 \boxtimes \cH_2$.
\end{defn}

One can also compose twisted bimodule functors.
\begin{defn}
Let $\cH_1: \cN\to \cM$, and $\cH_2: \cM\to \cP$ be $\cF_1$ and $\cF_2$-twisted bimodule functors respectively. We can compose $\cH_1$ and $\cH_2$ as follows:
\begin{align*}
\cH_2 \circ \cH_1 (m) :=& \cH_2(\cH_1(m)), \\
L^{\cH_2\circ \cH_1}_{X,m} :=& \cH_2( L^{\cH_1}_{X,m}) \cdot L^{\cH_2}_{\cF_1(X), \cH_1(m)}, \\
R^{\cH_2\circ \cH_1}_{m,X} :=& \cH_2( R^{\cH_1}_{m,X}) \cdot R^{\cH_2}_{\cH_1(m),\cF_1(X)},
\end{align*}
to obtain a $(\cF_2\circ \cF_1)$-twisted bimodule functor.
\end{defn}

We now define natural transformations between $\cF$-twisted bimodule functors. Again this is a straightforward generalisation of the untwisted case.
\begin{defn}
Let $\cH_1$ and $\cH_2$ be $\cF$-twisted bimodule functors with the same source and target. An $\cF$-twisted natural transformation $\mu: \cH_1 \to \cH_2$ is a natural transformation satisfying:
$$\begin{CD}
\cF(X) \triangleright  \cH_1(m)@>\id_{\cF(X)} \triangleright \mu_{m}>> \cF(X) \triangleright  \cH_2(m)\\
@V L^{\cH_1}_{X,m} VV @VV L^{\cH_2}_{X,m} V\\
\cH_1(X\triangleright m)@> \mu_{X\triangleright m} >> \cH_2(X\triangleright  m)
\end{CD} \text{ and } 
\begin{CD}
\cH_1(m) \triangleleft \cF(X)@>  \mu_m \triangleleft \id_{\cF(X)}>> \cH_2(m) \triangleleft \cF(X)\\
@V R^{\cH_1}_{m,X} VV @VV R^{\cH_2}_{m,X} V\\
\cH_1(m\triangleleft X)@> \mu_{m\triangleleft X} >> \cH_2(m \triangleleft X)
\end{CD}$$
\end{defn}
These $\cF$-twisted natural transformations have natural vertical and horizontal compositions, which we denote $\cdot$ and $\circ$ respectively. The details of these operations are straightforward so we omit the details. The relative tensor product of two $\cF$-twisted natural transformations is less clear, so we include the details.

\begin{defn}
Let $\cH_1,\cH_1' : \cM_1 \to \cN_1$, and $\cH_2,\cH_2': \cM_2 \to \cN_2$ be $\cF$-twisted bimodule functors, and $\mu_1: \cH_1\to \cH_1'$ and $\mu_2: \cH_2 \to \cH_2'$ be $\cF$-twisted natural transformations. Consider the natural isomorphism 
\begin{equation*}
\id_{B_{\cH_1',\cH_2'}}\circ(\mu_1 \times \mu_2): B_{\cN_1,\cN_2}\circ (\cH_1 \times \cH_2) \to  B_{\cN_1,\cN_2}\circ (\cH_1' \times \cH_2').
\end{equation*}
It can be checked that this natural isomorphism is $\cC$-balanced, and thus via Equation~\eqref{eq:funball} induces a natural transformation
\begin{equation*}
\mu_1\boxtimes \mu_2 : \cH_1\boxtimes \cH_2 \to \cH_1'\boxtimes \cH_2'.
\end{equation*}
We define the relative tensor product of $\mu_1$ and $\mu_2$ as the above $\mu_1\boxtimes \mu_2$.
\end{defn}

If we restrict $\cF$ to be the identity functor on $\cC$ then we can assemble the above information into the $3$-category $\underline{\underline{\BrPic}}(\cC)$. When $\cF$ is non-trivial one wonders what higher categorical structure the above information fits into. Unfortunately we haven't been able to provide a satisfactory answer to this question. The best we can guess is some sort of higher bimodule category over $\underline{\underline{\BrPic}}(\cC)$. We haven't put so much thought into this question due to the following Theorem which shows that all the above information for twisted bimodule functors is actually contained in $\underline{\underline{\BrPic}}(\cC)$, though in a not so obvious manner. 
 
 \begin{remark}
 We remind the reader that the definition of the $\cC$-bimodules $\cC_\cF$ and $_\cF\cC$ can be found in Definition~\ref{def:righttwist}.
 \end{remark}

\begin{thm}\label{lem:untwisting}
Let $\cC$ a fusion category, and $\cF$ a monoidal auto-equivalence of $\cC$. Then
\begin{enumerate}
\item\label{thm:part1} $\cF$-twisted bimodule functors $\cH: \cM \to \cN$ are in bijection (up to natural isomorphism) with untwisted bimodule functors 
\begin{equation*}
\widehat{\cH} : \cC_\cF\boxtimes \cM \boxtimes _\cF\cC  \to \cN.
\end{equation*}

\item\label{thm:part2} Suppose $\cH_1:\cM_1 \to \cN_1$ and $\cH_2:\cM_2 \to \cN_2$ are $\cF$-twisted bimodule functors, then there exists a natural isomorphism of bimodule functors:
\begin{equation*}
\Omega^\boxtimes_{\cH_1,\cH_2}: \widehat{\cH_1\boxtimes \cH_2}\circ (\Id_{\cC_\cF} \boxtimes \Id_{\cM_1} \boxtimes \operatorname{R}_{_\cF\cC} \boxtimes \Id_{\cM_2} \boxtimes \Id_{_\cF\cC}) \to ( \widehat{\cH_1} \boxtimes \widehat{\cH_2}),
\end{equation*}
where $\operatorname{R}_{_\cF\cC}$ is a specific choice of bimodule equivalence $_\cF\cC \boxtimes \cC_\cF \to \cC$.

\item\label{thm:part3} Suppose $\cH_1:\cM \to \cN$ is an $\cF_1$-twisted bimodule functor and $\cH_2:\cN\to \cP$ is an $\cF_2$-twisted bimodule functor, then there exists a natural isomorphism of bimodule functors:
\begin{equation*}
\Omega^\circ_{\cH_1,\cH_2}:\widehat{\cH_2\circ \cH_1}\circ(S_{\cF_2,\cF_1} \boxtimes \Id_{\cM} \boxtimes S^*_{\cF_1,\cF_2}) \simeq  \widehat{\cH_2} \circ (\Id_{\cC_{\cF_2}} \boxtimes \widehat{\cH_1} \boxtimes \Id_{_{\cF_2}\cC}),
\end{equation*}
where $S_{\cF_2,\cF_1}: \cC_{\cF_2}\boxtimes \cC_{\cF_1} \to \cC_{\cF_2\circ \cF_1}$ and $S^*_{\cF_1,\cF_2}: _{\cF_1}\cC \boxtimes  _{\cF_2}\cC \to _{\cF_2\circ \cF_1}\cC$ are certain bimodule equivalences.

\item\label{thm:part4} Natural isomorphisms of $\cF$-twisted bimodule functors $\mu: \cH_1 \to \cH_2$ are in bijection with natural isomorphisms of untwisted bimodule functors $\widehat{\mu} : \widehat{\cH_1} \to \widehat{\cH_2}$.

\item\label{thm:part5} If $\mu_1: \cH_1 \to \cH_1'$, $\mu_2: \cH_2 \to \cH_2'$ are natural isomorphisms of $\cF$-twisted bimodule functors $\cH_1,\cH_1' : \cM_1 \to \cN_1$ and $\cH_2,\cH_2':\cM_2 \to \cN_2$ then
\begin{equation*}
\Omega^\boxtimes_{\cH_1',\cH_2'}[\widehat{\mu_1 \boxtimes \mu_2}\circ (\id_{\Id_{\cC_\cF}} \boxtimes \id_{\Id_{\cM_1}}  \boxtimes \id_{\operatorname{R_{_\cF\cC}}} \boxtimes \id_{\Id_{\cM_2}} \boxtimes \id_{\Id_{_\cF\cC}})] = [ \widehat{\mu_1} \boxtimes \widehat{\mu_2}]\Omega^\boxtimes_{\cH_1,\cH_2}.
\end{equation*}

\item\label{thm:part6} Let $\cH_1,\cH_1' : \cM\to \cN$ be $\cF_1$-twisted bimodule functors, and $\cH_2,\cH_2': \cN\to \cP$ be $\cF_2$-twisted bimodule functors. Let $\mu_1: \cH_1\to \cH_1'$ be a $\cF_1$-twisted natural transformation, and $\mu_2:\cH_2 \to \cH_2'$ be a $\cF_2$-twisted natural transformation. Then
\begin{equation*}
\Omega^\circ_{\cH_1',\cH_2'}[ \widehat{\mu_2\circ \mu_1} \circ (\id_{S_{\cF_2,\cF_1}} \boxtimes \id_{\Id_{\cM}} \boxtimes \id_{S^*_{\cF_1,\cF_2}})] = [\widehat{\mu_2} \circ ( \id_{\Id_{\cC_{\cF_2}}} \boxtimes \widehat{\mu_1} \boxtimes \id_{\Id_{_{\cF_2}\cC}})] \Omega^\circ_{\cH_1,\cH_2}.
\end{equation*}

\item\label{thm:part7} If $\mu_1 : \cH \to \cH'$ and $\mu_2: \cH' \to \cH''$ are natural isomorphisms of $\cF$-twisted bimodule functors then
\begin{equation*}
\widehat{\mu_2\mu_1} = \widehat{\mu_2}\widehat{\mu_1}.
\end{equation*}

\end{enumerate}
\end{thm}
\begin{proof}
The proof of this Theorem can be found in Appendix~\ref{app:proof}.
\end{proof}

 The strength of this Theorem is that it allows us to perform all calculations regarding $\cF$-twisted bimodule functors in the 3-category $\BBBrPic(\cC)$, where we have a nice graphical calculus to work with. To see this Theorem in action, we point the reader to the next section, and one can see that applying this Theorem is not as bad as the messy statement makes it appear.
 
 The above Theorem begs to be a functor of 3-categories, with the families of moprhisms $\Omega^\boxtimes$ and $\Omega^\circ$ playing the roles of the higher tensorators. However without an adequate description of the ``3-category of twisted bimodule functors'', it is impossible to state the theorem as such. We are forced to present the Theorem in its current messy state, which is nevertheless adequate for our uses later in this paper.
 
  As the proof of this Theorem is essentially constructing a functor of 3-categories, it is extremely technical. We have hidden the proof in Appendix~\ref{app:proof}, which we advise the reader to avoid.

\section{Data determined by a graded equivalence of graded categories}\label{sec:data}
We now begin our classification of $\cC$-based equivalences between two $G$-graded extensions of $\cC$. We start by de-constructing a $\cC$-based equivalence between extensions, and extracting four key pieces of data.

Let $\cD$ and $\cE$ be two $G$-graded extensions of a fusion category $\cC$, which, by the classification of $G$-graded extensions of $\cC$ (see \cite{MR2677836} \& Section~\ref{sec:gradedex}), corresponding to triples $(d,M^\cD,A^\cD)$ and $(e,M^\cE,A^\cE)$ respectively. Let  
\[   (\mathcal{F}, \tau): \cD \to  \cE   \]
be a $\cC$-based equivalence. 

\noindent The first piece of data we can extract from $\mathcal{F}$ is a monoidal auto-equivalence $\mathcal{F}_e : \cC \to \cC$. 

\begin{trivlist}\leftskip=3em
\item As $\mathcal{F}$ is $\cC$-based we have by definition that $\mathcal{F}(\cC) \subseteq \cC$. Therefore, as $\mathcal{F}$ is an equivalence, we have that 
\[  \mathcal{F}_e := \mathcal{F}|_\cC : \cC \to \cC\]
is a monoidal auto-equivalence of $\cC$.
\end{trivlist}

\noindent The second piece of data is an automorphism $\phi$ of the grading group $G$.
\begin{trivlist}\leftskip=3em
 \item Let $X_g$ and $Y_g$ be two simple objects in the $g$-graded component of $\cD$. Then $X_g$ and $Y_g$ are objects in the invertible bimodule category $d_g$. As $d_g$ is an indecomposible $\cC$-bimodule category, there exists $Z_e \in \cC$ such that $X_g$ is a sub-object of $Z_e \otimes Y_g$. As $\mathcal{F}$ is monoidal, we then have that $\cF(X_g)$ is a sub-object of $\cF(Z_e) \otimes \cF(Y_g)$. Using the fact that $\cF$ is a $\cC$-based equivalence, we see that $\cF(Z_e) \in \cC$, and thus the simple objects $\cF(X_g)$ and $\cF(Y_g)$ live in the same graded piece of $\cE$. Thus the functor $\cF$ induces a well-defined map on the grading group $G$. As $\cF$ is monoidal equivalence, this map is in fact an automorphism of $G$. We call this automorphism $\phi^{-1}$, and thus $\phi$ is the automorphism of $G$ satisfying
\[    \cF(  d_{\phi(g)} ) \subseteq  e_{g}.\]
\end{trivlist}

\noindent The third piece of data is a collection of untwisted bimodule functors $\widehat{F}$. 
\begin{trivlist}\leftskip=3em
\item The restriction of $\cF$ to $d_{\phi(g)}$ gives a functor 
\begin{equation}\label{eq:twistfun} F_g:=  \cF|_{d_{\phi(g)}} :   d_{\phi(g)} \overset{\sim} {\rightarrow} e_g. \end{equation}
The structure map $\tau$ equips $F_g$ with the structure of an $\cF_e$-twisted bimodule equivalence. An application of Theorem~\ref{lem:untwisting} to the collection $\{F \}$ then gives a collection of untwisted bimodule functors:
\begin{equation}\label{eq:Fcol} \widehat{F} :=  \{ \widehat{F_g} : \cC_{\cF_e} \boxtimes  d_{\phi(g)}\boxtimes _{\cF_e}\cC \to e_g \} =   \left \{ \raisebox{-.5\height}{ \includegraphics[scale = .7]{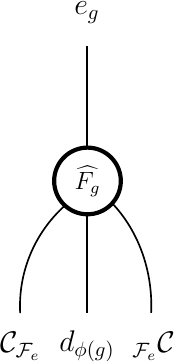}}  \right \}. \end{equation}
\end{trivlist}

\noindent The final piece of data is a collection of untwisted bimodule natural isomorphisms $\widehat{T}$. 
\begin{trivlist}\leftskip=3em
\item The restriction of the natural isomorphism $\tau$ onto the graded pieces $d_{\phi(g)}$ and $d_{\phi(h)}$ gives a natural isomorphism 
\[  \tau|_{d_{\phi(g)}, d_{\phi(h)}} : \otimes^\cE|_{e_g,e_h} \circ   (F_g \times F_h) \overset{\sim}{\rightarrow}  F_{gh} \circ \otimes^{\cD}|_{\phi(g),\phi(h)} .   \]
It is routine to check that this isomorphism is $\cC$-balanced, and thus induces, via Equation~\eqref{eq:funball}, a natural isomorphism:
\begin{equation}\label{data:tensorator}
T_{g,h} : M_{g, h}^\cE \circ (F_g \boxtimes F_h) \overset{\sim}{\rightarrow}  F_{gh} \circ M_{\phi(g),\phi(h)}^\cD.
\end{equation}
A direct calculation shows that the natural isomorphism $T_{g,h}$ is $\cF_e$-twisted. 

An application of Theorem~\ref{lem:untwisting} to the collection $\{ T \}$ gives a collection of untwisted bimodule natural isomorphisms:
\begin{equation}\label{data:tenso} \widehat{T}:= \left \{\widehat{T_{g,h}}(\id_{M_{g,h}^\cE} \circ {\Omega^\boxtimes_{F_g,F_h}}^{-1})\right \} = \left \{\raisebox{-.5\height}{ \includegraphics[scale = .5]{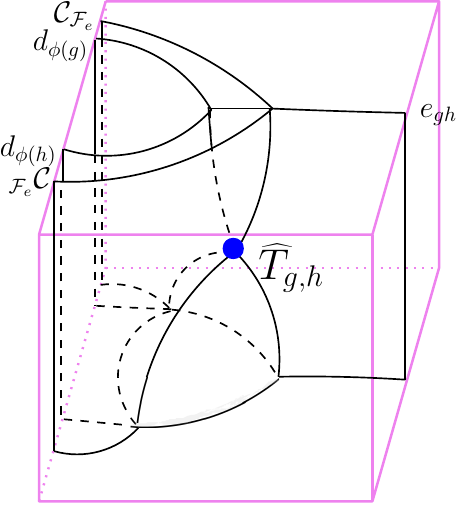}} \right \}. \end{equation}
\end{trivlist}
This quadruple of data constructed from $(\cF, \tau)$ satisfies some conditions. The tensorator axiom implies the following equality of $\cF_e$-twisted natural isomorphisms
\begin{align}\label{data:3obs}
&[T_{f,gh}\circ( \id_{\Id_{d_{\phi(f)}}} \boxtimes \id_{M^\cD_{\phi(g),\phi(h)}})] 
[ \id_{M^\cE_{f,gh}}\circ( \id_{F_f} \boxtimes T_{g,h})]
  [A^\cE_{f,g,h}\circ (\id_{F_f}\boxtimes \id_{F_g}\boxtimes \id_{F_h})]=\\
&[ \id_{F_{fgh}}  \circ {A^\cD_{\phi(f),\phi(g),\phi(h)}} ] 
[  T_{fg,h} \circ ( \id_{M^\cD_{\phi(f),\phi(g)}}\boxtimes \id_{\Id_{d_{\phi(h)}}})] 
[ \id_{M^\cE_{fg,h}} \circ (T_{f,g} \boxtimes \id_{F_h})],
\end{align}
which via a rearrangement, and an application of Theorem~\ref{lem:untwisting} is equivalent to the bimodule natural automorphism $v(\cF_e,\phi,\widehat{F},\widehat{T})_{f,g,h}$ defined in Figure~\ref{eq:v3} being the trivial automorphism of \[H_{f,g,h} := \raisebox{-.5\height}{ \includegraphics[scale = .5]{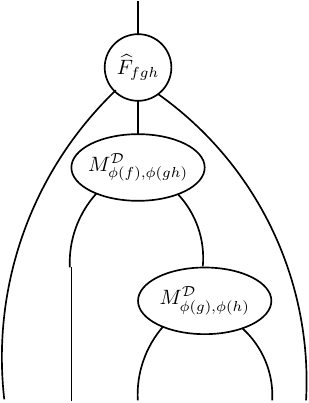}}\] for all $f,g,h \in G$.
\begin{figure}
\caption{Graphical description of $v(\cF_e,\phi,\widehat{F},\widehat{T})_{f,g,h}$}\label{eq:v3}
\[
\begin{CD}
\raisebox{-.5\height}{ \includegraphics[scale = .4]{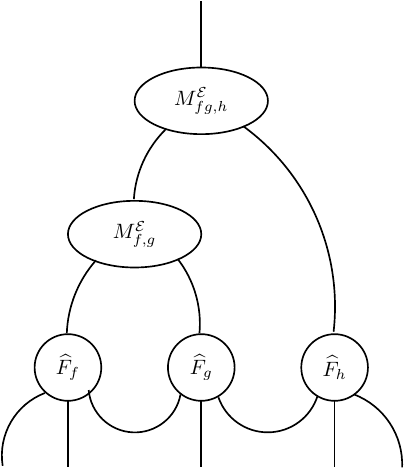}} @>\raisebox{-.5\height}{ \includegraphics[scale = .4]{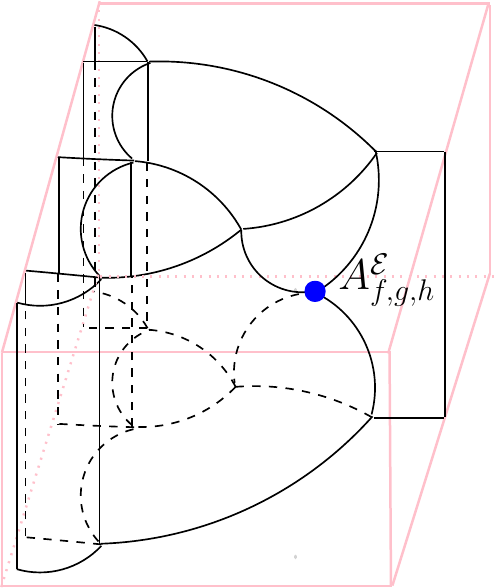}}>> \raisebox{-.5\height}{ \includegraphics[scale = .4]{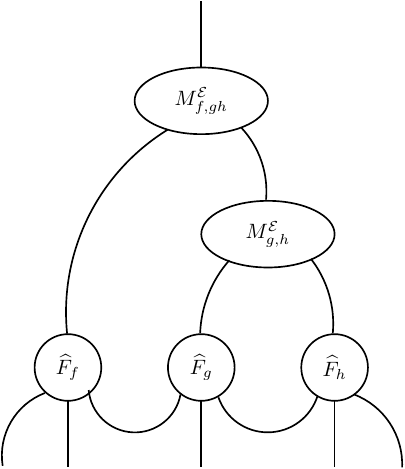}} \\
@A\raisebox{-.5\height}{ \includegraphics[scale = .4]{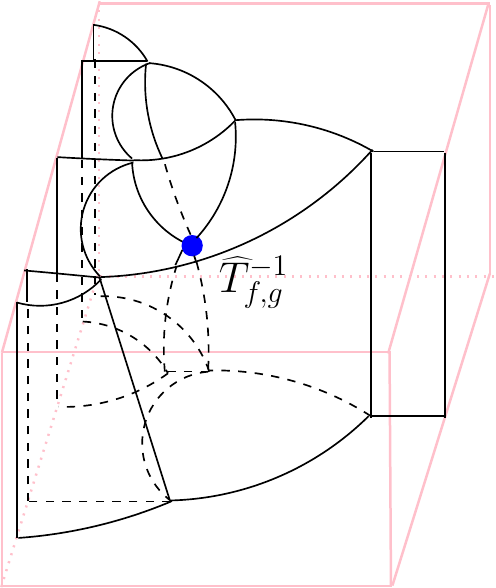}}AA @VV\raisebox{-.5\height}{ \includegraphics[scale = .4]{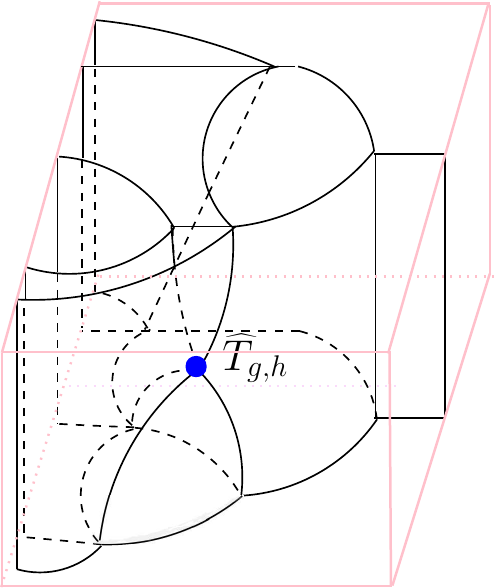}}V\\
\raisebox{-.5\height}{ \includegraphics[scale = .4]{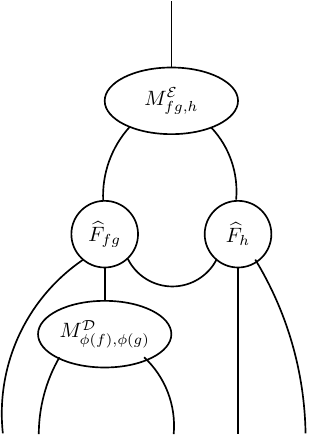}} @. \raisebox{-.5\height}{ \includegraphics[scale = .4]{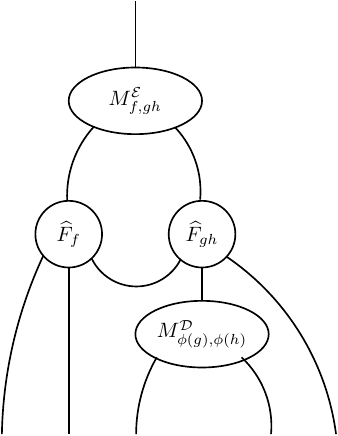}} \\
@A\raisebox{-.5\height}{ \includegraphics[scale = .4]{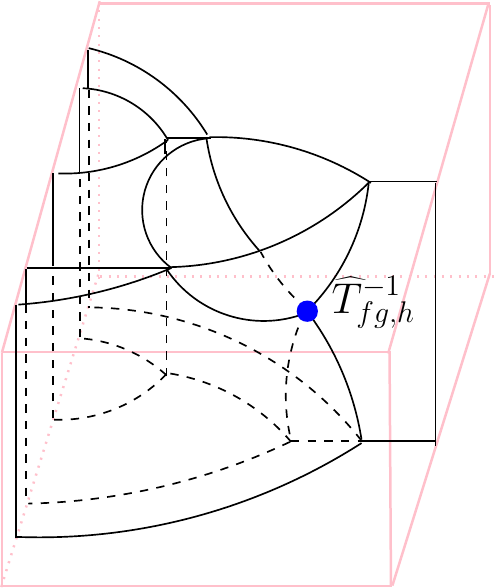}}AA @VV\raisebox{-.5\height}{ \includegraphics[scale = .4]{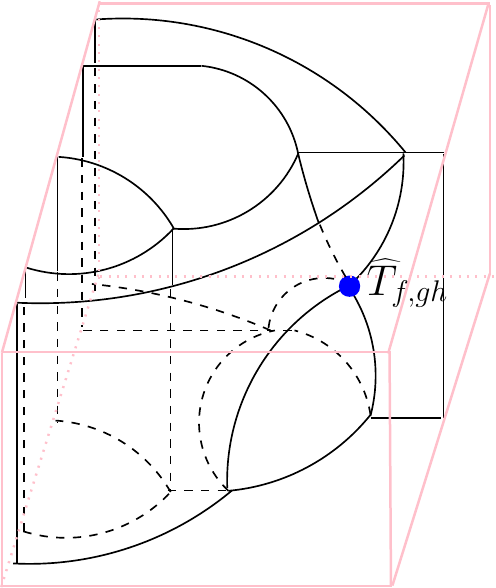}}V\\
\raisebox{-.5\height}{ \includegraphics[scale = .4]{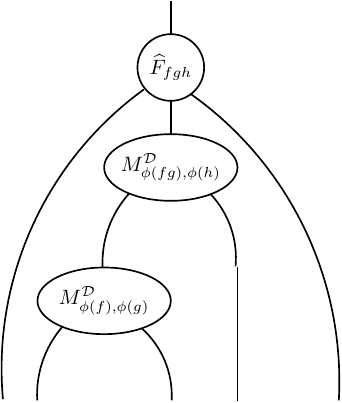}} @<<\raisebox{-.5\height}{ \includegraphics[scale = .4]{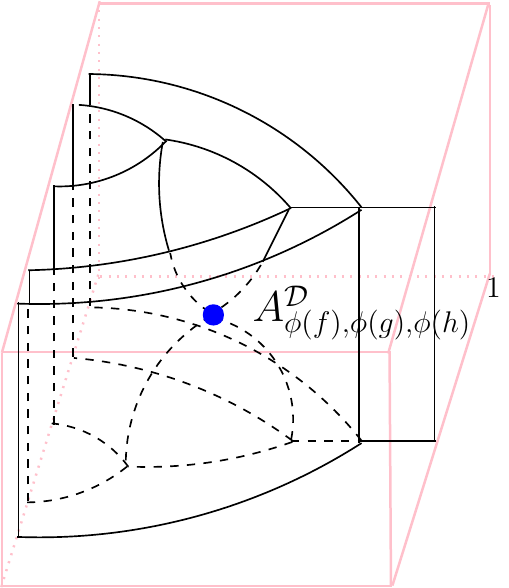}}< \raisebox{-.5\height}{ \includegraphics[scale = .4]{hex4}} \\
\end{CD}
\]\end{figure}
														
In summary the $\cC$-based auto-equivalence $(\cF,\mu)$ determines the following quadruple of data:
\begin{enumerate}
\item A monoidal auto-equivalence $\cF_e: \cC \to \cC$,

\item a group automorphism $\phi: G \to G$,

\item a collection $\widehat{F}$ of bimodule functors,

\item a collection $\widehat{T}$ of bimodule natural isomorphisms, such that $v(\cF_e,\phi,\widehat{F},\widehat{T})_{f,g,h}$ is trivial for all $f,g,h \in G$.

\end{enumerate} 

Conversely, given any such quadruple $(\cF_e,\phi,\widehat{F},\widehat{T})$, with $v(\cF_e,\phi,\widehat{F},\widehat{T})$ vanishing, we can construct a $\cC$-based equivalence $\cD\to \cE$ as follows. To start we use Theorem~\ref{lem:untwisting} on the collections $\widehat{F}$ to get a collection of $\cF_e$-twisted bimodule functors
\[  F_g:=  \cF_{d_{\phi(g)}} :   d_{\phi(g)} \overset{\sim} {\rightarrow} e_g. \]
We assemble this collection of $\cF_e$-twisted bimodule functors to get an abelian functor
\[   \bigoplus_G F_g : \cD \to \cE.\]

Next we use Theorem~\ref{lem:untwisting} on the collection $\widehat{T}$ to get a collection of $\cF_e$-twisted bimodule natural isomorphisms
\[ M_{g, h}^\cE \circ (F_g \boxtimes F_h) \overset{\sim}{\rightarrow}  F_{gh} \circ M_{\phi(g),\phi(h)}^\cD. \]
These bimodule natural isomorphism are $\cC$-balanced, and thus via Equation~\ref{eq:funball} give natural isomorphisms
\[   \otimes^\cE|_{e_g,e_h} \circ   (F_g \times F_h) \overset{\sim}{\rightarrow}  F_{gh} \circ \otimes^{\cD}|_{\phi(g),\phi(h)},   \]
which serve as a tensorator for the above abelian functor.  The fact that $v(\cF_e,\phi,\widehat{F},\widehat{T})$ vanishes ensures that our tensorator satisfies the hexagon equation. Thus we have constructed a $\cC$-based equivalence $\cD\to \cE$.

It is straightforward to check that these de-construction and re-contruction processes are inverse to each other. Thus $\cC$-based equivalences $\cD \to \cE$ are classified, up to natural isomorphism of $\cC$-based equivalences (Definition~\ref{def:basediso}), by quadruples of data as described above.

The goal of the next several sections will be to classify quadruples $(\cF_e,\phi,\widehat{F},\widehat{T})$ in terms of cohomological data, in a similar vein to the classification of $G$-graded extensions we covered in Section~\ref{sec:gradedex}.

\begin{remark}\label{lem:twistquad}
$\cC$-based equivalence $\cD \to \cE$ are also classified by quadruples $(\cF_e,\phi, F, T)$, where $F$ is a collection of $\cF_e$-twisted bimodule functors as in Equation~\eqref{eq:twistfun}, and $T$ is a collection of $\cF_e$-twisted bimodule natural isomorphisms as in Equation~\eqref{data:tensorator}, satisfying Equation~\eqref{data:3obs}. However these quadruples are much harder to directly classify, so we restrict our attention to the untwisted quadruples for this paper. Converting from a quadruple of the form $(\cF_e,\phi, F, T)$, to one of the form $(\cF_e,\phi, \widehat{F}, \widehat{T})$ is exactly an application of Theorem~\ref{lem:untwisting}.
\end{remark}

\subsection*{Composition of quadruples}
Suppose we have three $G$-graded extensions of $\cC$ which we call $\cD$, $\cE$, and $\cY$, and $\cC$-based equivalences $\cD \to  \cE$ and $\cE \to \cY$. As discussed above, these two $\cC$-based equivalences correspond to quadruples $(\cF_e, \phi, \widehat{F}, \widehat{T})$ and $(\cF_e', \phi',  \widehat{F}', \widehat{T}')$ respectively. Clearly we can compose the two $\cC$-based equivalences to obtain a $\cC$-based equivalence $\cD \to \cY$. The natural question is to ask what does the quadruple for the   composition look like.  By reconstructing the equivalences from the quadruples, composing, then extracting the resulting quadruple we can arrive at an answer. However the resulting quadruple is complicated so we omit the details. For twisted quadruples as in Remark~\ref{lem:twistquad} the answer is simpler. Let $(\cF_e, \phi, F , T )$ and $(\cF_e', \phi',  F' , T')$ be two such twisted quadruples, with appropriate sources and targets, then the resulting twisted quadruple of their composition is 
\[ (\cF_e', \phi', F_g', T_{g,h}')\circ (\cF_e, \phi, F_g, T_{g,h}) = (\cF_e' \circ F_e, \phi' \circ \phi, F_{\phi(g)}' \circ F_g, [T'_{\phi(g),\phi(h)} \circ (\id_{F_g} \boxtimes \id_{F_h})][ \id_{F'_{\phi(gh)}} \circ T_{g,h}]). \]

\section{Existence and classification of quasi-monoidal $\cC$-based equivalences}\label{sec:quasi}
The goal for this section is the intermediate step of classifying of all quasi-monoidal $\cC$-based equivalences between two $G$-graded extensions of $\cC$. Recall a quasi-monoidal equivalence $\cJ$ is an equivalence that admits an (unspecified) isomorphism $\cJ(X) \otimes \cJ(Y)\to \cJ(X\otimes Y)$ not satisfying any additional conditions.

Let $\cD$ and $\cE$ be two $G$-graded extensions of $\cC$, which correspond to the triples $(d,M^\cD,A^\cD)$ and $(e,M^\cE,A^\cE)$ respectively. It follows from the discussion of Section~\ref{sec:data} that quasi-monoidal $\cC$-based equivalences $\cD \to \cE$ are classified by triples $(\cF_e,\phi,\widehat{F})$ where $\cF_e$ is a monoidal auto-equivalence of $\cC$, $\phi$ is an automorphism of $G$, and $\widehat{F}$ is a collection of untwisted bimodule functors as in Equation~\eqref{eq:Fcol}, with the collection $\widehat{F}$ satisfying the condition
\begin{equation}\label{eq:quasicond} \raisebox{-.5\height}{ \includegraphics[scale = .6]{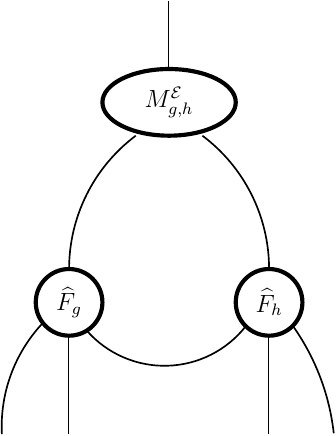}}  \eq \raisebox{-.5\height}{ \includegraphics[scale = .6]{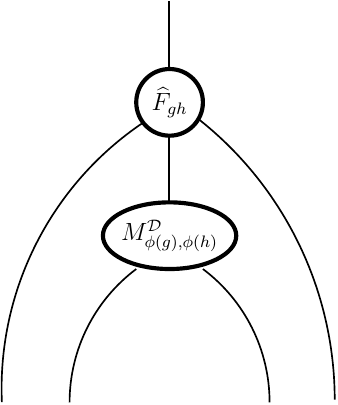}}.   \end{equation}
for all $g,h \in G$.

To begin we choose $\cF_e$ and $\phi$. For us to pick a collection $\widehat{F}$ we need for the invertible $\cC$-bimodules
\[     \cC_{\cF_e} \boxtimes d_{\phi(g)} \boxtimes _{\cF_e}\cC \quad \text{ and }\quad e_g,\]
to be equivalent for all $g\in G$. Thus we can choose a collection $\widehat{F}$ as in Equation~\eqref{eq:Fcol} if and only if the the two group homomorphisms $G\to \BrPic(\cC)$ given by
\begin{equation}\label{eq:equation}   \operatorname{Inn} (  _{\cF_e}\cC ) \circ d \circ \phi \quad \text{ and } \quad e  \end{equation}
are equal.

\begin{remark}\label{rem:inn}
Here we write $ \operatorname{Inn} (    _{\cF_e}\cC )$ for the inner automorphism of the group $\BrPic(\cC)$ induced from the object $_{\cF_e}\cC $.
\end{remark} 

Now that we have determined when we can choose a collection $\widehat{F}$, we wish to determine when the collection $\widehat{F}$ satisfies Equation~\ref{eq:quasicond}, and hence the triple $(\cF_e, \phi, \widehat{F})$ gives rise to a quasi-monoidal $\cC$-based equivalence $\cD \to \cE$. To help us determine when Equation~\ref{eq:quasicond} is satisfied we define
\begin{align*} T(\cF_e,\phi,\widehat{F})_{g,h} :=&  \operatorname{R}_{e_{gh}} \circ [ M^\cE_{g,h} \circ ( \widehat{F_g}\boxtimes \widehat{F_h}) \circ ( \Id_{\cC_{\cF_e}} \boxtimes \Id_{d_{\phi(g)}} \boxtimes \operatorname{R}^{-1}_{\cC_{\cF_e}} \boxtimes \Id_{d_{\phi(h)}} \boxtimes \Id_{_{\cF_e}\cC})   \\
& \circ ( \Id_{\cC_{\cF_e}} \boxtimes {M^\cD_{\phi(g),\phi(h)}}^{-1} \boxtimes \Id_{_{\cF_e}\cC}) \circ \widehat{F_{gh}}^{-1} \boxtimes \Id_{e_{gh}^{op}})]\circ \operatorname{R}^{-1}_{e_{gh}} : G\times G \to \Inv(\cZ(\cC)), \end{align*}
and note that Equation~\ref{eq:quasicond} is satisfied if and only if $T(\cF_e,\phi,\widehat{F})_{g,h} $ is trivial for all $g,h \in G$. The function $T(\cF_e,\phi,\widehat{F})$ only depends on the choice of $\cF_e$, $\phi$, and $\widehat{F}$, and not on the choice of maps $R_{e_{gh}} : e_{gh} \boxtimes e_{gh}^{op} \to \cC$. Graphically we draw:
\begin{equation*}
T(\cF_e,\phi,\widehat{F})_{g,h} =\raisebox{-.5\height}{ \includegraphics[scale = .6]{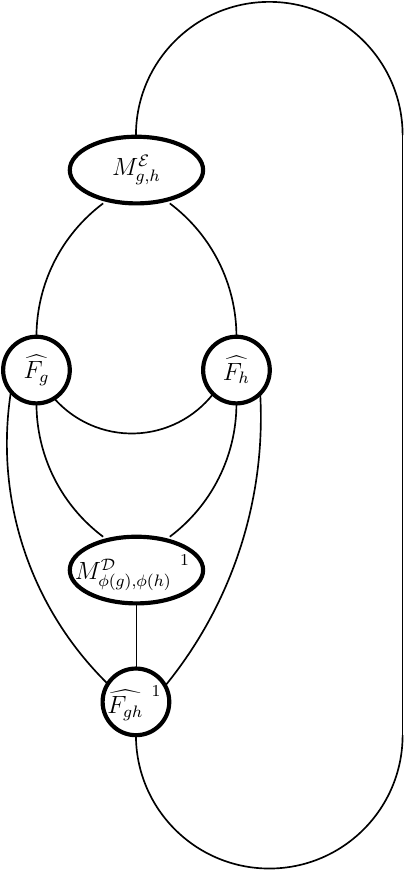}}
\end{equation*}

\begin{defn}\label{def:syseq}
We call a collection of bimodule equivalence $\widehat{F}$ such that $T(\cF_e,\phi,\widehat{F})$ vanishes, a \textit{system of equivalences for $(\cF_e,\phi)$}.
\end{defn}

 We consider systems of equivalences for $(\cF_e,\phi)$ up to quasi-monoidal $\cC$-based isomorphism of the corresponding quasi-monoidal $\cC$-based equivalences. Explicitly this means that two systems of equivalences $\widehat{F}$ and $\widehat{F}'$ for $(\cF_e,\phi)$ are the equivalent if and only if there exists a collection of bimodule natural isomorphisms $\{ \mu_g : \widehat{F}_g \to \widehat{F}'_g \}$ with $\mu_e = \id_{\widehat{F}_e}$.

We directly verify that $T(\cF_e,\phi,\widehat{F})$ is a 2-cocycle valued in $\Inv(\cZ(\cC))$.

\begin{remark}
The action of $G$ on $\Inv(\cZ(\cC))$ is given by the homomorphism $e : G \to \BrPic(\cC)$.
\end{remark}

\begin{lem}
We have
\[ T(\cF_e,\phi,\widehat{F})_{g,h}^{e_f} T(\cF_e,\phi,\widehat{F})_{f,gh} =  T(\cF_e,\phi,\widehat{F})_{f,g}T(\cF_e,\phi,\widehat{F})_{fg,h}   \]
\end{lem}
\begin{proof}
We prove $T(\cF_e,\phi,\widehat{F})$ is a 2-cocycle using the graphial calculus for $\BBrPic(\cC)$. At each stage in this calculation we indicate with a red box the region we will modify, and explain the modification in written words beneath the picture proof.
\begin{align*}
&T(\cF_e,\phi,\widehat{F})_{g,h}^{e_f} T(\cF_e,\phi,\widehat{F})_{f,gh}\eq  \raisebox{-.5\height}{ \includegraphics[scale = .5]{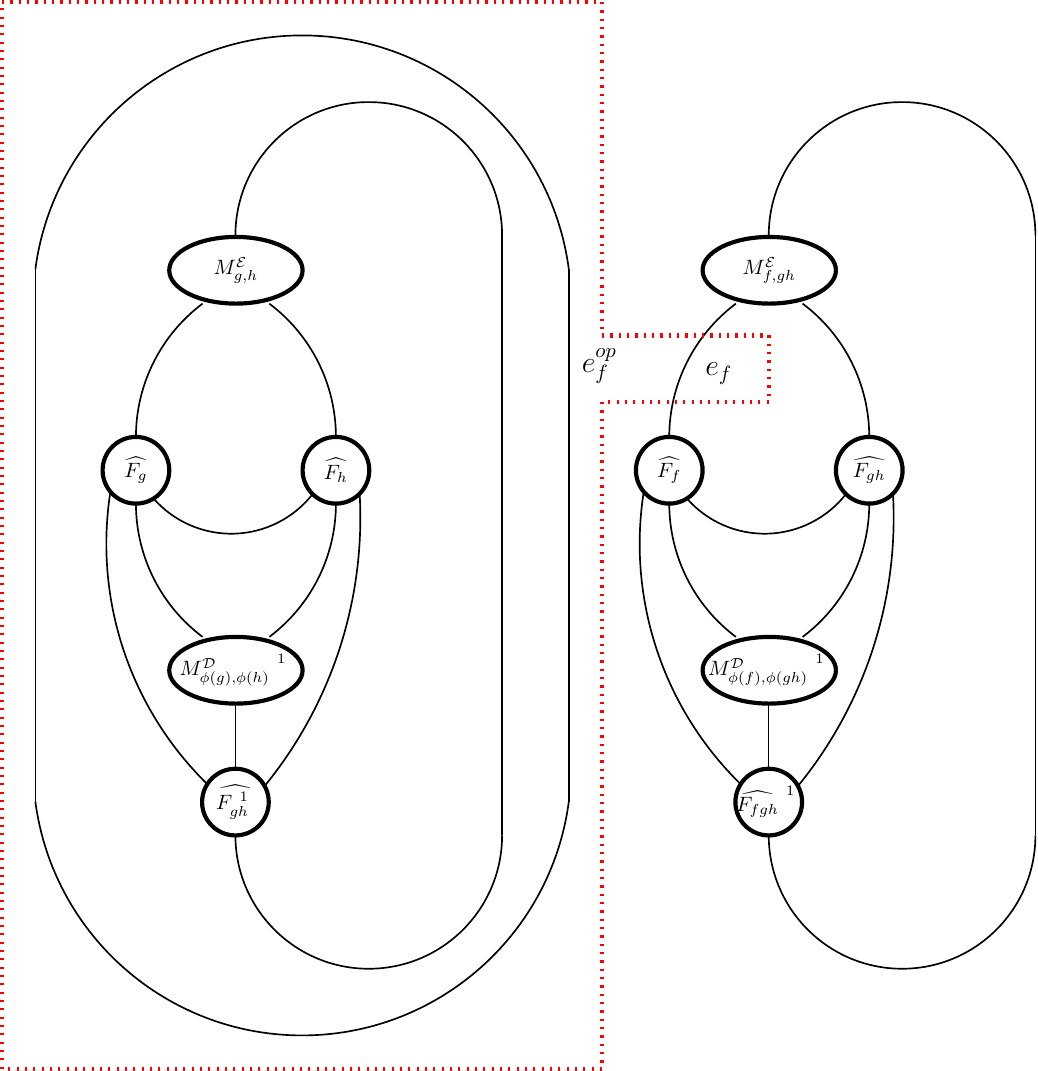}} \\
&= \quad \raisebox{-.5\height}{ \includegraphics[scale = .5]{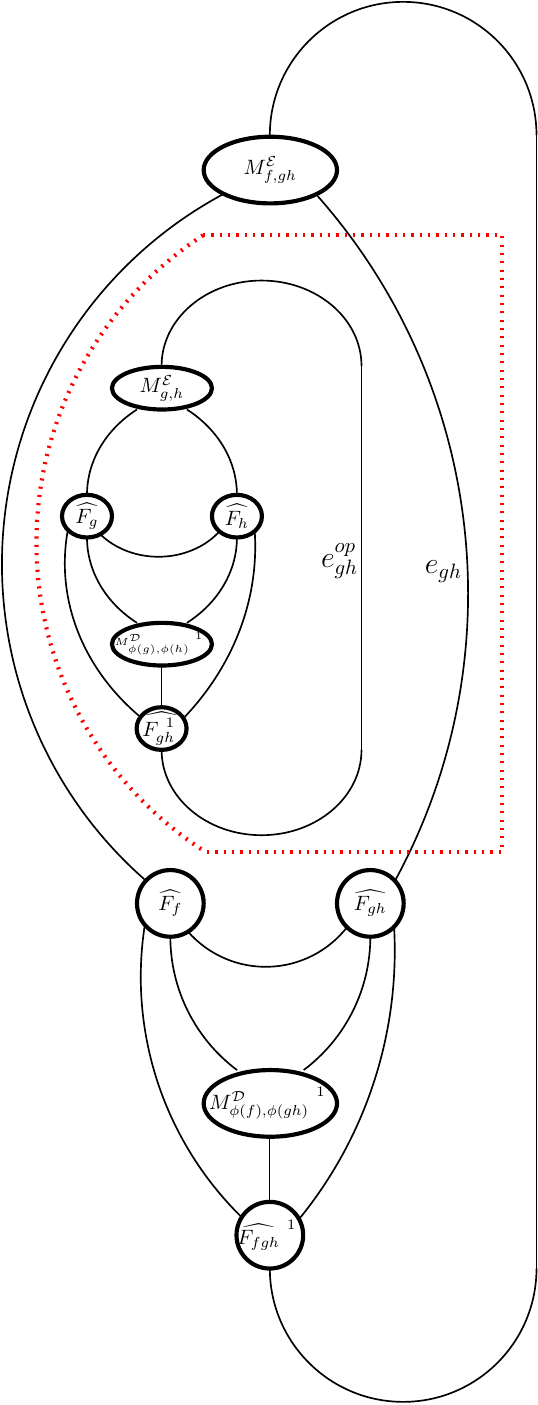}} \eq \raisebox{-.5\height}{ \includegraphics[scale = .5]{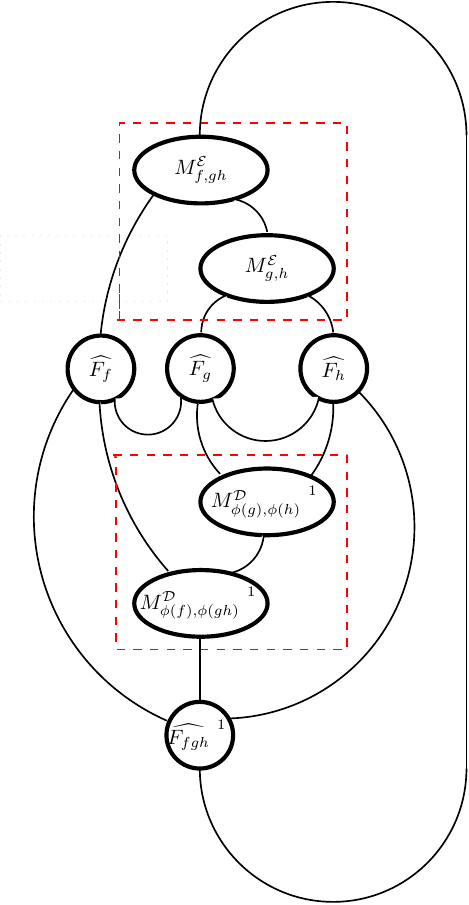}} \eq \raisebox{-.5\height}{ \includegraphics[scale = .6]{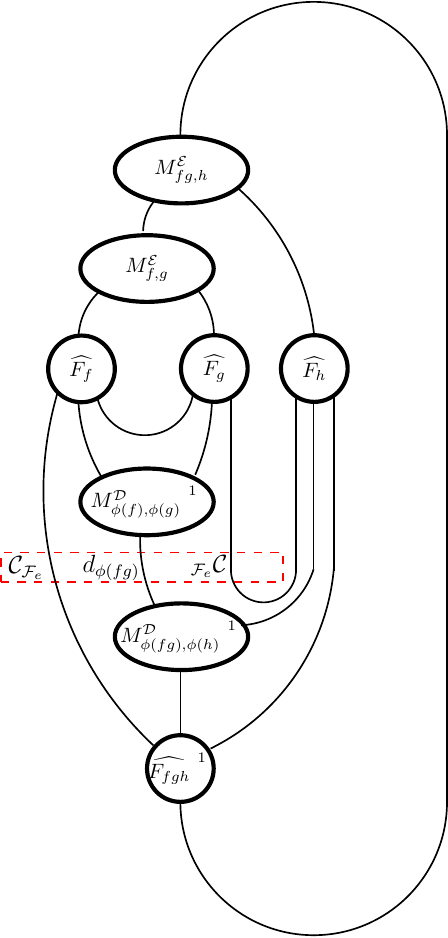}}\\
& = \quad   \raisebox{-.5\height}{ \includegraphics[scale = .6]{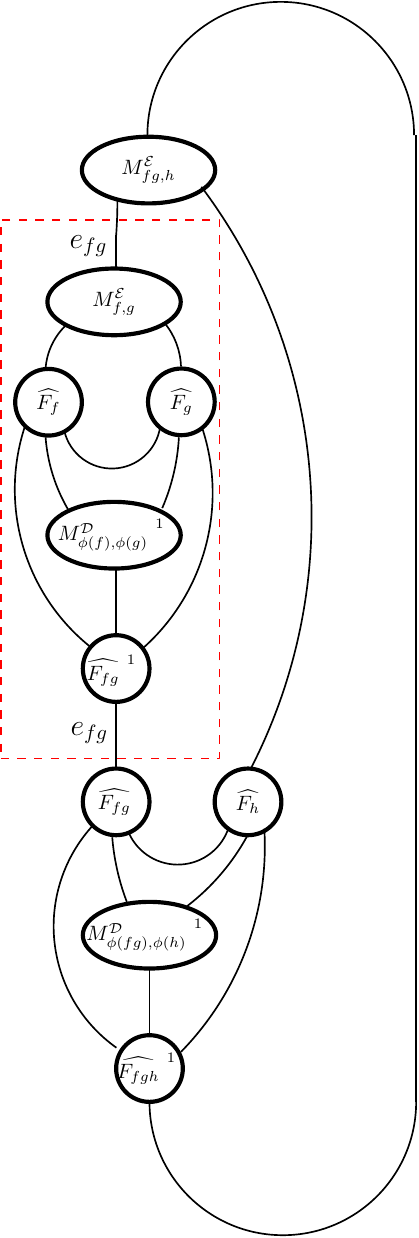}}\eq   \raisebox{-.5\height}{ \includegraphics[scale = .6]{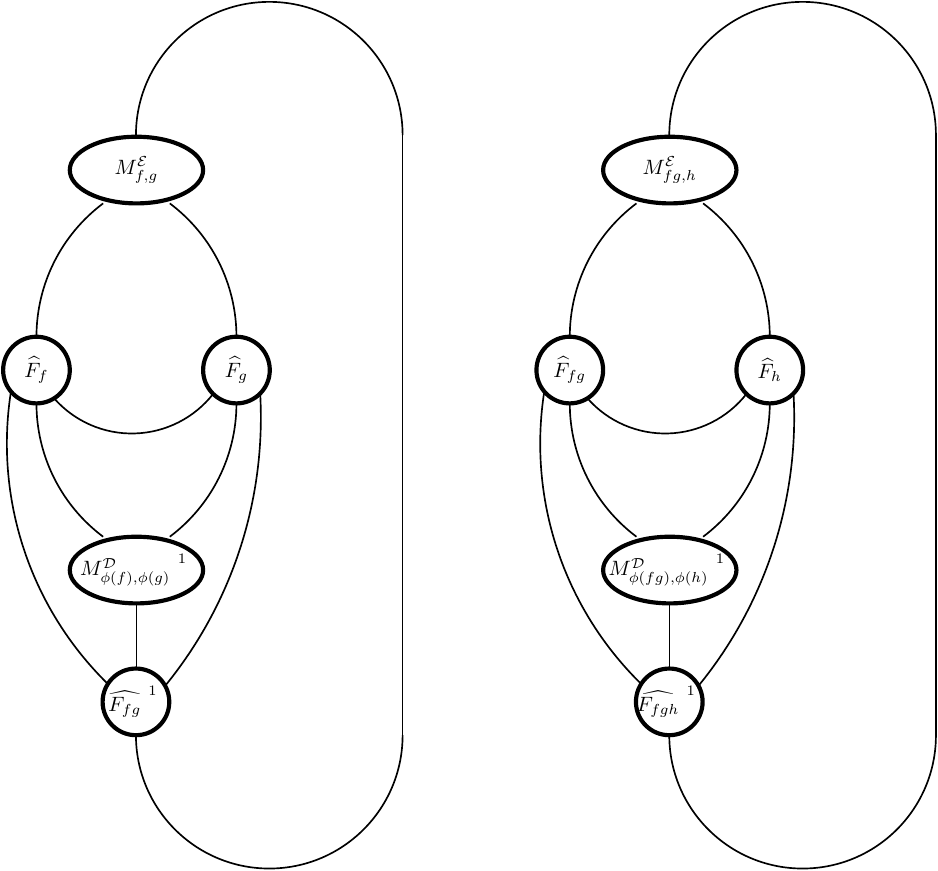}}   \\
&  = T(\cF_e,\phi,\widehat{F})_{f,g}T(\cF_e,\phi,\widehat{F})_{fg,h}.
\end{align*}
For the first equality we apply Lemma~\ref{lem:lact}. For the second equality we recouple the $e_{gh}^{op}$ and $e_{gh}$ strings, then apply Lemma~\ref{lem:anue}. For the third equality we use the fact that $M^\cD$ and $M^\cE$ are systems of products. For the fourth equality we replace the $\cC_{\cF_e} \boxtimes d_{\phi(fg)} \boxtimes _{\cF_e}\cC$ string with $\widehat{F}_{fg}^{-1} \circ \widehat{F}_{fg}$. For the fifth equality we apply Lemma~\ref{lem:anue}, and then recouple the resulting $e_{gh}$ cups and caps.
\end{proof}

To simplify our proofs for the remainder of this section we introduce the following notation.
\begin{defn}\label{def:equivact}
Let $\rho \in C^1(G, \Inv(\cZ(\cC)))$, and $\widehat{F}$ a collection of bimodule equivalences as in Equation~\eqref{eq:Fcol}. We define
\[   \widehat{F}^\rho :=   \{ \rho_g \boxtimes \widehat{F}_g \} = \left \{   \raisebox{-.5\height}{ \includegraphics[scale = .6]{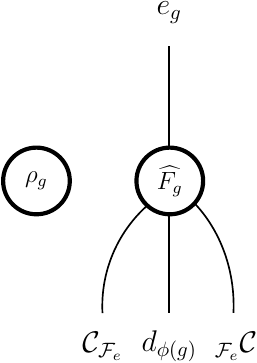}}  \right \}.\]
\end{defn}
As was the case for analogous action on systems of products in Section~\ref{sec:gradedex}, this action of $C^1(G, \Inv(\cZ(\cC)))$ does not act of systems of equivalences for $(\cF_e, \phi)$. That is for $\rho \in C^1(G, \Inv(\cZ(\cC)))$, we do not have that
\[   T(\cF_e,\phi, \widehat{F}^\rho)_{g,h} = T(\cF_e,\phi, \widehat{F})_{g,h}. \]

Instead we have the following Lemma.
\begin{lem}\label{lem:2cob}
Let $\rho \in C^1(G, \Inv(\cZ(\cC)))$, then we have the relation
\begin{align*}
 T(\cF_e,\phi, \widehat{F}^\rho)_{g,h} = \partial^2(\rho)_{g,h} T(\cF_e,\phi, \widehat{F})_{g,h}.
 \end{align*}
\end{lem}
\begin{proof}
We compute using Lemma~\ref{lem:lact}
 \[ T(\cF_e,\phi, \widehat{F}^\rho)_{g,h} =\raisebox{-.5\height}{ \includegraphics[scale = .6]{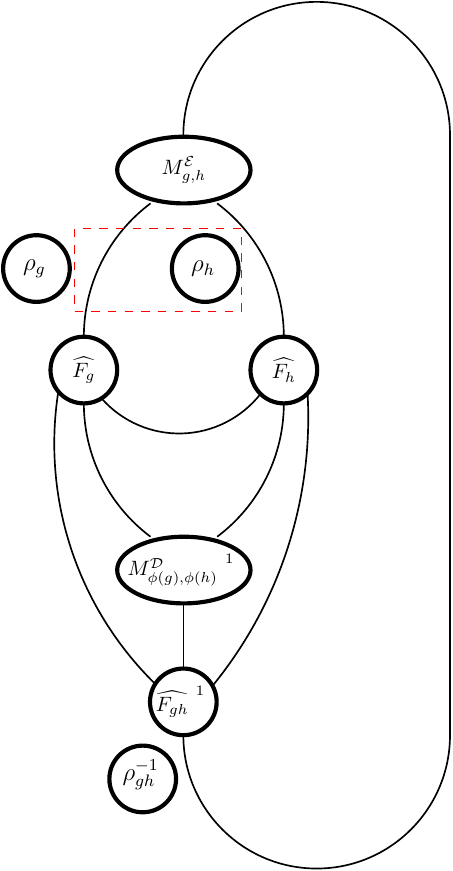}} = \raisebox{-.5\height}{ \includegraphics[scale = .6]{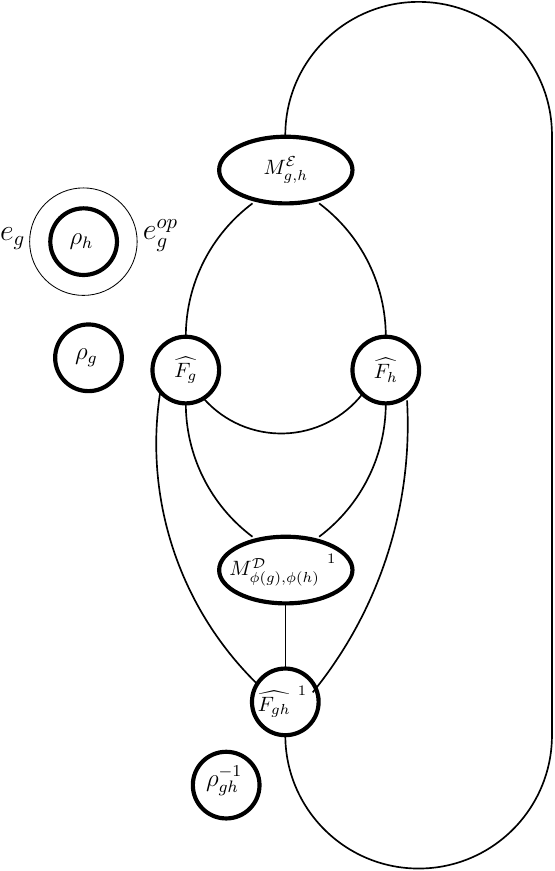}}=  \partial^2(\rho)_{g,h}T(\cF_e,\phi,\widehat{F})_{g,h}.      \]
 \end{proof}

Suppose we have two collections of bimodule equivalences $\widehat{F}$ and $\widehat{F}'$ as in Equation~\eqref{eq:Fcol}. Then each $\widehat{F}_g'$ differs from $\widehat{F}_g$ by an element $L_g \in \Inv(\cZ(\cC))$. Thus $L$ is an element of $C^1(G, \Inv(\cZ(\cC)))$, such that $\widehat{F}' = \widehat{F}^L$ (in the notation of Definition~\ref{def:equivact}).

Applying Lemma~\ref{lem:2cob} shows $T(\cF_e,\phi, \widehat{F'})$ and $ T(\cF_e,\phi, \widehat{F})$ differ by a coboundary in $Z^2(G, \Inv(\cZ(\cC)))$. Thus the cohomology class of $ T(\cF_e,\phi, \widehat{F})$ is independent from the choice of $\widehat{F}$, and only depends on the choice of $\cF_e$ and $\phi$

\begin{defn}\label{def:o2obs}
We write $o_2(\cF_e,\phi)$ for the cohomology class of $ T(\cF_e,\phi, \widehat{F})$ in $H^2(G, \Inv(\cZ(\cC)))$
\end{defn}
From the element $o_2(\cF_e,\phi) \in H^2(G, \Inv(\cZ(\cC)))$ we can deduce a necessary and sufficient condition for there to exist a system of equivalences for $(\cF_e, \phi)$

\begin{lem}\label{lem:non-emptytorsor1}
For fixed $\cF_e$ and $\phi$, there exists a system of equivalences for $(\cF_e,\phi)$ if and only if the two homomorphisms $G \to \BrPic(\cC)$ given by
\[ \operatorname{Inn} (  _{\cF_e}\cC ) \circ d \circ \phi \quad \text{ and } \quad  e\]
are equal, and the cohomology class $o_2(\cF_e,\phi)$ is trivial in $H^2(G,\Inv(\cZ(\cC)))$.
\end{lem}
\begin{proof}
Suppose $\widehat{F}$ is a system of equivalences for $(\cF_e,\phi)$. By definition $T(\cF_e,\phi, \widehat{F})$ is trivial in $Z^2(G,\Inv(\cZ(\cC)))$, and therefore $o_2(\cF_e,\phi)$ is the trivial element of $H^2(G,\Inv(\cZ(\cC)))$. As each $\widehat{F}_g$ is a bimodule equivalence $\cC_{\cF_e}\boxtimes d_{\phi(g)} \boxtimes _{\cF_e}\cC \to e_g$ we immediately have that $\operatorname{Inn} (  _{\cF_e}\cC ) \circ d \circ \phi = e$.

Conversely, suppose $\operatorname{Inn} (  _{\cF_e}\cC ) \circ d \circ \phi = e$, then we can make an arbitrary choice of bimodule equivalences 
\[   \widehat{F}_g: \cC_{\cF_e}\boxtimes d_{\phi(g)} \boxtimes _{\cF_e}\cC \to e_g.\] 
As $o_2(\cF_e,\phi)$ is trivial in $H^2(G,\Inv(\cZ(\cC)))$ we have that the 2-cocycle $T(\cF_e,\phi,\widehat{F})$ is a coboundary, and so $T(\cF_e,\phi,\widehat{F})= \partial^2(L)$ for some $L \in C^1(G, \Inv(\cZ(\cC)))$. Using the action of $C^1(G, \Inv(\cZ(\cC)))$ on $M$ from Definition~\ref{def:equivact}, we construct a new collection of bimodule equivalences $M^{L^{-1}}$. We compute with Lemma~\ref{lem:2cob} that
\[ T(\cF_e,\phi,\widehat{F}^{L^{-1}}) \cong \partial^2( L^{-1})\partial^2(L) = \Id.    \] 
Thus $\widehat{F}^{L^{-1}}$ is a system of equivalences for $(\cF_e,\phi)$.
\end{proof}

Given that we have found exactly when there exists a system of equivalences for $(\cF_e, \phi)$, our next goal is to classify all systems of equivalences for $(\cF_e, \phi)$.

\begin{lem}\label{lem:tor1}
For fixed $\cF_e$ and $\phi$, such that 
\[  \operatorname{Inn} (  _{\cF_e}\cC ) \circ d \circ \phi = e, \]
and such that $o_2(\cF_e,\phi)$ is trivial in $H^2(G,\Inv(\cZ(\cC)))$, we have that systems of equivalences for $(\cF_e,\phi)$ form a torsor over $Z^1(G, \Inv(\cZ(\cC)))$.
\end{lem}
\begin{proof}\hspace{1em}
The action of $Z^2(G,\Inv(\cZ(\cC)))$ on system of equivalences for $(\cF_e,\phi)$ is as defined in Definition~\eqref{def:equivact}.

\noindent We begin by showing the action of $Z^2(G, \Inv(\cZ(\cC)))$ on systems of equivalences for $(\cF_e,\phi)$ is well defined.
\begin{trivlist}\leftskip=3em 
\item Let $\rho \in Z^1(G,\Inv(\cZ(\cC))$, and $\widehat{F}$ a system of equivalences for $(\cF_e,\phi)$. Then $\rho$ acts on $\widehat{F}$ as in Definition~\eqref{def:equivact}. As $\rho$ is a 1-cocycle, we have $\partial^2(\rho)$ is trivial, and thus Lemma~\ref{lem:2cob} implies that 
\[   T(\cF_e,\phi,\widehat{F}^\rho) = T(\cF_e,\phi,\widehat{F}).\]
 As $\widehat{F}$ is a system of equivalences for $(\cF_e,\phi)$ we know that $T(\cF_e,\phi,\widehat{F})$ vanishes, and thus $T(\cF_e,\phi,\widehat{F}^\rho)$ also vanishes. Hence $\widehat{F}^\rho$ is a system of equivalences for $(\cF_e,\phi)$.
\end{trivlist}

\noindent Next we show that the action of $Z^1(G, \Inv(\cZ(\cC)))$ on systems of equivalences for $(\cF_e,\phi)$ is transitive.
\begin{trivlist}\leftskip=3em 
\item Let $\widehat{F}$ and $\widehat{F}'$ be two systems of equivalences for $(\cF_e,\phi)$. As explained earlier, $\widehat{F_g}' = L_g \widehat{F_g}$ for some $L \in C^1(G,\Inv(\cZ(\cC)))$, therefore $\widehat{F}'= \widehat{F}^L$. Using Lemma~\ref{lem:lact} we get that 
\[  T(\cF_e,\phi, \widehat{F}') = \partial^2(L)T(\cF_e,\phi,\widehat{F}).\] 
As $\widehat{F}$ and $\widehat{F}'$ are both systems of equivalences for $(\cF_e,\phi)$, we have by definition that $T(\cF_e,\phi, \widehat{F}')$ and $T(\cF_e,\phi,\widehat{F})$ vanish. Thus $\partial^2(L)$ must be trivial, and so $L \in Z^1(G,\Inv(\cZ(\cC)))$.
\end{trivlist}

\noindent  Finally we show that the action of $Z^1(G, \Inv(\cZ(\cC)))$ on systems of equivalences for $(\cF_e,\phi)$ is free.
\begin{trivlist}\leftskip=3em 
\item  Let $\rho \in Z^1(G, \Inv(\cZ(\cC)))$ and $\widehat{F}$ a system of equivalences for $(\cF_e,\phi)$. Suppose $\widehat{F}^\rho$ is equivalent to $\widehat{F}$ as systems of equivalences for $(\cF_e,\phi)$. Then there exist bimodule natural isomorphisms $\mu_g:\widehat{F}_g^\rho = \rho_g \boxtimes \widehat{F}_g \to \widehat{F}_g$. Thus each $\rho_g$ must be the trivial element of $\Inv(\cZ(\cC))$, and hence $\rho$ is the trivial element of $Z^1(G, \Inv(\cZ(\cC)))$.
\end{trivlist}
Putting everything together we have that systems of equivalences for $(\cF_e,\phi)$ form a $Z^1(G,\Inv(\cZ(\cC)))$ torsor.
\end{proof}

We find it somewhat surprising that systems of equivalences for $(\cF_e,\phi)$ form a $Z^1(G,\Inv(\cZ(\cC)))$ torsor, rather than a $H^1(G,\Inv(\cZ(\cC)))$ torsor. A sceptical reader may argue that we have chosen too weak a notion of isomorphism between $\cC$-based equivalences, and that a suitable stronger notion of isomorphism between $\cC$-based equivalence would result in systems of equivalences forming a $H^1(G,\Inv(\cZ(\cC)))$ torsor. To convince the reader that we have chosen the best notion of $\cC$-based isomorphism we present the following example. In this example we construct two quasi-monoidal $\cC$-based auto-equivalences of a $G$-graded fusion category such that the corresponding systems of equivalences differ by a coboundary. However the corresponding quasi-monoidal auto-equivalences are non-isomorphic, even as abelian functors, let alone as quasi-monoidal functors.
\begin{example}
Consider the pointed category $\Vec(D_{2\cdot3})$ with trivial associator. Recall $D_{2\cdot3}$ is the semi-direct product of $\mathbb{Z}/3\mathbb{Z}$ by $\mathbb{Z}/2\mathbb{Z}$. Thus the category $\Vec(D_{2\cdot3})$ is a $\mathbb{Z}/2\mathbb{Z}$-extension of $\Vec(\mathbb{Z}/3\mathbb{Z})$, with the non-trivially graded piece the bimodule $\Vec(\mathbb{Z}/3\mathbb{Z})$ twisted by the non-trivial monoidal auto-equivalence of $\Vec(\mathbb{Z}/3\mathbb{Z})$. We denote this non-trivial bimodule $\cM$

Consider the identity auto-equivalence of $\Vec(D_{2\cdot3})$. By forgetting the tensorator of this auto-equivalence we get a quasi-monoidal auto-equivalence of $\Vec(D_{2\cdot 3})$, and hence a system of equivalences $\widehat{F}$ for $( \Id_{ \Vec(\mathbb{Z}/3\mathbb{Z})} , \id_{\mathbb{Z}/2\mathbb{Z}})$. Explicitly we have that 
\begin{align*}
\widehat{F}_0 &= \Id_{\Vec(\mathbb{Z}/3\mathbb{Z})} \\
\widehat{F}_1 &= \Id_{\cM}. 
\end{align*}

As the fusion category $\Vec(\mathbb{Z}/3\mathbb{Z})$ admits a modular braiding \cite{premodular-3}, we have that 
\[     \Inv(\cZ (  \Vec(\mathbb{Z}/3\mathbb{Z}) )) =    \Vec(\mathbb{Z}/3\mathbb{Z}) \times   \Vec(\mathbb{Z}/3\mathbb{Z}).\]
A direct computation shows that the bimodule $\cM$ acts on $\Inv(\cZ (  \Vec(\mathbb{Z}/3\mathbb{Z}) )) $ by applying the non-trivial automorphism of $\mathbb{Z}/3\mathbb{Z}$ to both factors. With this information we can compute that the function $\rho: G \to \Inv(\cZ(\cC))$ defined by 
\[   \rho_0 := e\times e \quad \quad \rho_1 := e\times r\]
is a 1-cocycle, and further, that $\rho$ is a coboundary.

We now consider the system of equivalences $\widehat{F}^\rho$. We see that $\widehat{F}^\rho_1$ acts non-trivially on the objects of $\cM$ (as $e\times r$ forgets to a non-trivial invertible object of $\Vec(\mathbb{Z}/3\mathbb{Z})$ ). Therefore the quasi-monoidal auto-equivalence associated to the system of equivalences $\widehat{F}^\rho$ acts non-trivially on the objects of $\Vec(D_{2\cdot3})$. Thus the quasi-monoidal functors associated to $\widehat{F}$ and $\widehat{F}^\rho$ can not be isomorphic even as plain abelian functors, let alone as quasi-monoidal functors.
\end{example}

\section{Existence and classification of tensorators for graded quasi-monoidal equivalences}\label{sec:ten}
In the previous section we classify systems of equivalences for a fixed $\cF_e$ and $\phi$. In this section we now fix $\widehat{F}$ a system of equivalences for $(\cF_e, \phi)$, and aim to classify collections of bimodule natural isomorphisms
\begin{equation}\label{eq:T}     \widehat{T} \eq    \left \{\raisebox{-.5\height}{ \includegraphics[scale = .5]{taugh}} \right \}\end{equation}
such that the bimodule natural automorphism $v(\cF_e,\phi,\widehat{F},\widehat{T})_{f,g,h}$ defined in Figure~\ref{eq:v3} is the identity for all $f,g,h \in G$. It follows from the discussion of Section~\ref{sec:data} that the quadruple $(\cF_e,\phi,\widehat{F},\widehat{T})$ then gives rise to a $\cC$-based equivalence.

\begin{defn}\label{def:systen}
We call a collection of bimodule natural isomorphisms $\widehat{T}$ (as in \ref{eq:T}) such that $v(\cF_e,\phi,\widehat{F},\widehat{T})$ vanishes a \textit{system of tensorators for $(\cF_e,\phi,\widehat{F})$}.
\end{defn}
 We consider system of tensorators up to $\cC$-based natural isomorphism of the corresponding $\cC$-based equivalences. Explicitly this means that we consider two systems of tensorators $\widehat{T}$ and $\widehat{T}'$ the same if and only if there exists a collection $\mu_g \in \mathbb{C}^\times$, such that $\mu_{g,h}\widehat{T}_{g,h} = \mu_g\mu_h \widehat{T}'_{g,h}$.

To begin, let us fix $\cF_e, \phi$, and $\widehat{F}$ a system of tensorators for $(\cF_e, \phi)$. As $\widehat{F}$ is a system of equivalences for $(\cF_e, \phi)$ we can choose a collection of bimodule natural isomorphisms $\widehat{T}$. However there is no guarantee that $v(\cF_e,\phi,\widehat{F},\widehat{T})$ vanishes for our choice. To help us determine when $v(\cF_e,\phi,\widehat{F},\widehat{T})$ vanishes we observe that $v(\cF_e,\phi,\widehat{F},\widehat{T})_{f,g,h}$ is a natural automorphism of the bimodule $H_{f,g,h}$ (as defined in Section~\ref{sec:data}), hence we can identify it with a non-zero complex number. Thus $v(\cF_e,\phi,\widehat{F},\widehat{T})$ is a function $G\times G \times G \to \mathbb{C}^\times$. We can also identify $v(\cF_e,\phi,\widehat{F},\widehat{T})_{f,g,h}$ with a natural automorphism of $\Id_{\cC}$ in the following way:
\begin{align*} V(\cF_e,\phi,\widehat{F},\widehat{T})_{f,g,h} :=&     \operatorname{r}_{\operatorname{R}_{e_{fgh} }}  \cdot [ \id_{\operatorname{R}_{e_{fgh}}} \circ (  \operatorname{r}_{H_{f,g,h}} \boxtimes \id_{\Id_{e_{fgh}}}) \circ \id_{\operatorname{R}^{-1}_{e_{fgh}}}    ]  \\
 &\cdot [ \id_{\operatorname{R}_{e_{fgh}}} \circ ( \id_{H^{-1}_{f,g,h}} \boxtimes \id_{\Id_{e_{fgh}^{op}}}) \circ (  v(\cF_e,\phi,\widehat{F},\widehat{T})_{f,g,h} \boxtimes \id_{\Id_{e_{fgh}^{op}}}) \circ \id_{\operatorname{R}^{-1}_{e_{fgh}}}    ] \\ 
&\cdot [ \id_{\operatorname{R}_{e_{fgh}}} \circ (  \operatorname{R}^{-1}_{H_{f,g,h}} \boxtimes \id_{\Id_{e_{fgh}^{-1}}}) \circ \id_{\operatorname{R}^{-1}_{e_{fgh}}}    ]\cdot  \operatorname{r}^{-1}_{\operatorname{R}_{e_{fgh} }}.
\end{align*}

Graphically we can think of $V(\cF_e,\phi,\widehat{F},\widehat{T})_{f,g,h}$ as a compactification of the foam describing $v(\cF_e,\phi,\widehat{F},\widehat{T})_{f,g,h}$ in Figure~\ref{eq:v3}. Morally we can regard $V(\cF_e,\phi,\widehat{F},\widehat{T})_{f,g,h}$ as a sort of 3-dimensional trace of the 3-morphism $v(\cF_e,\phi,\widehat{F},\widehat{T})$.

As with $v(\cF_e,\phi,\widehat{F},\widehat{T})$, we can regard $V(\cF_e,\phi,\widehat{F},\widehat{T})$ as a function $G\times G \times G \to \mathbb{C}^\times$. In fact when regarded as functions $G\times G \times G \to \mathbb{C}^\times$ we have that $v(\cF_e,\phi,\widehat{F},\widehat{T})$ and $V(\cF_e,\phi,\widehat{F},\widehat{T})$ are equal. We have two representatives for the same function as each representative is easier to work with in certain situations. For example, in the following Lemma we show that $V(\cF_e,\phi,\widehat{F},\widehat{T})$ is a 3-cocycle, which then implies that the corresponding function $G \times G \times G \to \mathbb{C}^\times$ is also a 3-cycle. We use the 3-dimensional graphical calculus for $\BBBrPic(\cC)$ for the proof of this Lemma.

\begin{lem}
We have 
\[ V(\cF_e,\phi,\widehat{F},\widehat{T})_{g,h,k}V(\cF_e,\phi,\widehat{F},\widehat{T})_{f,gh,k}V(\cF_e,\phi,\widehat{F},\widehat{T})_{f,g,h} = V(\cF_e,\phi,\widehat{F},\widehat{T})_{fg,h,k}V(\cF_e,\phi,\widehat{F},\widehat{T})_{f,g,hk}.\] 
\end{lem}
\begin{proof}
Showing that
\[ V(\cF_e,\phi,\widehat{F},\widehat{T})_{g,h,k}V(\cF_e,\phi,\widehat{F},\widehat{T})_{f,gh,k}V(\cF_e,\phi,\widehat{F},\widehat{T})_{f,g,h} = V(\cF_e,\phi,\widehat{F},\widehat{T})_{fg,h,k}V(\cF_e,\phi,\widehat{F},\widehat{T})_{f,g,hk}.\]  
follows somewhat similar to the calculation showing $T(\cF_e,\phi,\widehat{F})$ is a 2-cocycle, except now the proof is in 3 dimensions, so we need to recouple sheets instead of strings. This calculation is in fact somewhat easier, as each $V(\cF_e,\phi,\widehat{F},\widehat{T})_{-,-,-}$ can be contracted to a scaler and moved anywhere in the diagram through linearity.

 The calculation follows by merging both 
\[ V(\cF_e,\phi,\widehat{F},\widehat{T})_{g,h,k}V(\cF_e,\phi,\widehat{F},\widehat{T})_{f,gh,k}V(\cF_e,\phi,\widehat{F},\widehat{T})_{f,g,h}\] 
and 
\[ V(\cF_e,\phi,\widehat{F},\widehat{T})_{fg,h,k}V(\cF_e,\phi,\widehat{F},\widehat{T})_{f,g,hk}\] 
into single components through sheet recoupling. We then cancel 3-morphisms with their inverses, then apply the fact $A^\cD$ and $A^\cE$ are systems of associators to show these two terms are equal. We omit the exact details of this calculation as they are fairly straightforward, but difficult to draw.
\end{proof}

Suppose we have choose two different collections of bimodule natural isomorphisms $\widehat{T}$ and $\widehat{T}'$. Then each $\widehat{T}'_{g,h}$ differs from $\widehat{T}_{g,h}$ by a non-zero complex scaler $\lambda_{g,h}$, that is 
\[     \widehat{T}'_{g,h} = \lambda_{g,h} \widehat{T}_{g,h}.\]
Using linearity it is easy to see that 
\[v(\cF_e,\phi,\widehat{F},\widehat{T}')_{f,g,h} = \lambda_{g,h}\lambda_{f,gh}\lambda_{f,g}^{-1}\lambda^{-1}_{fg,h}v(\cF_e,\phi,\widehat{F},\widehat{T})_{f,g,h} =\partial^3(\lambda) v(\cF_e,\phi,\widehat{F},\widehat{T})_{f,g,h} . \]
Thus $v(\cF_e,\phi,\widehat{F},\widehat{T}')$ differs from $v(\cF_e,\phi,\widehat{F},\widehat{T})$ by a coboundary in $Z^3(G,\mathbb{C}^\times)$. This implies that the cohomology class of $v(\cF_e,\phi,\widehat{F},\widehat{T})$ is independent from the choice of $\widehat{T}$, and only depends on $\cF_e$, $\phi$, and $\widehat{F}$.

\begin{defn}\label{def:o3obs}
We write $o_3(\cF_e, \phi, \widehat{F})$ for the cohomology class of $v(\cF_e,\phi,\widehat{F},\widehat{T})$ in $H^3(G, \mathbb{C}^\times)$.
\end{defn}

From the element $o_3(\cF_e, \phi, \widehat{F}) \in H^3(G, \mathbb{C}^\times)$ we can deduce a necessary and sufficient condition for there to exist a system of tensorators for $(\cF_e, \phi, \widehat{F})$.

\begin{lem}\label{lem:non-emptytorsor2}
Suppose $\widehat{F}$ is a system of equivalences for $(\cF_e,\phi)$. Then there exists a system of tensorators for $(\cF_e,\phi,\widehat{F})$ if and only if $o_3(\cF_e,\phi,\widehat{F})$ is the trivial element of $H^3(G,\mathbb{C}^\times)$.
\end{lem}
\begin{proof}
Suppose $\widehat{T}$ is a system of tensorators for $(\cF_e,\phi,\widehat{F})$. Then by definition $v(\cF_e,\phi,\widehat{F},\widehat{T})$ is the trivial element of $Z^3(G,\mathbb{C}^\times)$, and thus $o_3(\cF_e,\phi,\widehat{F})$ is trivial in $H^3(G,\mathbb{C}^\times)$.

Conversely suppose $o_3(\cF_e,\phi,\widehat{F})$ is trivial in $H^3(G,\mathbb{C}^\times)$. Let $\widehat{T}$ be an arbitrary collection of bimodule natural isomorphisms as in \ref{eq:T}, which we can choose as $\widehat{F}$ is a system of equivalences for $(\cF_e,\phi)$. As $o_3(\cF_e,\phi,\widehat{F})$ is trivial, we have that $v(\cF_e,\phi,\widehat{F},\widehat{T}) = \partial^3(\lambda)$ for some $\lambda \in C^2(G,\mathbb{C}^\times)$. We define a new collection of bimodule natural isomorphisms 
\[  \widehat{T}_{g,h}^{\lambda^{-1}}:=\lambda_{g,h}^{-1}T_{g,h}. \]
Using linearity is is straightforward to see that $v(\cF_e,\phi,\widehat{F},\widehat{T}^{\lambda^{-1}})$ vanishes, and thus $\widehat{T}^{\lambda^{-1}}$ is a system of tensorators for $(\cF_e,\phi,\widehat{F})$.
\end{proof}

Given that we have found exactly when there exists a system of tensorators for $(\cF_e,\phi,\widehat{F})$, we now aim to classify all systems of tensorators for $(\cF_e,\phi,\widehat{F})$.

\begin{lem}\label{lem:tor2}
Suppose $o_3(\cF_e,\phi,\widehat{F})$ is trivial, then systems of tensorators for $(\cF_e,\phi,\widehat{F})$ form a torsor over $H^2(G,\mathbb{C}^\times)$.
\end{lem}

\begin{proof}
Let $\widehat{T}$ be a system of tensorators for $(\cF_e,\phi,\widehat{F})$. We define an action of $\lambda \in Z^2(G,\mathbb{C}^\times)$ on $\widehat{T}$ by 
\[  \widehat{T}^\lambda_{g,h} :=   \lambda_{g,h}\widehat{T}_{g,h}.\]

\noindent We begin by showing this action of $Z^2(G,\mathbb{C}^\times)$ on systems of tensorators for $(\cF_e,\phi,\widehat{F})$ is well defined.
\begin{trivlist}\leftskip=3em 
\item Let $\lambda \in Z^2(G,\mathbb{C}^\times)$, and $\widehat{T}$ a system of tensorators for $(\cF_e,\phi,\widehat{F})$. Using linearity we compute
\[   v(\cF_e,\phi,\widehat{F},\widehat{T}^\lambda) = \partial^3(\lambda) v(\cF_e,\phi,\widehat{F},\widehat{T}) = v(\cF_e,\phi,\widehat{F},\widehat{T})  . \]
As $\widehat{T}$ is a system of tensorators for $(\cF_e,\phi,\widehat{F})$, we have that $v(\cF_e,\phi,\widehat{F},\widehat{T})$ vanishes. Hence $v(\cF_e,\phi,\widehat{F},\widehat{T}^\lambda)$ also vanishes, and thus $\widehat{T}^\lambda$ is a system of tensorators for $(\cF_e,\phi,\widehat{F})$.
\end{trivlist}

\noindent Next we show that the action of $Z^2(G,\mathbb{C}^\times)$ on systems of tensorators for $(\cF_e,\phi,\widehat{F})$ is transitive.
\begin{trivlist}\leftskip=3em 
\item Let $\widehat{T}$ and $\widehat{T}'$ be two systems of tensorators for $(\cF_e,\phi,\widehat{F})$. As explained in the earlier discussion, each $\widehat{T}'_{g,h}$ differs from $\widehat{T}_{g,h}$ by a non-zero complex number, hence $\widehat{T}' = \widehat{T}^\lambda$ for some $\lambda \in C^2(G, \mathbb{C}^\times)$. Using linearity we compute
\[ v(\cF_e,\phi,\widehat{F},\widehat{T}') = \partial^3(\lambda)v(\cF_e,\phi,\widehat{F},\widehat{T}). \]
As $\widehat{T}$ and $\widehat{T}'$ are systems of tensorators for $(\cF_e,\phi,\widehat{F})$, both $v(\cF_e,\phi,\widehat{F},\widehat{T})$ and $v(\cF_e,\phi,\widehat{F},\widehat{T}')$ are trivial. Thus $\partial^3(\lambda)$ is also trivial, and so $\lambda \in  Z^2(G,\mathbb{C}^\times)$. 
\end{trivlist}

\noindent Now we show the action of $Z^2(G,\mathbb{C}^\times)$ on systems of tensorators for $(\cF_e,\phi,\widehat{F})$ descends to a well-defined action of $H^2(G, \mathbb{C}^\times)$.
\begin{trivlist}\leftskip=3em 
\item Let $\widehat{T}$ be a system of tensorators for $(\cF_e,\phi,\widehat{F})$ and suppose $\lambda$ is a coboundary in $Z^2(G,\mathbb{C}^\times)$. Then $\lambda = \partial^2(\mu)$ for some $\mu \in C^1(G, \mathbb{C}^\times)$. We compute 
\[  \widehat{T}^\lambda_{g,h} = \lambda_{g,h}\widehat{T}_{g,h} = \mu_{h}\mu_{gh}^{-1}\mu_{h}\widehat{T}_{g,h}. \]
Thus the systems of tensorators $\widehat{T}$ and $\widehat{T}^\lambda$ are equivalent.
\end{trivlist}

\noindent Finally we show the action of $H^2(G, \mathbb{C}^\times)$ on systems of tensorators for $(\cF_e,\phi,\widehat{F})$ is free.
\begin{trivlist}\leftskip=3em 
\item Let $\lambda \in   Z^2(G, \mathbb{C}^\times)$ and $T$ a system of tensorators for $(\cF_e,\phi,\widehat{F})$. Suppose the systems of tensorators $\widehat{T}$ and $\widehat{T}^\lambda$ are equivalent. Then there exists $\mu \in C^1(G,\mathbb{C}^\times)$ such that 
\[  \mu_{gh} \widehat{T}_{g,h} = \mu_g\mu_h \widehat{T}^\lambda_{g,h} =  \mu_g\mu_h\lambda_{g,h} \widehat{T}_{g,h}. \]
Therefore $\lambda_{g,h} =  \mu_g\mu_{gh}^{-1}\mu_h$, which gives that $\lambda$ is a coboundary. 
\end{trivlist}

Putting everything together we have that systems of tensorators for $(\cF_e,\phi,\widehat{F})$ form a torsor over $H^2(G,\mathbb{C}^\times)$.
\end{proof}

\section{Classification of $\cC$-based equivalences}\label{sec:together}
In this section we complete our classification of $\cC$-based equivalences between two $G$-graded extensions of a fusion category $\cC$.

\begin{thm}\label{thm:mainclass}
Let $\cD$ and $\cE$ be two $G$-graded extensions of a fusion category $\cC$, corresponding to triples $(d, M^\cD, A^\cD)$ and $(e, M^\cE, A^\cE)$ respectively. We have that $\cC$-based equivalences between $\cD$ and $\cE$ are classified (up to $\cC$-based natural isomorphism) by quadruples 
\[    (\cF_e, \phi, \widehat{F}, \widehat{T}),\]
where
\begin{itemize}
\item $\cF_e$ is a monoidal auto-equivalence of $\cC$,
\item $\phi$ is an automorphism of the group $G$,
\item $\widehat{F}$ is a system of equivalences for $(\cF_e, \phi)$, as defined in Definition~\ref{def:syseq},
\item $\widehat{T}$ is a system of tensorators for $(\cF_e, \phi,\widehat{F})$, as defined in Definition~\ref{def:systen}, such that the obstruction $v(\cF_e, \phi,\widehat{F}, \widehat{T})_{f,g,h}$, as defined in Equation~\eqref{eq:v3}, is trivial for all $f,g,h \in G$.
\end{itemize}

Further, we have that system of equivalences for $(\cF_e, \phi)$ form a torsor over $Z^1(G,\Inv(\cZ(\cC)))$ if and only if the two homomorphisms $G \to \BrPic(\cC)$ given by
\[ \operatorname{Inn} (  _{\cF_e}\cC ) \circ d \circ \phi \quad \text{ and } \quad  e\]    
are equal, and the cohomology class 
\[ o_2(\cF_e, \phi) \in H^2(G, \Inv(\cZ(\cC)))\]
defined in Definition~\ref{def:o2obs} is trivial. 

Furthermore, we have that systems of tensorators for $(\cF_e, \phi,\widehat{F})$ form a torsor over $H^2(G, \mathbb{C}^\times)$ if and only if the cohomology class 
\[ o_3(\cF_e, \phi, \widehat{F}) \in H^3(G, \mathbb{C}^\times)\]
defined in Definition~\ref{def:o3obs} is trivial. 
\end{thm}
Recall $\operatorname{Inn} (  _{\cF_e}\cC )$ is the inner automorphism of the group $\BrPic(\cC)$ induced from the invertible bimodule $_{\cF_e}\cC$.
\begin{proof}
It follows from Section~\ref{sec:data} that $\cC$-based equivalences $\cD\to \cE$ are classified (up to $\cC$-based natural isomorphism) by quadruples $(\cF_e, \phi, \widehat{F}, \widehat{T})$, where $\cF_e$ is a monoidal auto-equivalences of $\cC$, $\phi$ is a group automorphism of $G$, $\widehat{F}$ is a system of equivalences for $(\cF_e, \phi)$, and $\widehat{T}$ is a system of tensorators for $(\cF_e, \phi,\widehat{F})$. Thus to classify $\cC$-based equivalences $\cD\to \cE$ we need to classify such quadruples.

Let $\cF_e$ be any monoidal auto-equivalence of $\cC$, and $\phi$ an automorphism of the group $G$. Lemma~\ref{lem:non-emptytorsor1} says we can pick a system of equivalences for $(\cF_e, \phi)$ if and only if the two homomorphisms $G \to \BrPic(\cC)$ given by
\[ \operatorname{Inn} (  _{\cF_e}\cC ) \circ d \circ \phi \quad \text{ and } \quad  e\]    
are equal, and if the cohomology class $ o_2(\cF_e, \phi) \in H^2(G, \Inv(\cZ(\cC)))$ is trivial. Once we can pick a system of equivalences for $(\cF_e, \phi)$, Lemma~\ref{lem:tor1} then gives that systems of equivalences for $(\cF_e, \phi)$ form a torsor over $Z^1(G, \Inv(\cZ(\cC)))$.

Let $\widehat{F}$ be a system of equivalences for $(\cF_e, \phi)$. Lemma~\ref{lem:non-emptytorsor2} says we can pick a system of tensorators for  $(\cF_e, \phi, \widehat{F})$ if and only if the cohomology class $ o_3(\cF_e, \phi,\widehat{F}) \in H^3(G, \mathbb{C}^\times)$ is trivial. Once we can pick a system of tensorators for $(\cF_e, \phi, \widehat{F})$, Lemma~\ref{lem:tor2} then gives that systems of tensorators for $(\cF_e, \phi,\widehat{F})$ form a torsor over $H^2(G, \mathbb{C}^\times)$.
\end{proof}

\section{Applications}\label{sec:examples}
In this final section we explore the applications of our main classification Theorem~\ref{thm:mainclass}. Unfortunately our classification result is hard to use in practice due to the obstruction $o_3$, which requires knowledge of the entire Brauer-Picard 3-category to compute. The obstruction $o_3$ lives in the cohomology group $H^3(G, \mathbb{C}^\times)$, which is only trivial when $G$ is the trivial group, thus we can't use cohomology tricks to show the obstruction $o_3$ vanishes, and have to explicitly compute it. For this section we focus on applications of Theorem~\ref{thm:mainclass} that avoid having to compute the obstruction $o_3$. We find many interesting and practical applications.

\subsection*{Gauge auto-equivalences of a $G$-graded category}

For a given fusion category $\cC$ it is an important question to compute the group of monoidal auto-equivalences of $\cC$. A difficult obstruction to computing this group is the existence of monoidal equivalences that fix all objects up to isomorphism. Such monoidal auto-equivalences are known as \textit{gauge auto-equivalences of $\cC$}. For all known fusion categories it is expected that the group of gauge auto-equivalences of that category is isomorphic to $H^2(G, \mathbb{C}^\times)$, where $G$ is the universal grading group of that category. However proving this fact for any example is difficult in practice, and new tools are desperately needed to determine the gauge auto-equivalences of a category.

Suppose $\cD$ is a $G$-graded extension of $\cC$, and $\cF$ is a gauge auto-equivalence of $\cD$. As $\cF$ fixes all objects up to isomorphism, it follows that $\cF$ is a $\cC$-based auto-equivalence of $\cD$, and hence corresponds to a quadruple as in Theorem~\ref{thm:mainclass}. This fact allows us to prove the following Theorem regarding the gauge auto-equivalences of $\cD$.

\begin{thm}
Let $\cD = \bigoplus \cD_g$ be a $G$-graded extension of $\cC$. Suppose that the following three conditions are satisfied
\begin{itemize}
\item  $\cC$ has no non-trivial gauge auto-equivalences,
\item  If $z \in \Inv(\cC)$ is a fixed point for some $\cD_g$, then $z = \mathbf{1}$,
\item  $\ad(\cC) \simeq \cC$.
\end{itemize}
Then the gauge auto-equivalence group of $\cD$ is isomorphic to $ H^2(G, \mathbb{C}^\times)$.
\end{thm}

\begin{proof}
Let $\cF$ be a gauge auto-equivalence of $\cD$, then under Theorem~\ref{thm:mainclass} $\cF$ corresponds to a quadruple $(\cF_e, \phi, \widehat{F}, \widehat{T})$. As $\cF$ is a gauge auto-equivalence of $\cD$, it must restrict to a gauge auto-equivalence on $\cC$, thus $\cF_e$ must be trivial. Further we can also deduce that $\phi$ must be the identity automorphism of $G$. 

We now wish to determine the bimodule equivalences $\widehat{F}_g : \cD_g \to \cD_g$. From Lemma~\ref{lem:anue} we have $\widehat{F}_g = (z_g,\gamma^{z_g})    \boxtimes \Id_{\cD_g}$, with each $(z_g,\gamma^{z_g}) \in \Inv(\cZ(\cC))$. As $\cF$ is a gauge auto-equivalence, we must have that each $\widehat{F}_g$ fixes the objects of $\cD_g$, and thus each $(z_g,\gamma^{z_g})$ forgets to an object of $\Inv(\cC)$ that fixes the objects of $\cD_g$, which implies that $z_g = \mathbf{1}$. As $\ad(\cC) \simeq \cC$ there is a unique half-braiding on the unit \cite{MR3354332}, thus $(z_g,\gamma^{z_g})$ is the trivial object of $ \Inv(\cZ(\cC))$, and so $\widehat{F}_g = \Id_{\cD_g}$.

The existence of the identity auto-equivalence for $\cD$ gives a system of tensorators for $(\Id_\cC, \id_G, \{   \Id_{\cD_g} \})$, and thus $\widehat{T}$ must be a twist of this system of tensorators by an element of $H^2(G, \mathbb{C}^\times)$. Conversely, any element of $H^2(G, \mathbb{C}^\times)$ gives a new system of tensorators. Hence we have shown that gauge auto-equivalences of $\cD$ correspond to elements of $H^2(G, \mathbb{C}^\times)$. 
\end{proof}

An important class of fusion categories are the \textit{categories of level $k$ integrable representations of an affine Lie algebra $\widehat{\mathfrak{g}}$}, which we denote as $\cC(\mathfrak{g}, k)$. We direct the reader to \cite{1810.09055} for more details. Using the above Theorem we are able to show that if $\mathfrak{g}$ is a Lie algebra of type $A_n, B_n, C_n, D_{2n+1}$, or $G_2$, then the category $\cC(\mathfrak{g}, k)$ has no non-trivial gauge auto-equivalences, and if $\mathfrak{g}$ is a Lie algebra of type $D_{2n}$, then the gauge auto-equivalence group of $\cC(\mathfrak{g}, k)$ is isomorphic to $\Z{2}$. We will prove this result in a future publication, and explore the consequences of this result to the program to classify quantum subgroups of the simple Lie groups.

\subsection*{Classification of pointed cyclic categories up to monoidal equivalence}
Let $G$ be a finite group. It is a well-known result that categories with $\Vec(G)$ fusion rules are classified by elements of $H^3(G, \mathbb{C}^\times)$. However this classification is up to equivalence of $G$-graded extensions of $\Vec$, and not up to monoidal equivalence. We can apply Theorem~\ref{thm:mainclass} to see that categories with $\Vec(G)$ fusion rules are classified up to monoidal equivalence by the set
\[    H^3(G, \mathbb{C}^\times) / \Aut(G).\]
While this classification is interesting in theory, it isn't so great in practice as the action $\Aut(G)$ on $H^3(G, \mathbb{C}^\times)$ requires explicit 3-cocyles to compute.

Let us focus on the case $G = \Z{N}$ where we can say something concrete. It is well-known that
\[ \Z{N}\cong H^3(G, \mathbb{C}^\times),\]
with an explicit isomorphism given by
\[    n \mapsto \omega_n,\]
with 
\[ \omega_n(a,b,c) = e^{ \frac{2i \pi n}{N^2}a(b + c  -  [ ( b + c    ) \text{ mod } N ]     }.\]

Let $p \in (\Z{N})^\times = \Aut(\Z{N})$. A direct computation shows that 
\[   \omega_n(p\cdot a, p\cdot b, p\cdot c) = \omega_{p^2\cdot n}(a,b,c).\]

Thus 
\[  H^3(G, \mathbb{C}^\times) / \Aut(\Z{N})\cong (\Z{N} ) /  \{ p^2 : p \in  (\Z{N})^\times  \}   \]
and so categories with $\Vec(\Z{N})$ fusion rules are classified up to monoidal equivalence by the set $ (\Z{N} ) /  \{ p^2 : p \in  (\Z{N})^\times  \}$. Computing the size of this classifying set for a fixed $N$ is a straightforward task in modular arithmetic. From this classifying set we conjecture the following formula for $\Omega(N)$, the number of monoidal equivalence classes of fusion categories with $\Vec(\Z{N})$ fusion rules:
\[    \Omega(N= p_1^{k_1}p_2^{k_2}\cdots p_M^{k_M}) = \gamma_{p_1}(k_1)\gamma_{p_2}(k_2)\cdots \gamma_{p_M}(k_M),\]
where
\[   \gamma_p(k) = \begin{cases}
2    & \text{ if } p = 2 \text{ and } k = 1\\
3     & \text{ if } p \neq 2 \text{ and } k = 1\\
4(k-1)  &   \text{ if } p = 2 \text{ and } k \neq 1\\
2k + 1    &  \text{ if } p \neq 2 \text{ and } k  \neq 1.
\end{cases}\]
A computer verifies this conjectured formula for $N \leq 5000$.

\subsection*{An action of $\Aut_\otimes(\cC) \times \Aut(G)$ on $\Ext$}

Recall that $\Ext$ is the set of all $G$-graded extensions of $\cC$, up to equivalence of extensions. As demonstrated by Example~\ref{ex:Z9}, equivalence of extensions is a stronger notion than monoidal equivalence. Thus $\Ext$ over-counts the number of monoidally inequivalent $G$-graded extensions of $\cC$, which is an issue when trying to use the classification of $G$-graded extensions for other classification results. While Theorem~\ref{thm:mainclass} gives a way in theory to determine when two $G$-graded extensions of $\cC$ are $\cC$-based equivalent (and hence monoidally equivalent), as we have already said, it is near impossible to use this Theorem in practice. The aim of this subsection is to find an easier way to determine when two $G$-graded extensions of $\cC$ are monoidally equivalent, rather than directly applying Theorem~\ref{thm:mainclass}.

We begin by defining a group action on the set $\Ext$. Let $\cH \times \psi \in \Aut_\otimes(\cC)\times \Aut(G)$, and $\cD \in \Ext$. By the classification of $G$ graded extensions of $\cC$ we have that $\cD$ corresponds to a triple $(d, M , A)$ as in Section~\ref{sec:gradedex}. We define
\begin{align*}
d^{\cH\times \psi } :=& \operatorname{Inn}(_{\cH}\cC) \circ d \circ \psi, \\
M_{g,h}^{\cH\times \psi } :=&     (\Id_{\cC_\cH}\boxtimes M_{\psi(g),\psi(h)} \boxtimes \Id_{_{\cH}\cC})\circ (\Id_{\cC_\cH}\boxtimes \Id_{c_{\psi(g)}} \boxtimes \operatorname{R}_{_\cH\cC}\boxtimes \Id_{c_{\psi(h)}} \boxtimes \Id_{_{\cH}\cC}), \\ 
A^{\cH\times \psi }_{f,g,h} :=& [\id_{\Id_{\cC_\cH}} \boxtimes A_{\psi(f),\psi(g),\psi(h)} \boxtimes \id_{\Id_{_{\cH}\cC}}][\id_{\Id_{\cC_\cH}} \boxtimes \id_{\Id_{c_{\psi(f)}}} \boxtimes \id_{\operatorname{R}_{_\cH\cC}} \boxtimes \id_{\Id_{c_{}\psi(g)}}  \boxtimes \id_{\operatorname{R}_{_\cH\cC}} \boxtimes \id_{\Id_{c_h}} \boxtimes \id_{\Id_{\cC_\cH}}],
\end{align*}
where $\operatorname{R}_{_\cH\cC}$ is the specific bimodule equivalence $  _\cH\cC \boxtimes \cC_\cH \to \cC$ from part (2) of Theorem~\ref{lem:untwisting}, and $\operatorname{Inn}(_{\cH}\cC)$ is the inner automorphism of $\BrPic(\cC)$ defined in Remark~\ref{rem:inn}.

In the graphical calculus we have:
\[ M_{g,h}^{\cH\times \psi } = \raisebox{-.5\height}{ \includegraphics[scale = .6]{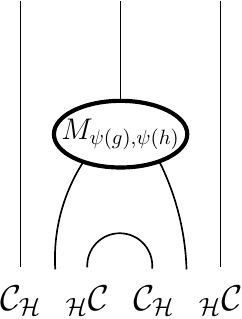}} \]
and
\[ A^{\cH\times \psi }_{f,g,h}= \raisebox{-.5\height}{ \includegraphics[scale = .6]{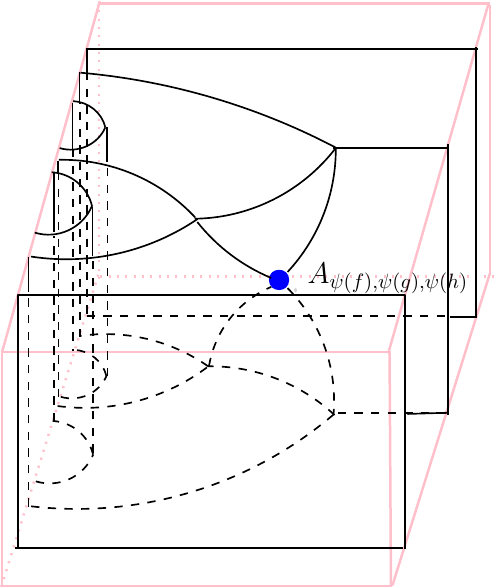}}.  \]

\begin{lem}\label{lem:Gact}
The map
\[    (\cH\times \psi) \cdot (d, M , A) \mapsto (d^{\cH\times \psi }, M^{\cH\times \psi } , A^{\cH\times \psi })\]
gives a well-defined action of $\Aut_\otimes(\cC) \times \Aut(G)$ on $\Ext$.
\end{lem}
\begin{proof}
It is straightforward to verify that
\[     (\cH_2\times \psi_2) \cdot[ (\cH_1\times \psi_1) \cdot (d, M , A)] =[ (     \cH_2\circ \cH_1 )\times(\psi_2\circ \psi_1) ]\cdot (d, M , A).\]

Hence all we need to show is that the triple $(d^{\cH\times \psi }, M^{\cH\times \psi } , A^{\cH\times \psi })$ gives a $G$-graded extension of $\cC$. This is equivalent to showing that the obstruction $v( d^{\cH\times \psi }, M^{\cH\times \psi } , A^{\cH\times \psi })_{f,g,h,k}$ from Figure~\ref{fig:v4} is trivial for all $f,g,h,k \in G$. Using the graphical calculus, and some simple isotopies we can reduce the vanishing of $v( d^{\cH\times \psi }, M^{\cH\times \psi } , A^{\cH\times \psi })_{f,g,h,k}$ to the vanishing of $v( d, M , A)_{f,g,h,k}$. As the triple $(d,M,A)$ corresponds to to a $G$-graded extension of $\cC$, we have $v( d, M , A)_{f,g,h,k}$ is trivial for all $f,g,h,k \in G$. Thus the triple $(d^{\cH\times \psi }, M^{\cH\times \psi } , A^{\cH\times \psi })$ gives a $G$-graded extension of $\cC$.
\end{proof}

We now present the main result of this paper, Theorem~\ref{thm:mainprac}, which shows that any two $G$-graded extensions that are in the same orbit of the action of $\Aut_\otimes(\cC) \times \Aut(G)$, are in fact monoidally equivalence, despite possibly being different as extensions of $\cC$.

\begin{proof}[Proof of Theorem~\ref{thm:mainprac}]
Fix a $G$-graded extension of $\cC$ corresponding to a triple $(d,M,A)$, and $\cH \times \psi \in \Aut_\otimes(\cC) \times \Aut(G)$. We will produce a $\cC$-based equivalence between two the $G$-graded extensions associated to the triples $ (d^{\cH\times \psi }, M^{\cH\times \psi } , A^{\cH\times \psi })$ and $(d,M,A)$. To produce such a $\cC$-based equivalence we will apply Theorem~\ref{thm:mainclass} and produce a quadruple $(\cF_e, \phi, \widehat{F}, \widehat{T})$ such that the obstruction $v(\cF_e, \phi, \widehat{F}, \widehat{T} )_{f,g,h}$ vanishes for all $f,g,h \in G$.

Our choice for such a quadruple is:
\begin{align*}
\cF_e :=& \quad \cH \\
\phi :=& \quad \psi \\
\widehat{F}_g :=& \quad \Id_{\cC_\cH}\boxtimes \Id_{d_{\psi(g)}} \boxtimes \Id_{_\cH\cC} \\
\widehat{T}_{g,h} :=& \quad (\id_{\Id_{\cC_\cH}} \boxtimes \id_{M_{\psi(g),\psi(h)}} \boxtimes \id_{\Id_{\cH\cC}} ) \circ (\id_{\Id_{\cC_\cH}} \boxtimes \id_{d_{\psi(g)}} \boxtimes \operatorname{r}_{R_{_\cH\cC}} \boxtimes \id_{d_{\psi(h)}}\boxtimes \id_{\Id_{_\cH\cC}})
\end{align*}
In the graphical calculus we have:
\[  \widehat{F}_g \eq \raisebox{-.5\height}{ \includegraphics[scale = .6]{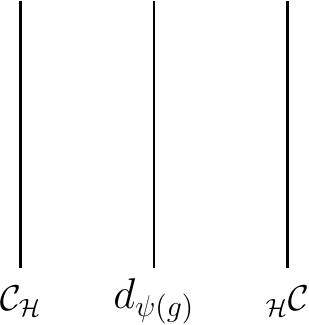}} \]
and 
\[ \widehat{T}_{g,h} \eq \raisebox{-.5\height}{ \includegraphics[scale = .6]{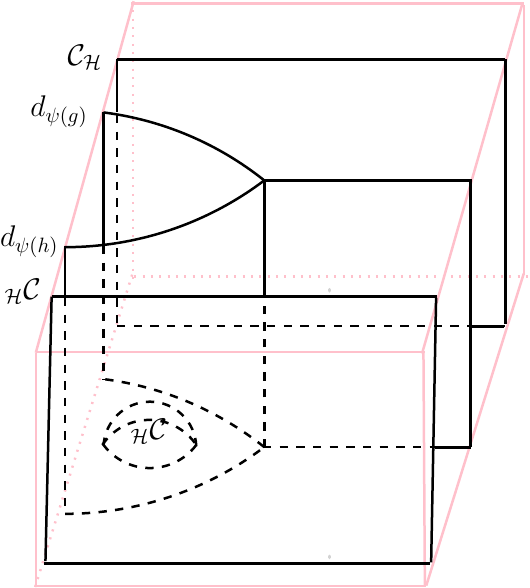}} .\]

We verify that $v(\cF_e, \phi, \widehat{F}, \widehat{T} )_{f,g,h}$ vanishes using the graphical calculus. This computation simply consists of the popping of two $\operatorname{r}_{R_{_\cH\cC}}\cdot \operatorname{r}^{-1}_{R_{_\cH\cC}}$ bubbles, and cancelling $A_{\psi(f),\psi(g),\psi(h)}$ with its inverse.
\end{proof}

This Theorem is still somewhat hard to apply in practice, as computing the action of $\Aut_\otimes(\cC) \times \Aut(G)$ on $M$ and $A$ involves some knowledge of the higher layers of $\BBBrPic(\cC)$. The aim of this paper was to develop easy to use methods to determine when two $G$-graded extensions of $\cC$ are monoidally equivalence. Towards this goal we introduce Corollaries~\ref{cor:cor1} and \ref{cor:cor2} of Theorem~\ref{thm:mainprac}. These two corollaries give us practical methods to determine a useful upper bound on the number of $G$-graded extensions of $\cC$, up to monoidal equivalence. To apply these two corollaries we only require knowledge of the groups $\BrPic(\cC)$ and $\Inv(\cZ(\cC))$ along with information about how $\BrPic(\cC)$ acts on $\Inv(\cZ(\cC))$. This information is computable for large classes of fusion categories, while the higher structure of $\BBBrPic(\cC)$ is only known for the simplest of cases.

\subsection*{Auto-equivalences of bosionic and fermionic modular categories}

Let $(\cC,f)$ be a modular category $\cC$ along with a distinguished boson or fermion $f$. The category $\cC$ has a $\Z{2} = \{+, -\}$ grading, with 
\begin{align*}
\cC_+ :=& \{ X \in \cC : \operatorname{br}_{X,f} = +\operatorname{br}_{f,X}^{-1} \} \\
\cC_- :=& \{ X \in \cC : \operatorname{br}_{X,f} = -\operatorname{br}_{f,X}^{-1}\}.
\end{align*}
The modularity of $\cC$ ensures that this grading is faithful. Using Theorem~\ref{thm:mainclass} we can construct a monoidal auto-equivalence of the category $\cC$ from the object $f$.

\begin{thm}\label{lem:fermod}
Let $(\cC,f)$ be a modular category $\cC$ along with a distinguished boson or fermion $f$. There is a monoidal auto-equivalence $\cF_f$ of $\cC$ defined by
\begin{align*}
\cF_f|_{\cC_+} :=& \Id_{\cC_+} \\
\cF_f|_{\cC_-} :=&  f \otimes \Id_{\cC_-},
\end{align*}
with tensorator
\begin{align*}
\tau_{X,Y} := \begin{cases}  \id_X\otimes \id_Y &\text{ if } X\in \cC_+ \text{ and } Y\in \cC_+, \\
			        \id_f \otimes \id_X\otimes \id_Y &\text{ if } X\in \cC_-\text{ and }Y \in \cC_+, \\
			         \operatorname{br}_{X,f}\otimes \id_Y &\text{ if } X\in \cC_+ \text{ and } Y \in \cC_-, \\
				 (\ev_f \otimes \id_X \otimes \id_Y   )(\id_f \otimes \operatorname{br}_{X,f} \otimes \id_{Y})&\text{ if } X\in \cC_-\text{ and } Y \in \cC_-.\end{cases}
\end{align*}
Furthermore this auto-equivalence is braided if and only if $f$ is a fermion.
\end{thm} 
\begin{proof}
To construct the monoidal auto-equivalence $\cF$ we regard $\cC$ as a $\Z{2}$-graded extension of $\cC_+$ and apply Theorem~\ref{thm:mainclass} and construct a quadruple $(\cF_e, \phi, \widehat{F}, \widehat{T})$ such that the obstruction $v(\cF_e, \phi, \widehat{F}, \widehat{T})_{f,g,h}$ from Figure~\ref{eq:v3} vanishes for all $f,g,h \in  \{+,-\}$.

We choose $\cF_e = \Id_{\cC_+}$ and $\phi = \id_{\Z{2}}$. To choose $\widehat{F}$ we note that the object $f$ lives in $\cC_+$, and lifts to $\cZ(\cC)$ via the braiding on $\cC_+$. Thus $(f, \operatorname{br}_{f,-})$ is an element of $\cZ(\cC)$. We then set
\begin{align*}
\widehat{F}_+ &= \Id_{\cC_+}, \\
\widehat{F}_- &= (f, \operatorname{br}_{f,-}) \boxtimes \Id_{\cC_-}.
\end{align*}
Our choice for the collection $\widehat{T}$ are the bimodule natural isomorphisms induced via Equation~\eqref{eq:funball} from the following $\cC_+$-balanced natural isomorphisms
 \begin{align*}  
 \tau_{X_+, Y_+ } &= \id_{X_+}\otimes \id_{Y_+}  &:   {\otimes}\circ (\widehat{F}_+\boxtimes \widehat{F}_+)\to \widehat{F}_+ \circ {\otimes}    ,\\
\tau_{X_-, Y_+ } &=        \id_f \otimes \id_{X_-}\otimes \id_{Y_+} &:   {\otimes}\circ (\widehat{F}_-\boxtimes \widehat{F}_+)\to \widehat{F}_- \circ {\otimes}  , \\
\tau_{X_+, Y_-} &=         \operatorname{br}_{X_+,f}\otimes \id_{Y_-} &:   {\otimes}\circ (\widehat{F}_+\boxtimes \widehat{F}_-)\to \widehat{F}_- \circ {\otimes}  , \\
\tau_{X_-, Y_- } &=	 (\ev_f \otimes \id_{X_-} \otimes \id_{Y_-}   )(\id_f \otimes \operatorname{br}_{X_-,f} \otimes \id_{Y_-}) &:   {\otimes}\circ (\widehat{F}_-\boxtimes \widehat{F}_-)\to \widehat{F}_+ \circ {\otimes} .
\end{align*}

We then directly verify that obstruction $v(\cF_e, \phi, \widehat{F}, \widehat{T})_{f,g,h}$ is trivial for all $f,g,h \in \{+,-\}$. Note that verifying this obstruction is trivial is exactly just verifying the hexagon equation for all $X,Y,Z \in \cC$. This computation reduces to showing the following braids are equal for $X_-,Y_-, Z_- \in \cC_-$:
\begin{center} \raisebox{-.5\height}{ \includegraphics[scale = .6]{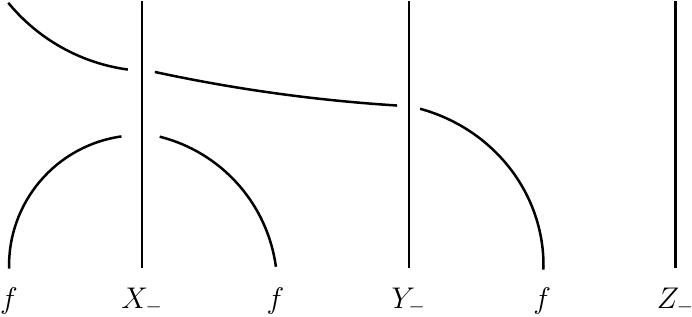}} \qquad and \qquad \raisebox{-.5\height}{ \includegraphics[scale = .6]{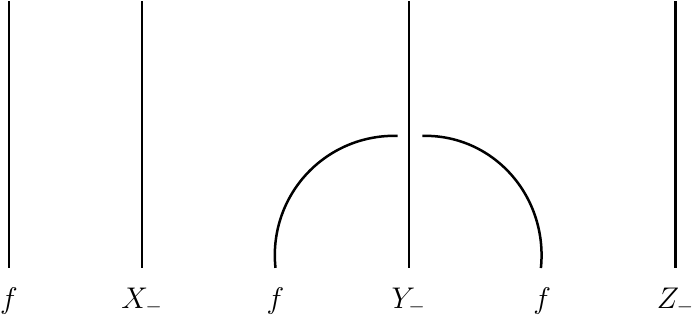}}.   \end{center}
As $f$ is a boson or fermion, we can recouple the two adjacent $f$ strands, giving equality of the two braids above.

All that is left to prove is that $\cF_f$ is braided if and only if $f$ is a fermion. We directly compute for $X_-,Y_- \in \cC_-$,

\resizebox{1 \textwidth}{!}{ $\tau_{Y_-,X_-}\circ \operatorname{br}_{\cF_f(X_-),\cF_f(Y_-)} =  \raisebox{-.5\height}{ \includegraphics[scale = .4]{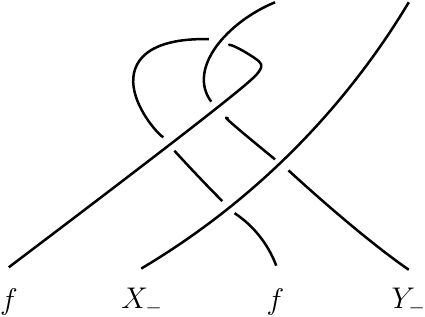}}       = -\raisebox{-.5\height}{ \includegraphics[scale = .4]{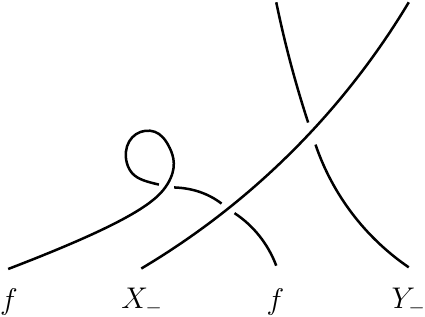}} = -t_f \raisebox{-.5\height}{ \includegraphics[scale = .4]{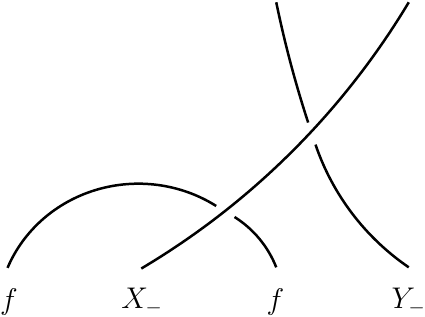}} = -t_f  \cF_f( \operatorname{br}_{X_-,Y_-})\circ \tau_{X_-,Y_-}.$}
Thus $\cF_z$ is braided if and only if $t_f= -1$, i.e. $f$ is a fermion.
\end{proof}

To illustrate the usefulness of the above Theorem, we work the following example.

\begin{example}
Let $k \in \mathbb{N}$ and consider $\cC(\mathfrak{sl}_2, k)$, the modular tensor category of level $k$ integrable representations of $\widehat{\mathfrak{sl}_2}$ (see \cite{1810.09055}). This modular category always has a non-trivial invertible object $(k)$, which is a boson if $k \equiv 0 \pmod 4$, and is a fermion if $k \equiv 2 \pmod 4$. Therefore an application of Theorem~\ref{lem:fermod} gives auto-equivalences of these modular categories. These auto-equivalences are non-trivial when $k>2$. A planar algebra argument shows there are no non-trivial gauge auto-equivalences of $\cC(\mathfrak{sl}_2, k)$, and a simple fusion rule argument shows the only non-trivial auto-equivocalness of $\cC(\mathfrak{sl}_2, k)$ are the ones constructed by Theorem~\ref{lem:fermod}. Thus we get the following result:
\begin{center}\begin{tabular}{c|c|c}
$k\pmod 4$ & $\operatorname{Aut}_\otimes (\cC(\mathfrak{sl}_2, k))$ & $\operatorname{Aut}_\otimes^\text{br} ( \cC(\mathfrak{sl}_2, k))$ \\
\hline
$0$ & $\Z{2}$ & $\{e\}$ \\
$1$ & $\{e\} $ & $\{e\}$ \\ 
$2$ & $\Z{2}$ & $\Z{2}$ \\
$3$ & $\{e\} $ & $\{e\}$ 
\end{tabular}\end{center}
This recovers the authors result from \cite{MR3808052}, with a far more elegant and simpler proof.
\end{example}

The process of constructing a fusion ring automorphism of a modular tensor category from an invertible object has long been known to physicists as \textit{simple current automorphisms} \cite{MR1887583}. Theorem~\ref{lem:fermod} considers the simplest non-trivial invertible object possible in a modular tensor category, an element of order 2, and shows that the associated simple current automorphism lifts to an auto-equivalence of the modular tensor category. Further, it gives conditions for when the associated auto-equivalence is braided or not, which is additional information that is not not be seen at the level of fusion ring automorphisms.

It seems reasonable to expect that all simple current automorphisms lift to auto-equivalences of the categories. Unfortunately we are unable to extend the proof of Lemma~\ref{lem:fermod} to an arbitrary symmetric object of order $M$. Such an extension of Lemma~\ref{lem:fermod} would have many exciting applications, such as to the classification of quantum subgroups of $\cC(\mathfrak{sl}_N, k)$.

\appendix

\section{Proof of Theorem~\ref{lem:untwisting}}\label{app:proof}
\begin{remark}
The following proof contains lengthy equations involving twisted bimodule functors and twisted natural transformations. To assist with readability we highlight in blue the relevant parts of each equation we change at each step of the equation. We make the following conventions regarding notation:
\begin{itemize}
\item We assume $\cC$ is a strict fusion category.

\item Throughout the proof $\cF$ will be a fixed monoidal auto-equivalence of $\cC$, with structure isomorphisms $\tau^\cF$.

\item If $\cW : \cM_1 \times \cM_2 \to \cN$ is a $\cC$ balanced functor, we will write $b^\cW$ for the balancing isomorphisms of $\cW$. If $\cW : \cM_1 \times  \cdots \cM_N \to \cN$ is a $\cC$ balanced functor, we will write ${b^\cW}^i$ for the balancing isomorphism of $\cW$ in the $i$-th position.

\item We will simply write $\cH$ for a twisted (or untwisted) bimodule equivalence, suppressing the structure isormorphisms. We will write $R^\cH$ and $L^\cH$ to refer to these surpressed structure isomorphisms for $\cH$.

\item Let $\cM_1$ and $\cM_2$ invertible $\cC$-bimodules. We write $B_{\cM_1,\cM_2}$ for the canonical $\cC$-balanced functor $\cM_1 \times \cM_2\to \cM_1\boxtimes \cM_2$.

\end{itemize}
\end{remark}
\begin{proof}
\textbf{Proof of statement~\ref{thm:part1}:}

Our goal is to construct a bijection between the set of $\cF$-twisted functors $\cM \to \cN$, and the set of untwisted bimodule functors $\cC_\cF \boxtimes \cM \boxtimes _\cF\cC \to \cN$. We begin by constructing a bijection between the set of $\cF$-twisted functors $\cM \to \cN$, and the set of $\cC$-balanced bimodule functors $\cC_\cF \times \cM \times _\cF\cC \to \cN$.

Let $\cH : \cM \to \cN$ a $\cF$-twisted bimodule functor. We construct a new functor 
\[ \overline{\cH} : \cC_\cF \times \cM \times _\cF\cC \to \cN\]
 by
\[
X \times m \times Y \mapsto X\triangleright \cH(m) \triangleleft Y.
\]
The functor $\overline{\cH}$ is $\cC$-balanced in the first and second position respectively via the maps:
\begin{align*}
{b^{\cH}_{X,Z,m,Y}}^{1} &:= \id_X \triangleright L^\cH_{Z,m} \triangleleft \id_Y: X\triangleright \cF(Z) \triangleright \cH(m) \triangleleft Y \to X\triangleright \cH(Z\triangleright m) \triangleleft Y\\
{b^\cH_{X,m,Z,Y}}^2 &:= \id_X \triangleright R^\cH_{m,Z} \triangleleft \id_Y: X \triangleright \cH(m)\triangleleft \cF(Z) \triangleleft Y \to X\triangleright \cH( m\triangleleft Z) \triangleleft Y.
\end{align*}
We directly verify that the above maps satisfies the condition to be balanced in the first position, but leave the other conditions to be checked by the reader.
\begin{align*}
{b^\cH_{X,Z_1\otimes Z_2,m,Y}}^{1}\overline{\cH}(\id_X \triangleright \tau^\cF_{Z_1,Z_2} \times \id_m \times \id_Y) &= \highlight{(\id_X \triangleright L^\cH_{Z_1\otimes Z_2,m} \triangleleft \id_Y)(\id_X \triangleright \tau^F_{Z_1,Z_2} \triangleright \id_m \triangleleft \id_Y)} \\
&= (\id_X \triangleright \id_{\cF(Z_1)} \triangleright L^\cH_{Z_2,m} \triangleleft \id_Y)(\id_X \triangleright L^\cH_{Z_1,Z_2\triangleright m} \triangleleft \id_Y)\\
&= {b^\cH_{X\triangleleft Z_1,Z_2,m,Y}}^{1}{b^\cH_{X,Z_1,Z_2 \triangleright m,Y}}^{1}.
\end{align*}
Thus the map $\cH \mapsto \overline{\cH}$ gives a map from the set of $\cF$-twisted functors $\cM \to \cN$ to the set of $\cC$-balanced bimodule functors $\cC_\cF \times \cM \times _\cF\cC \to \cN$.


We now construct an inverse to the map $\cH \to \overline{\cH}$. Let
\[  \cJ: \cC_\cF \times \cM \times _\cF\cC \to \cN,\]
 a $\cC$-balanced functor. We define a functor $\underline{\cJ}:\cM \to \cN$ by
\begin{equation*}
\underline{\cJ}(m) := \cJ(\mathbf{1}, m , \mathbf{1}).
\end{equation*}
The functor $\underline{\cJ}$ has the structure of a $\cF$-twisted bimodule functor via the following maps:
\begin{align*}
L^{\underline{\cJ}}_{Z,m}&:= {b^{\cJ}_{\mathbf{1},Z,m,\mathbf{1}}}^{1} L^{\cJ}_{\cF(Z),\mathbf{1},m,\mathbf{1}}\text{ and } \\
R^{\underline{\cJ}}_{m,Z}&:=  {({b^{\cJ}_{\mathbf{1},m,Z,\mathbf{1}}}^1)}^{-1} R^{\cJ}_{\mathbf{1},m,\mathbf{1},\cF(Z)}.
\end{align*}
We leave it to the reader to verify that $L^{\underline{\cJ}}$ and $R^{\underline{\cJ}}$ satisfy the conditions for $\underline{\cJ}$ to be a $\cF$-twisted bimodule functor.

We now show that the maps $\cH \mapsto \widehat{\cH}$ and $\cJ \mapsto \underline{\cJ}$ are inverse to each other (up to isomorphism), giving bijectivity. Let $\cH : \cM\to \cN$ a $\cF$-twisted bimodule equivalence, then
\[ \underline{(\overline{\cH})}(m) = \overline{\cH}(\mathbf{1},m,\mathbf{1}) = \mathbf{1}\triangleright m \triangleleft \mathbf{1} = m. \]
Thus $\underline{(\overline{\cH})}$ is strictly equal to $\cH$ as a $\cF$-twisted equivalence.

Let $\cJ: \cC_\cF \times \cM \times _\cF\cC \to \cN$ a $\cC$-balanced bimodule functor, then 
\[ R^\cJ_{X\times m \times \mathbf{1}, Y}[L^\cJ_{X,\mathbf{1}\times m \times \mathbf{1}} \triangleleft \id_Y] :  \overline{(\underline{\cJ})}(X,m,Y) = X\triangleright \underline{\cJ}(m)\triangleleft Y = X\triangleright \cJ(\mathbf{1} \times m \times \mathbf{1}) \triangleleft Y \to \cJ(X,m,Y).\]
gives a natural isomorphism $\overline{(\underline{\cJ})} \to \cJ$. We directly check that this natural isomorphism is balanced (only in the first position).
\begin{align*}
&b^\cJ_{X,Z,m,Y}R^\cJ_{X\otimes \cF(Z) \times m \times \mathbf{1}, Y}\highlight{[L^\cJ_{X\otimes \cF(Z), \mathbf{1}\times m \times \mathbf{1}}\triangleleft \id_Y]}\\
&= \highlight{b^\cJ_{X,Z,m,Y}R^\cJ_{X\otimes \cF(Z) \times m \times \mathbf{1}, Y}}[L^\cJ_{X,\cF(Z)\times m \times \mathbf{1}} \triangleleft \id_Y][\id_X \triangleright L^\cJ_{\cF(Z),\mathbf{1}\times m \times \mathbf{1}} \triangleleft \id_Y] \\
&=R^\cJ_{X\times Z\triangleright m\times \mathbf{1},Y}\highlight{[b^\cJ_{X,Z,m,\mathbf{1}} \triangleleft \id_Y][L^\cJ_{X,\cF(Z)\times m \times \mathbf{1}} \triangleleft \id_Y]}[\id_X \triangleright L^\cJ_{\cF(Z),\mathbf{1}\times m \times \mathbf{1}} \triangleleft \id_Y] \\
&=R^\cJ_{X\times Z\triangleright m\times \mathbf{1},Y}   [L^\cJ_{X,\mathbf{1}\times Z\triangleright m \times \mathbf{1}}\triangleleft \id_Y] \highlight{   [\id_X\triangleright b^\cJ_{\mathbf{1},Z,m,\mathbf{1}} \triangleleft \id_Y][\id_X \triangleright L^\cJ_{\cF(Z),\mathbf{1}\times m \times \mathbf{1}} \triangleleft \id_Y]}\\
&= R^\cJ_{X\times Z\triangleright m\times \mathbf{1},Y}   [L^\cJ_{X,\mathbf{1}\times Z\triangleright m \times \mathbf{1}}\triangleleft \id_Y] b^{\overline{(\underline{\cJ})}}_{X,Z,m,Y}.
\end{align*}
Thus $\overline{(\underline{\cJ})}$ is naturally isomorphic to $\cJ$ as a $\cC$-balanced functor. Hence the maps $\cH \mapsto \widehat{\cH}$ and $\cJ \mapsto \underline{\cJ}$ are inverse to each other.

We now apply the equivalence
\[ \operatorname{Fun}^{\text{bal}}( \cC_\cF \times \cM \times _\cF\cC \to \cN  ) \simeq \operatorname{Fun}( \cC_\cF \boxtimes \cM \boxtimes _\cF\cC \to \cN  )   \]
from \eqref{eq:funball} to get a bijection between the set of $\cF$-twisted functors $\cM \to \cN$, and the set of untwisted bimodule functors $\cC_\cF \boxtimes \cM \boxtimes _\cF\cC \to \cN$. We write $\cH \mapsto \widehat{\cH}$ for the map from the set of $\cF$-twisted functors $\cM \to \cN$ to the set of untwisted bimodule functors $\cC_\cF \boxtimes \cM \boxtimes _\cF\cC \to \cN$ induced by the map $\cH \mapsto \overline{\cH}$.

\textbf{Proof of statement~\ref{thm:part2}:}

For this statement we are required to make a specific choice of map $R_{_\cF\cC} : _\cF\cC\boxtimes \cC_\cF \to \cC$. We make the canonical choice, the functor induced by the $\cC$-balanced functor $r_{_\cF\cC}: X\times Y \mapsto \cF^{-1}(X\otimes Y)$, with identity balancing maps. 

Our goal is to construct a bimodule natural isomorphism:
\[ \Omega^\boxtimes_{\cH_1,\cH_2}: \widehat{\cH_1\boxtimes \cH_2}\circ (\Id_{\cC_\cF} \boxtimes \Id_{\cM_1} \boxtimes \operatorname{R}_{_\cF\cC} \boxtimes \Id_{\cM_2} \boxtimes \Id_{_\cF\cC}) \to ( \widehat{\cH_1} \boxtimes \widehat{\cH_2}),\]
which is equivalent under the equivalence
\[ \operatorname{Fun}^{\text{bal}}( \cC_\cF \times \cM_1 \times _\cF\cC \times  \cC_\cF \times \cM_2 \times _\cF\cC \to \cN_1\boxtimes \cN_2) \simeq \operatorname{Fun}( \cC_\cF \boxtimes \cM_1 \boxtimes _\cF\cC \boxtimes  \cC_\cF \boxtimes \cM_2 \boxtimes _\cF\cC \to \cN_1\boxtimes \cN_2)   \]
 to constructing a $\cC$-balanced natural isomorphism:
\[ \omega^\boxtimes_{\cH_1,\cH_2}: \overline{ B_{\cN_1,\cN_2}(\cH_1\times \cH_2) }\circ (\Id_{\cC_\cF}\times \Id_{\cM_1} \times \operatorname{r}_{_\cF\cC} \times \Id_{\cM_2} \times \Id_{_\cF\cC}) \to B_{\cN_1,\cN_2}(\overline{\cH_1}\times\overline{\cH_2}). \]

We define the components of such a natural isomorphism as follows (We write $B$ here instead of $B_{\cN_1,\cN_2}$ for readability.):
\begin{align*}
\omega^\boxtimes_{X\times m_1 \times Y \times V \times m_2 \times W} :=&   {b^{B}_{X\triangleright \cH_1(m_1),Y,V\triangleright \cH_2(m_2) \triangleleft W}}^{-1}  B(\id_X \triangleright \id_{\cH_1(m_1)} \times {L^{\cH_2}_{\cF^{-1}(Y\otimes V), m_2}}^{-1} \triangleleft \id_W) \\
& R^{B}_{X \triangleright \cH_1(m_1) \times \cH_2( \cF^{-1}(Y\otimes V) \triangleright m_2),W}[L^{B}_{X, \cH_1(m_1) \times \cH_2(\cF^{-1}(Y\otimes V) \triangleright m_2)} \triangleleft \id_W] \\
: X\triangleright B( \cH_1(m_1) \times \cH_2(&\cF^{-1}(Y\otimes V) \triangleright m_2))\triangleleft W \to B(X\triangleright \cH_1(m_1) \triangleleft Y \times V \triangleright \cH_2(m_2) \triangleleft W).
\end{align*}
We directly verify that $\omega_{\cH_1,\cH_2}^\boxtimes$ is $\cC$-balanced in the first position, and leave the other four to the reader. 
\begin{align*}
& b^{B(\overline{\cH_1}\times \overline{\cH_2})}_{X,Z,m_1,Y,V,m_2,W} \cdot \omega^\boxtimes_{X\triangleleft Z \times  m_1\times Y\times V \times m_2\times W} \\
&=B(\id_X\triangleright L^{\cH_1}_{Z,m_1} \triangleleft \id_Y \times \id_V \triangleright \id_{\cH_2(m_2)} \triangleleft \id_W) {b^{B}_{( X\otimes \cF(Z) )\triangleright \cH_1(m_1),Y,V\triangleright \cH_2(m_2) \triangleleft W}}^{-1} \\
&B(\id_X \triangleright\id_{\cF(Z)} \triangleright \id_{\cH_1(m_1)} \times {L^{\cH_2}_{\cF^{-1}(Y\otimes V), m_2}}^{-1} \triangleleft \id_W)R^B_{X\triangleright \cF(Z) \triangleright \cH_1(m_1) \times \cH_2( \cF^{-1}(Y\otimes V) \triangleright m_2),W}\\
&\highlight{[L^B_{X\otimes \cF(Z), \cH_1(m_1) \times \cH_2(\cF^{-1}(Y\otimes V) \triangleright m_2)} \triangleleft \id_W]} \\
&= \highlight{B(\id_X\triangleright L^{\cH_1}_{Z,m_1} \triangleleft \id_Y \times \id_V \triangleright \id_{\cH_2{m_2}} \triangleleft \id_W) {b^{B}_{X\triangleright \cF(Z)\triangleright \cH_1(m_1),Y,V\triangleright \cH_2(m_2) \triangleleft W}}^{-1}}\\
&B(\id_X \triangleright\id_{\cF(Z)} \triangleright \id_{\cH_1(m_1)} \times {L^{\cH_2}_{\cF^{-1}(Y\otimes V), m_2}}^{-1} \triangleleft \id_W)R^B_{X\triangleright \cF(Z) \triangleright \cH_1(m_1) \times \cH_2( \cF^{-1}(Y\otimes V) \triangleright m_2),W}\\
&[L^B_{X,\cF(Z)\triangleright \cH_1(m_1) \times \cH_2(\cF^{-1}(Y\otimes V) \triangleright m_2)} \triangleleft \id_W][\id_X \triangleright L^B_{X\otimes \cF(Z), \cH_1(m_1) \times \cH_2(\cF^{-1}(Y\otimes V) \triangleright m_2)} \triangleleft \id_W] \\
&= {b^{B}_{X \triangleright \cH_1(Z \triangleright m_1),Y,V \triangleright \cH_2(m_2) \triangleleft W}}^{-1}      \highlight{     B(\id_X \triangleright L^{\cH_1}_{Z,m_1} \times \id_Y \triangleright \id_V \triangleright \cH_2(m_2)\triangleleft \id_W) } \\ 
&\highlight{B(\id_X \triangleright\id_{\cF(Z)} \triangleright \id_{\cH_1(m_1)} \times {L^{\cH_2}_{\cF^{-1}(Y\otimes V), m_2}}^{-1} \triangleleft \id_W) }    R^B_{X\triangleright \cF(Z) \triangleright \cH_1(m_1) \times \cH_2( \cF^{-1}(Y\otimes V) \triangleright m_2),W} \\
&[L^B_{X,\cF(Z)\triangleright \cH_1(m_1) \times \cH_2(\cF^{-1}(Y\otimes V) \triangleright m_2)} \triangleleft \id_W][\id_X \triangleright L^B_{X\otimes \cF(Z), \cH_1(m_1) \times \cH_2(\cF^{-1}(Y\otimes V) \triangleright m_2)} \triangleleft \id_W] \\
&=  {b^{B}_{X \triangleright \cH_1(Z \triangleright m_1),Y,V \triangleright \cH_2(m_2) \triangleleft W}}^{-1}B(\id_X \triangleright \id_{\cH_1(Z\triangleright m_1)} \times {L^{\cH_2}_{\cF^{-1}(Y\otimes V),m_2}}^{-1} \triangleleft \id_W )  \\
&\highlight{B(\id_X \triangleright L^{\cH_1}_{Z,m_1} \times \id_{\cH_2(\cF^{-1}(Y\otimes V)\triangleright m_2)} \triangleleft \id_W) R^B_{X\triangleright \cF(Z) \triangleright \cH_1(m_1) \times \cH_2( \cF^{-1}(Y\otimes V) \triangleright m_2),W}} \\
&[L^B_{X,\cF(Z)\triangleright \cH_1(m_1) \times \cH_2(\cF^{-1}(Y\otimes V) \triangleright m_2)} \triangleleft \id_W][\id_X \triangleright L^B_{X\otimes \cF(Z), \cH_1(m_1) \times \cH_2(\cF^{-1}(Y\otimes V) \triangleright m_2)} \triangleleft \id_W] \\
&=  {b^{B}_{X \triangleright \cH_1(Z \triangleright m_1),Y,V \triangleright \cH_2(m_2) \triangleleft W}}^{-1}B(\id_X \triangleright \id_{\cH_1(Z\triangleright m_1)} \times {L^{\cH_2}_{\cF^{-1}(Y\otimes V),m_2}}^{-1} \triangleleft \id_W )\\
&R^B_{X\triangleright \cH_1(Z\triangleright m_1) \times \cH_2(\cF^{-1}(Y\otimes V) \triangleright m_2),W}     \highlight{  [B(\id_X \triangleright L^{\cH_1}_{Z,m_1} \times \id_{\cH_2(\cF^{-1}(Y\otimes V )\triangleright m_2)}) \triangleleft \id_W]} \\
&\highlight{ [L^B_{X,\cF(Z)\triangleright \cH_1(m_1) \times \cH_2(\cF^{-1}(Y\otimes V) \triangleright m_2)} \triangleleft \id_W]}[\id_X \triangleright L^B_{X\otimes \cF(Z), \cH_1(m_1) \times \cH_2(\cF^{-1}(Y\otimes V) \triangleright m_2)} \triangleleft \id_W] \\
&=  {b^{B}_{X \triangleright \cH_1(Z \triangleright m_1),Y,V \triangleright \cH_2(m_2) \triangleleft W}}^{-1}B(\id_X \triangleright \id_{\cH_1(Z\triangleright m_1)} \times {L^{\cH_2}_{\cF^{-1}(Y\otimes V),m_2}}^{-1} \triangleleft \id_W ) \\
&R^B_{X\triangleright \cH_1(Z\triangleright m_1) \times \cH_2(\cF^{-1}(Y\otimes V) \triangleright m_2),W}      [ L^B_{X,\cH_1(Z\triangleright m_1)\times \cH_2( \cF^{-1}(Y\otimes V) \triangleright m_2)} \triangleleft \id_W]                         \\
&           [\id_X \triangleright B(L^{\cH_1}_{Z,m_1} \times \id_{\cH_2(\cF^{-1}(Y\otimes V)\triangleright m_2)})\triangleleft \id_W][\id_X \triangleright L^B_{X\otimes \cF(Z), \cH_1(m_1) \times \cH_2(\cF^{-1}(Y\otimes V) \triangleright m_2)} \triangleleft \id_W] \\
&= \omega^\boxtimes_{X\times Z\triangleright m_1\times Y \times V\times m_2\times W}\cdot b^{\overline{B(\cH_1\times \cH_2)}}_{X,Z,m_1,Y,V,m_2,W}.
\end{align*}

\textbf{Proof of statement ~\ref{thm:part3}:}  

For this statement we are required to make a choice of bimodule equivalences
\[     S_{\cF_2,\cF_1}: \cC_{\cF_2}\boxtimes \cC_{\cF_1} \to \cC_{\cF_2\circ \cF_1} \quad \text{ and } \quad S^*_{\cF_1,\cF_2}: _{\cF_1}\cC \boxtimes  _{\cF_2}\cC \to _{\cF_2\circ \cF_1}\cC.\] 

We choose the bimodule equivalences induced from the $\cC$-balanced functors:
\[   s_{\cF_2,\cF_1} : X\boxtimes Y \mapsto X\otimes \cF_2(Y)\quad  \text{ and } \quad s^*_{\cF_1,\cF_2}: X\boxtimes Y \mapsto \cF_2(X)\otimes Y.\]
respectively. The balancing morphisms for these two functors are given by:
\begin{align*}
\id_X \otimes \tau^{\cF_2}_{Z,Y} &: X\otimes \cF_2(Z)\otimes \cF_2(Y) \to X\otimes \cF_2(Z\otimes Y), \text{ and } \\
{\tau^{\cF_2}_{X,Z}}^{-1}\otimes \id_Y& : \cF_2(X\otimes Z) \otimes Y \to \cF_2(X) \otimes \cF_2(Z) \otimes Y.
\end{align*}

Our goal is to construct a bimodule natural isomorphism
\[ \Omega^\circ_{\cH_1,\cH_2}:\widehat{\cH_2\circ \cH_1}\circ(S_{\cF_2,\cF_1} \boxtimes \Id_{\cM} \boxtimes S^*_{\cF_1,\cF_2}) \simeq  \widehat{\cH_2} \circ (\Id_{\cC_{\cF_2}} \boxtimes \widehat{\cH_1} \boxtimes \Id_{_{\cF_2}\cC}), \]
which is equivalent under the equivalence
\[ \operatorname{Fun}^{\text{bal}}( \cC_{\cF_2} \times\cC_{\cF_1} \times \cM \times _{\cF_1}\cC \times _{\cF_2}\cC \to \cN) \simeq \operatorname{Fun}( \cC_{\cF_2} \boxtimes\cC_{\cF_1} \boxtimes \cM \boxtimes _{\cF_1}\cC \boxtimes _{\cF_2}\cC \to \cN)     \] 
to constructing a $\cC$-balanced natural isomorphism:
\[ \omega^\circ_{\cH_1,\cH_2}: \overline{\cH_2\circ \cH_1}\circ(s_{\cF_2,\cF_1} \times \Id_{\cM} \times s^*_{\cF_1,\cF_2}) \to  \overline{\cH_2} \circ (\Id_{\cC_{\cF_2}} \times \overline{\cH_1} \times \Id_{_{\cF_2}\cC}). \]

Our choice for $\omega^\circ_{\cH_1,\cH_2}$  has components (we drop the $\cH_1$, $\cH_2$ subscript for readability):
\begin{align*}
\omega^\circ_{X\times Y \times m \times V \times W}&:=[\id_X \triangleright R^{\cH_2}_{Y\triangleright \cH_1(m),V} \triangleleft \id_W][\id_X \triangleright L^{\cH_2}_{Y,\cH_1(m)}\triangleleft \id_{\cF(V)} \triangleleft \id_W] \\ &:  (X\otimes \cF_2(Y)) \triangleright \cH_2(\cH_1(m)) \triangleleft ( \cF_2(V) \otimes W) \to X \triangleright \cH_2(Y\triangleright \cH_1(m) \triangleleft V) \triangleleft W.
\end{align*}

We only verify that $\omega^\circ_{\cH_1,\cH_2}$ is $\cC$-balanced in the first position and leave the other three positions to the reader.

\begin{align*}
&\omega^\circ_{X,Z\triangleright Y,m,V,W}\cdot b_{X,Z,Y,m,V,W}^{\overline{\cH_2\circ \cH_1}\circ ( s_{\cF_2,\cF_1} \times \id_\cM \times s^*_{\cF_1,\cF_2})}  \\
& = [\id_X \triangleright R^{\cH_2}_{Z\triangleright Y\triangleright \cH_1(m),V} \triangleleft \id_W]\highlight{[\id_X \triangleright L^{\cH_2}_{Z\otimes Y,\cH_1(m)}\triangleleft \id_{\cF_2(V)} \triangleleft \id_W][\id_X \triangleright \tau^{\cF_2}_{Z,Y} \triangleright \id_{\cH_2(\cH_1(m))} \triangleleft \id_{\cF_2(V)} \triangleleft \id_W]} \\
& = \highlight{[\id_X \triangleright R^{\cH_2}_{Z\triangleright Y\triangleright \cH_1(m),V} \triangleleft \id_W][\id_X \triangleright L^{\cH_2}_{Z, Y\triangleright \cH_1(m)} \triangleleft \id_{\cF_2(V)} \triangleleft \id_W]}[\id_X \triangleright \id_{\cF_2{Z}} \triangleright L^{\cH_2}_{Y,\cH_1(m)} \triangleleft  \id_{\cF_2(V)} \triangleleft \id_W] \\
&=  [\id_X \triangleright L^{\cH_2}_{Z, Y\triangleright \cH_1(m) \triangleleft V} \triangleleft \id_W][\id_X \triangleright \id_{\cF_2(Z)} \triangleright R^{\cH_2}_{Y\triangleright \cH_1(m), V} \triangleleft \id_W][\id_X \triangleright \id_{\cF_2{Z}} \triangleright L^{\cH_2}_{Y,\cH_1(m)} \triangleleft  \id_{\cF_2(V)} \triangleleft \id_W] \\
&=    b_{X,Z,Y,m,V,W}^{\overline{\cH_2}\circ (\Id_{\cC_{\cF_2}} \times \overline{\cH_1}\times \Id_{_{\cF_2}\cC}) }\cdot \omega^\circ_{X\triangleleft Z, Y,m,V,W}.
\end{align*}

\textbf{Proof of statement~\ref{thm:part4}:}

Let $\cH_1, \cH_2 : \cM \to \cN$ be $\cF$-twisted bimodule functors. Our goal is to construct a bijection between the set of $\cF$-twisted bimodule natural isomorphisms $\cH_1 \to \cH_2$, and the set of untwisted bimodule natural isomorphisms $\widehat{ \cH_1} \to \widehat{\cH_2}$. We achieve this by constructing a bijection between the set of $\cF$-twisted bimodule natural isomorphisms $\cH_1 \to \cH_2$, and the set of $\cC$-balanced natural isomorphisms $\overline{ \cH_1} \to \overline{\cH_2}$, which then induces the desired bijection through the equivalence
\[ \operatorname{Fun}^{\text{bal}}( \cC_\cF \times \cM \times _\cF\cC \to \cN  ) \simeq \operatorname{Fun}( \cC_\cF \boxtimes \cM \boxtimes _\cF\cC \to \cN  ) .  \]

Given  $\mu: \cH_1 \to \cH_2$ a natural isomorphism of $\cF$-twisted bimodule functors, we need to construct a $\cC$-balanced natural isomorphism $\overline{\mu}: \overline{\cH_1} \to \overline{\cH_2}$. We claim that the natural isomorphism
\begin{equation*}
\overline{\mu}_{X,m,Y} = \id_X \triangleright \mu_m \triangleleft \id_Y,
\end{equation*}
is $\cC$-balanced in all positions. To see this we directly compute (only in the first position, and leave the rest to the reader)
\begin{align*}
b^{\cH_2}_{X,Z,m,Y}\overline{\mu}_{X \triangleleft Z, m , Y} =& \highlight{[\id_X \triangleright L^{\cH_2}_{Z,M} \triangleleft \id_Y][\id_{X\triangleleft Z} \triangleright \mu_m \triangleleft \id_Y)]} \\
=& \id_X \triangleright [L^{\cH_2}_{Z,M}(\id_{\cF(Z)}\triangleright \mu_m)] \triangleleft \id_Y \\
=&  \id_X \triangleright \highlight{[\mu_{Z\triangleright m}L^{\cH_1}_{Z,m}]} \triangleleft \id_Y \\ 
=& \overline{\mu}_{X,Z\triangleright m,Y}b^{\cH_1}_{X,Z,m,Y}. 
\end{align*}
Thus the map $\mu \mapsto \overline{\mu}$ gives a map from the set of $\cF$-twisted bimodule natural isomorphisms $\cH_1 \to \cH_2$, to the set of $\cC$-balanced natural isomorphisms $\overline{ \cH_1} \to \overline{\cH_2}$.

We now give an inverse to the map $\mu \mapsto \overline{\mu}$. Let $\nu : \overline{\cH_1} \to \overline{\cH_2}$ a $\cC$-balanced natural isomorphism. We define a natural transformation $ \underline{\nu}: \cH_1 \to \cH_2$ by
\[ \underline{\nu}_m := \nu_{\mathbf{1} , m , \mathbf{1}}.   \]
It is straightforward to verify that $\underline{\nu}$ is a natural isomorphism of $\cF$-twisted functors. It also follows immediately that $\underline{(\overline{\mu})}_m = \mu_m$ and $\overline{(\underline{\nu})}_{X,m,Y} = \nu_{X,m,Y}$. Thus the map $\mu \mapsto \overline{\mu}$ is bijective, and hence induces a bijective map $\mu \mapsto \widehat{\mu}$ from the set of $\cF$-twisted bimodule natural isomorphisms $\cH_1 \to \cH_2$, to the set of untwisted bimodule natural isomorphisms $\widehat{ \cH_1} \to \widehat{\cH_2}$.

\textbf{Proof of statement~\ref{thm:part5}:}

Via the equivalence
\[ \operatorname{Fun}^{\text{bal}}( \cC_\cF \times \cM_1 \times _\cF\cC \times  \cC_\cF \times \cM_2 \times _\cF\cC \to \cN_1\boxtimes \cN_2) \simeq \operatorname{Fun}( \cC_\cF \boxtimes \cM_1 \boxtimes _\cF\cC \boxtimes  \cC_\cF \boxtimes \cM_2 \boxtimes _\cF\cC \to \cN_1\boxtimes \cN_2)   \]
we have that this statement is equivalent to showing that the $\cC$-balanced natural isomorphisms:
\[ \omega^\boxtimes_{\cH_1',\cH_2'}[\overline{\mu_1 \times \mu_2}\circ (\id_{\Id_{\cC_\cF}} \times \id_{\Id_{\cM_1}}  \times \id_{\operatorname{R_{_\cF\cC}}} \times \id_{\Id_{\cM_2}} \times \id_{\Id_{_\cF\cC}})] \]
and
\[  [ \overline{\mu_1} \times \overline{\mu_2}]\omega^\boxtimes_{\cH_1,\cH_2} \]
are equal. 

Let 
\[ X\times m_1 \times Y\times V \times m_2 \times W \in \cC_\cF \times \cM_1 \times _\cF\cC \times \cC_\cF \times \cM_2 \times _\cF\cC,\]
 then
\begin{align*}
 &\left(\omega^\boxtimes_{\cH_1',\cH_2'}[\overline{\mu_1 \times \mu_2}\circ (\id_{\Id_{\cC_\cF}} \times \id_{\Id_{\cM_1}}  \times \id_{\operatorname{R_{_\cF\cC}}} \times \id_{\Id_{\cM_2}} \times \id_{\Id_{_\cF\cC}})]\right)_{X\times m_1 \times Y\times V \times m_2 \times W} \\
&=  {b^{B}_{X\triangleright \cH_1'(m_1),Y,V\triangleright \cH_2'(m_2) \triangleleft W}}^{-1}  B(\id_X \triangleright \id_{\cH_1'(m_1)} \times {L^{\cH_2'}_{\cF^{-1}(Y\otimes V), m_2}}^{-1} \triangleleft \id_W) \highlight{ R^{B}_{X \triangleright \cH_1'(m_1) \times \cH_2'( \cF^{-1}(Y\otimes V) \triangleright m_2),W}}\\
&\highlight{[L^{B}_{X, \cH_1(m_1) \times \cH_2(\cF^{-1}(Y\otimes V) \triangleright m_2)} \triangleleft \id_W][\id_X \triangleright B( {\mu_1}_{m_1} \times {\mu_2}_{\cF^{-1}(Y\otimes V) \triangleright m_2}) \triangleleft \id_W]}  \\
&= {b^{B}_{X\triangleright \cH_1'(m_1),Y,V\triangleright \cH_2'(m_2) \triangleleft W}}^{-1}	\highlight{B(\id_X \triangleright \id_{\cH_1'(m_1)} \times { L^{\cH_2'}_{\cF^{-1}(Y\otimes V), m_2}}^{-1}  \triangleleft \id_W)B( \id_X \triangleright {\mu_1}_{m_1} \times {\mu_2}_{\cF^{-1}(Y\otimes V)\triangleright m_2} \triangleleft \id_W)} \\
& R^B_{X\triangleright \cH_1(m_1) \times \cH_2(\cF^{-1}(Y\otimes V)\triangleright m_2),W}[L^B_{X,  \cH_1(m_1) \times \cH_2( \cF^{-1}(Y\otimes V) \triangleright m_2)} \triangleleft \id_W] \\
&= \highlight{{b^{B}_{X\triangleright \cH_1'(m_1),Y,V\triangleright \cH_2'(m_2) \triangleleft W}}^{-1} B(\id_W \triangleright {\mu_1}_{m_1} \times \id_Y \triangleright \id_V \triangleright {\mu_2}_{m_2} \triangleleft \id_W ) }     B( \id_X \triangleright \id_{\cH_1(m_1)} \times {L^{\cH_2}_{\cF^{-1}(Y\otimes V), m_2}}^{-1} \triangleleft \id_W) \\
& R^B_{X\triangleright \cH_1(m_1) \times \cH_2(\cF^{-1}(Y\otimes V)\triangleright m_2),W}[L^B_{X,  \cH_1(m_1) \times \cH_2( \cF^{-1}(Y\otimes V) \triangleright m_2)} \triangleleft \id_W] \\
&=B(\id_X \triangleright {\mu_1}_{m_1} \triangleleft \id_Y \times  \id_V \triangleright {\mu_2}_{m_2} \triangleleft \id_W )  {b^{B}_{X\triangleright \cH_1(m_1), Y, V\triangleright \cH_2(m_2) \triangleleft W  }}^{-1}      B( \id_X \triangleright \id_{\cH_1(m_1)} \times {L^{\cH_2}_{\cF^{-1}(Y\otimes V), m_2}}^{-1} \triangleleft \id_W) \\
& R^B_{X\triangleright \cH_1(m_1) \times \cH_2(\cF^{-1}(Y\otimes V)\triangleright m_2),W}[L^B_{X,  \cH_1(m_1) \times \cH_2( \cF^{-1}(Y\otimes V) \triangleright m_2)} \triangleleft \id_W] \\
&=  \left(( \overline{\mu_1} \times \overline{\mu_2})\omega^\boxtimes_{\cH_1,\cH_2}\right)_{X\times m_1 \times Y\times V \times m_2 \times W}.
\end{align*}

\textbf{Proof of statement~\ref{thm:part6}:} 

Via the equivalence
\[ \operatorname{Fun}^{\text{bal}}( \cC_{\cF_2} \times\cC_{\cF_1} \times \cM \times _{\cF_1}\cC \times _{\cF_2}\cC \to \cN) \simeq \operatorname{Fun}( \cC_{\cF_2} \boxtimes\cC_{\cF_1} \boxtimes \cM \boxtimes _{\cF_1}\cC \boxtimes _{\cF_2}\cC \to \cN)     \] 
we have that this statement is equivalent to showing that the $\cC$-balanced natural isomorphisms:
\[ \omega^\circ_{\cH_1',\cH_2'}[ \overline{\mu_2\circ\mu_1} \circ (\id_{s_{\cF_2,\cF_1}\times \id_{\Id_\cM} \times \id_{s^*_{\cF_1,\cF_2}}})] \]
and 
\[ [\overline{\mu_2} \circ ( \id_{\Id_{\cC_{\cF_2}}} \times \overline{\mu_1} \times \id_{\Id_{_{\cF_2}\cC}})]\omega^\circ_{\cH_1,\cH_2} \]
are equal. Let 
\[ X\times Y\times m\times V\times W \in \cC_{\cF_2}\times \cC_{\cF_1} \times \cM \times _{\cF_1}\cC \times _{\cF_2}\cC,\] 
then
\begin{align*}
&\left(\omega^\circ_{\cH_1',\cH_2'}[ \overline{\mu_2\circ\mu_1} \circ (\id_{s_{\cF_2,\cF_1}\times \id_{\Id_\cM} \times \id_{s^*_{\cF_1,\cF_2}}})]\right)_{X\times Y\times m\times V\times W} \\
&=[\id_X \triangleright R^{\cH_2'}_{Y\triangleright \cH_1'(m),V} \triangleleft \id_W]\highlight{[\id_X \triangleright L^{\cH_2'}_{Y,\cH_1'(m)}\triangleleft \id_{\cF(V)} \triangleleft \id_W][\id_X \triangleright \id_{\cF(Y)} \triangleright {\mu_2}_{\cH_1'(m)} \triangleleft \id_{\cF(V)} \triangleleft \id_W]}\\
&[\id_X \triangleright \id_{\cF(Y)} \triangleright \cH_2({\mu_1}_m) \triangleleft \id_{\cF(V)} \triangleleft \id_W] \\
&= \highlight{[\id_X \triangleright R^{\cH_2'}_{Y\triangleright \cH_1'(m),V} \triangleleft \id_W] [\id_X\triangleright {\mu_2}_{Y\triangleright \cH_1'(m)} \triangleleft \id_{\cF(V)} \triangleleft W]}[ \id_X \triangleright L^{\cH_2}_{Y,\cH_1'(m)} \triangleleft \id_{\cF(V)} \triangleleft \id_W         ]  \\
&[\id_X \triangleright \id_{\cF(Y)} \triangleright \cH_2({\mu_1}_m) \triangleleft \id_{\cF(V)} \triangleleft \id_W] \\
&= [\id_X \triangleright {\mu_2}_{Y\triangleright \cH_1'(m)\triangleleft V} \triangleleft \id_{W}]\highlight{[\id_X \triangleright R^{\cH_2}_{Y\triangleright \cH_1'(m),V} \triangleleft \id_W][ \id_X \triangleright L^{\cH_2}_{Y,\cH_1'(m)} \triangleleft \id_{\cF(V)} \triangleleft \id_W         ]} \\
& \highlight{ [\id_X \triangleright \id_{\cF(Y)} \triangleright \cH_2({\mu_1}_m) \triangleleft \id_{\cF(V)} \triangleleft \id_W]}\\
&= [\id_X \triangleright {\mu_2}_{Y\triangleright \cH_1'(m)\triangleleft V} \triangleleft \id_W][\id_X \triangleright \cH_2(\id_Y \triangleright {\mu_1}_m \triangleleft \id_V )\triangleleft \id_W][\id_X \triangleright R^{\cH_2}_{Y\triangleright \cH_1(m),V} \triangleleft \id_W]  \\
&[\id_X \triangleright L^{\cH_2}_{Y,\cH_1(m)}\triangleleft \id_{\cF(V)} \triangleleft \id_W] \\
&=\left({[\overline{\mu_2} \circ ( \id_{\Id_{\cC_{\cF_2}}} \times \overline{\mu_1} \times \id_{\Id_{_{\cF_2\cC}}})]\omega^\circ_{\cH_1,\cH_2}}\right)_{X\times Y\times m\times V\times W}.
\end{align*}

\textbf{Proof of statement~\ref{thm:part7}:}

To show that $\widehat{\mu_1\mu_2}$ and $\widehat{\mu_1}\cdot \widehat{\mu_2}$ are equal we invoke the equivalence
\[ \operatorname{Fun}^{\text{bal}}( \cC_\cF \times \cM \times _\cF\cC \to \cN  ) \simeq \operatorname{Fun}( \cC_\cF \boxtimes \cM \boxtimes _\cF\cC \to \cN  ) \]
and show that the $\cC$-balanced natural isomorphisms $\overline{\mu_1\mu_2}$ and $\overline{\mu_1}\cdot \overline{\mu_2}$ are equal. We compute
\begin{align*}
\overline{\mu_2\mu_1}_{X,m,Y} &= \highlight{\id_X \triangleright {\mu_2}_m {\mu_1}_m \triangleleft \id_Y} \\
&= [\id_X \triangleright {\mu_2}_m \triangleleft \id_Y][\id_X \triangleright {\mu_1}_m \triangleleft \id_Y] \\
&= \overline{\mu_2}_{X,m,Y} \overline{\mu_1}_{X,m,Y}.
\end{align*}
\end{proof}
\bibliography{bibliography} 
\bibliographystyle{plain}

\end{document}